\documentclass[english,11pt]{amsart}
\usepackage{amssymb,amsmath,amsthm,mathtools,amsfonts}

\usepackage{graphicx}
\usepackage[usenames,dvipsnames]{xcolor}
\usepackage{cite}
\usepackage{url}
\usepackage{hyperref}
\usepackage{xcolor}
\usepackage[all]{xy}

\usepackage[version=4]{mhchem}

\hypersetup{
    bookmarks=true,         	
    unicode=false,          		
    pdftoolbar=true,        		
    pdfmenubar=true,        	
    pdffitwindow=false,     		
    pdfstartview={FitH},    		
   pdftitle={Soliton resolution for the modified KdV Equation },    					
   pdfauthor={Gong Chen, Jiaqi Liu},     	
    linktocpage=true,			
    pdfnewwindow=true,      	
    colorlinks=true,       			
    linkcolor=red,          		
    citecolor=PineGreen,    	
    filecolor=magenta,      		
    urlcolor=cyan           		
}

\usepackage{tikz}
\usetikzlibrary{decorations.pathmorphing}
\usetikzlibrary{arrows,decorations.pathmorphing}

\usepackage{float}
\usepackage[section]{placeins}		

\theoremstyle{plain}
\newtheorem{theorem}{Theorem}[section]
\newtheorem{corollary}[theorem]{Corollary}
\newtheorem{lemma}[theorem]{Lemma}
\newtheorem{proposition}[theorem]{Proposition}

\theoremstyle{definition}
\newtheorem{definition}[theorem]{Definition}
\newtheorem{problem}[theorem]{Problem}

\theoremstyle{remark}
\newtheorem{remark}[theorem]{Remark}

\numberwithin{figure}{section}
\numberwithin{equation}{section}

\DeclareMathOperator{\ad}{ad}

\DeclareMathOperator{\Real}{Re}
\DeclareMathOperator{\Imag}{Im}

\DeclareMathOperator{\Res}{Res}
\DeclareMathOperator{\sech}{sech}

\allowdisplaybreaks

\newenvironment{doublecases}
{
	\left\{ 
			\begin{array}{lllll}
}
{			
			\end{array} 
			\right.
}
			
\begin{document}

\title[Soliton resolution for  mKdV]{Soliton resolution for the modified KdV equation}
\author{Gong Chen}
\author{Jiaqi Liu}
\address[Chen]{Department of Mathematics, University of Toronto, Toronto, Ontario M5S 2E4, Canada }
\email{gc@math.toronto.edu}
\address[Liu]{Department of Mathematics, University of Toronto, Toronto, Ontario M5S 2E4, Canada }
\email{jliu@math.utoronto.ca}

\date{\today}

\begin{abstract}
The soliton resolution for  the focusing modified Korteweg-de vries (mKdV) equation is established for initial 
conditions in  some weighted Sobolev spaces. Our approach is based on the nonlinear steepest descent method and its reformulation through $\overline{\partial}$-derivatives. From the view of stationary points, we give precise asymptotic formulas along trajectory $x=\textrm{v}t$ for any fixed $ \textrm{v} $. To extend the asymptotics to solutions with initial data in low regularity spaces, we apply a global approximation via PDE techniques. As byproducts of our long-time asymptotics, we also obtain the asymptotic stability of nonlinear structures involving solitons and breathers.

\medskip

\end{abstract}

\maketitle
\tableofcontents

%
%

\newcommand{\eps}{\varepsilon}
\newcommand{\lam}{\lambda}

\newcommand{\bfN}{\mathbf{N}}
\newcommand{\calbR}{\mathcal{ \breve{R}}}
\newcommand{\rhobar}{\overline{\rho}}
\newcommand{\zetabar}{\overline{\zeta}}

\newcommand{\rarr}{\rightarrow}
\newcommand{\darr}{\downarrow}

\newcommand{\dee}{\partial}
\newcommand{\dbar}{\overline{\partial}}

\newcommand{\dint}{\displaystyle{\int}}

\newcommand{\dotarg}{\, \cdot \, }

%
%

\newcommand{\RHP}{\mathrm{LC}}			
\newcommand{\PC}{\mathrm{PC}}
\newcommand{\w}{w^{(2)}}
%
%

\newcommand{\zbar}{\overline{z}}

\newcommand{\bbC}{\mathbb{C}}
\newcommand{\bbR}{\mathbb{R}}

\newcommand{\calB}{\mathcal{B}}
\newcommand{\calC}{\mathcal{C}}
\newcommand{\calR}{\mathcal{R}}
\newcommand{\calS}{\mathcal{S}}
\newcommand{\calZ}{\mathcal{Z}}

\newcommand{\ba}{\breve{a}}
\newcommand{\bb}{\breve{b}}

\newcommand{\balpha}{\breve{\alpha}}
\newcommand{\brho}{\breve{\rho}}

\newcommand{\tPhi}{{\widetilde{\Phi}}}

\newcommand{\tp}{\text{p}}
\newcommand{\tq}{\text{q}}
\newcommand{\tr}{\text{r}}

\newcommand{\bfe}{\mathbf{e}}
\newcommand{\bfn}{\mathbf{n}}

\newcommand{\tA}{\tilde{A}}
\newcommand{\tB}{\tilde{B}}
\newcommand{\tomega}{\tilde{\omega}}
\newcommand{\tc}{\tilde{c}}

\newcommand{\mhat}{\hat{m}}

\newcommand{\bphi}{\breve{\Phi}}
\newcommand{\bN}{\breve{N}}
\newcommand{\bV}{\breve{V}}
\newcommand{\bR}{\breve{R}}
\newcommand{\bdelta}{\breve{\delta}}
\newcommand{\bzeta}{\breve{\zeta}}
\newcommand{\bbeta}{\breve{\beta}}
\newcommand{\bm}{\breve{m}}
\newcommand{\br}{\breve{r}}
\newcommand{\bnu}{\breve{\nu}}
\newcommand{\bbfN}{\breve{\mathbf{N}}}
\newcommand{\rbar}{\overline{r}}

\newcommand{\One}{\mathbf{1}}

%
%

\newcommand{\bigO}[2][ ]
{
\mathcal{O}_{#1}
\left(
{#2}
\right)
}

\newcommand{\littleO}[1]{{o}\left( {#1} \right)}

\newcommand{\norm}[2]
{
\left\Vert		{#1}	\right\Vert_{#2}
}

%
%

\newcommand{\rowvec}[2]
{
\left(
	\begin{array}{cc}
		{#1}	&	{#2}	
	\end{array}
\right)
}

\newcommand{\uppermat}[1]
{
\left(
	\begin{array}{cc}
	0		&	{#1}	\\
	0		&	0
	\end{array}
\right)
}

\newcommand{\lowermat}[1]
{
\left(
	\begin{array}{cc}
	0		&	0	\\
	{#1}	&	0
	\end{array}
\right)
}

\newcommand{\offdiagmat}[2]
{
\left(
	\begin{array}{cc}
	0		&	{#1}	\\
	{#2}	&	0
	\end{array}
\right)
}

\newcommand{\diagmat}[2]
{
\left(
	\begin{array}{cc}
		{#1}	&	0	\\
		0		&	{#2}
		\end{array}
\right)
}

\newcommand{\Offdiagmat}[2]
{
\left(
	\begin{array}{cc}
		0			&		{#1} 	\\
		\\
		{#2}		&		0
		\end{array}
\right)
}

\newcommand{\twomat}[4]
{
\left(
	\begin{array}{cc}
		{#1}	&	{#2}	\\
		{#3}	&	{#4}
		\end{array}
\right)
}

\newcommand{\unitupper}[1]
{	
	\twomat{1}{#1}{0}{1}
}

\newcommand{\unitlower}[1]
{
	\twomat{1}{0}{#1}{1}
}

\newcommand{\Twomat}[4]
{
\left(
	\begin{array}{cc}
		{#1}	&	{#2}	\\[10pt]
		{#3}	&	{#4}
		\end{array}
\right)
}

%
%
%

\newcommand{\JumpMatrixFactors}[6]
{
	\begin{equation}
	\label{#2}
	{#1} =	\begin{cases}
					{#3} {#4}, 	&	z \in (-\infty,\xi) \\
					\\
					{#5}{#6},	&	z \in (\xi,\infty)
				\end{cases}
	\end{equation}
}


%
%
%

\newcommand{\RMatrix}[9]
{
\begin{equation}
\label{#1}
\begin{aligned}
\left. R_1 \right|_{(\xi,\infty)} 	&= {#2} &	\qquad\qquad		
\left. R_1 \right|_{\Sigma_1}		&= {#3} 
\\[5pt]
\left. R_3 \right|_{(-\infty,\xi)} 	&= {#4} 	&	
\left. R_3 \right|_{\Sigma_2} 	&= {#5}
\\[5pt]
\left. R_4 \right|_{(-\infty,\xi)} 	&= {#6} &	
\left. R_4 \right|_{\Sigma_3} 	&= {#7} 
\\[5pt]
\left. R_6 \right|_{(\xi,\infty)}  	&= {#8} &	
\left. R_6 \right|_{\Sigma_4} 	&= {#9}
\end{aligned}
\end{equation}
}

%
%

%
%
%
%
%
%

\newcommand{\SixMatrix}[6]
{
\begin{figure}
\centering
\caption{#1}
\vskip 15pt
\begin{tikzpicture}
[scale=0.75]
%
%
\draw[thick]	 (-4,0) -- (4,0);
\draw[thick] 	(-4,4) -- (4,-4);
\draw[thick] 	(-4,-4) -- (4,4);
%
%
\draw	[fill]		(0,0)						circle[radius=0.075];
\node[below] at (0,-0.1) 				{$z_0$};
%
%
\node[above] at (3.5,2.5)				{$\Omega_1$};
\node[below]  at (3.5,-2.5)			{$\Omega_6$};
\node[above] at (0,3.25)				{$\Omega_2$};
\node[below] at (0,-3.25)				{$\Omega_5$};
\node[above] at (-3.5,2.5)			{$\Omega_7^+$};
\node[below] at (-3.5,-2.5)			{$\Omega_8^+$};
%
%
\node[above] at (0,1.25)				{$\twomat{1}{0}{0}{1}$};
\node[below] at (0,-1.25)				{$\twomat{1}{0}{0}{1}$};
%
%
\node[right] at (1.20,0.70)			{$#3$};
\node[left]   at (-1.20,0.70)			{$#4$};
\node[left]   at (-1.20,-0.70)			{$#5$};
\node[right] at (1.20,-0.70)			{$#6$};
\end{tikzpicture}
\label{#2}
\end{figure}
}

\newcommand{\sixmatrix}[6]
{
\begin{figure}
\centering
\caption{#1}
\vskip 15pt
\begin{tikzpicture}
[scale=0.75]
%
%
\draw[thick]	 (-4,0) -- (4,0);
\draw[thick] 	(-4,4) -- (4,-4);
\draw[thick] 	(-4,-4) -- (4,4);
%
%
\draw	[fill]		(0,0)						circle[radius=0.075];
\node[below] at (0,-0.1) 				{$-z_0$};
%
%
\node[above] at (3.5,2.5)				{$\Omega_7^-$};
\node[below]  at (3.5,-2.5)			{$\Omega_8^-$};
\node[above] at (0,3.25)				{$\Omega_2$};
\node[below] at (0,-3.25)				{$\Omega_5$};
\node[above] at (-3.5,2.5)			{$\Omega_3$};
\node[below] at (-3.5,-2.5)			{$\Omega_4$};
%
%
\node[above] at (0,1.25)				{$\twomat{1}{0}{0}{1}$};
\node[below] at (0,-1.25)				{$\twomat{1}{0}{0}{1}$};
%
%
\node[right] at (1.20,0.70)			{$#3$};
\node[left]   at (-1.20,0.70)			{$#4$};
\node[left]   at (-1.20,-0.70)			{$#5$};
\node[right] at (1.20,-0.70)			{$#6$};
\end{tikzpicture}
\label{#2}
\end{figure}
}

%
%
%
%

%
%
%
%

\newcommand{\JumpMatrixRightCut}[6]
{
\begin{figure}
\centering
\caption{#1}
\vskip 15pt
\begin{tikzpicture}[scale=0.85]
%
%
\draw [fill] (4,4) circle [radius=0.075];						
\node at (4.0,3.65) {$\xi$};										
%
%
\draw 	[->, thick]  	(4,4) -- (5,5) ;								
\draw		[thick] 		(5,5) -- (6,6) ;
\draw		[->, thick] 	(2,6) -- (3,5) ;								
\draw		[thick]		(3,5) -- (4,4);	
\draw		[->, thick]	(2,2) -- (3,3);								
\draw		[thick]		(3,3) -- (4,4);
\draw		[->,thick]	(4,4) -- (5,3);								
\draw		[thick]  		(5,3) -- (6,2);
%
%
\draw [  thick, blue, decorate, decoration={snake,amplitude=0.5mm}] (4,4)  -- (8,4);				
\node at (1.5,4) {$0 < \arg (\zeta-\xi) < 2\pi$};
%
%
\node at (8.5,8.5)  	{$\Sigma_1$};
\node at (-0.5,8.5) 	{$\Sigma_2$};
\node at (-0.5,-0.5)	{$\Sigma_3$};
\node at (8.5,-0.5) 	{$\Sigma_4$};
%
%
\node at (7,7) {${#3}$};						
\node at (1,7) {${#4}$};						
\node at (1,1) {${#5}$};						
\node at (7,1) {${#6}$};						
\end{tikzpicture}
\label{#2}
\end{figure}
}

%
%

\section{Introduction}

%
%
In this paper, we study the long-time dynamics of the focusing modified
Korteweg-de vries equation (mKdV)
\begin{equation}
\partial_{t}u+\partial_{xxx}u+6u^{2}\partial_{x}u=0,\ \left(x,t\right)\in\mathbb{R}\times\mathbb{R}^{+}.\label{eq:fMKdV}
\end{equation}
There is a vast body of literature regarding the mKdV equation, in
particular with the local and global well-posedness of the Cauchy
problem. For a summary of known results we refer the reader to Linares-Ponce
\cite{LP}. For the focusing mKdV  equation on the line, we mention the works on the local and global
well-posedness by Kato \cite{Kat}, Kenig-Ponce-Vega \cite{KPV},
Colliander-Keel-Staffilani-Takaoka-Tao \cite{CKSTT}, Guo \cite{Guo}
and Kishimoto \cite{Kis}. In particular,  it is proven by see Kenig, Ponce and Vega in \cite{KPV} that the equation is locally well-posed. 
And global well-posedness is proven by  Colliander, Keel, Staffilani, Takaoka and Tao in
\cite{CKSTT}. Finally Guo \cite{Guo} and Kishimoto \cite{Kis} established well-posedness in  $H^{s}\left(\mathbb{R}\right)$for
$s\ge\frac{1}{4}$. These results are complemented by several ill-posedness
results; see Kenig-Ponce-Vega \cite{KPV2}, Christ-Colliander-Tao
\cite{CCT} and references therein. 

\smallskip

Besides global regularity, another fundamental question for dispersive
PDEs concerns the asymptotic behavior for large time. For small data,
the defocusing and the focusing mKdV have similar asymptotics. For
example, using the complete integrability of the mKdV, in the seminal
work of Deift-Zhou \cite{DZ93}, the global existence and asymptotic
behavior can be studied for the defocusing case using inverse scattering
transforms and the nonlinear steepest descent approach to oscillatory
Riemann-Hilbert problems. Our recent work \cite{CL19} extends
these analysis to low regularity data. For small data, one can also
study the asymptotics without using completely integrability.
A proof of global existence and a (partial) derivation of the asymptotic
behavior for small localized solutions, without making use of complete
integrability, was later given by Hayashi and Naumkin \cite{HN1,HN2}
using the method of factorization of operators. Recently, Germain-Pusateri-Rousset
\cite{GPR} use ideas based on the space-time resonance to study
the long-time asymptotics of small data and the stability of solitary
waves. A precise derivation of asymptotics and a proof of asymptotic
completeness, was given by Harrop-Griffiths \cite{HGB} using wave
packets analysis.

\smallskip
 Compared with the defocusing modified Korteweg-de
vries equation
\[
\partial_{t}u+\partial_{xxx}u-6u^{2}\partial_{x}u=0,\ \left(x,t\right)\in\mathbb{R}\times\mathbb{R}^{+}
\]
studied in Deift-Zhou \cite{DZ93} and our earlier work \cite{CL19},
the first striking feature is the existence of solitons and breathers which do not decay in time (up to translations).
This is a remarkable consequence of the focusing interaction between the
nonlinearity and the dispersion.

Equation \eqref{eq:fMKdV} admits a solution of the
following form: for $c>0$
\begin{equation}
u=Q_{c}\left(x-ct\right)\label{eq:soliton1}
\end{equation}
where
\[
Q_{c}\left(x\right):=\sqrt{c}Q\left(\sqrt{c}x\right).
\]
solves
\[
Q_{c}''-cQ_{c}+Q_{c}^{3}=0,\ Q_{c}\in H^{1}\left(\mathbb{R}\right).
\]
and one can write down the solution explicitly
\[
Q\left(x\right):=2\sqrt{2}\partial_{x}\left(\arctan\left(e^{x}\right)\right).
\]
With these solitons, more complicated solutions are present,
such as multi-soliton solutions [ \cite{Hir}
\cite{WaO}  \cite{Schu} ].

In the context of the focusing mKdV, there exist other nonlinear structures
which do not decay in time. These nonlinear structures, of oscillatory
character,  are known as breathers [ \cite{Lamb} \cite{Wa} ].  They are periodic in time but spatially localized (after a suitable
space shift) real-valued functions. They are of the following
form: For $\alpha,\,\beta\in\mathbb{R}\backslash\left\{ 0\right\} $,

\begin{align}
B_{\alpha,\beta}\left(x,t\right) & =2\sqrt{2}\partial_{x}\left[\arctan\left(\frac{\beta\sin\left(\alpha\left(x+\delta t\right)\right)}{\alpha\cosh\left(\beta\left(x+\gamma t\right)\right)}\right)\right]\label{eq:breather}
\end{align}
where
\[
\delta:=\alpha^{2}-3\beta^{2},\ \gamma:=3\alpha^{2}-\beta^{2}.
\]
Notice that $-\gamma$ plays the role of velocity, which
can be positive or negative. Therefore, compared with
a soliton which only moves to the right, a breather can travel
in both directions. (We will use slightly different
notations later on to be consistent with the inverse scattering
literature.)

\smallskip

If we assume there are no breathers nor solitons, the pure radiation
will behave similarly to the defocusing mKdV. In \cite{DZ93} and \cite{CL19} it has been shown that the defocusing mKdV has asymptotic behaviors in different space-time regions. These include the soliton region,  the self-similar region and the oscillatory region. 

\smallskip

From our brief discussion above, one should realize that general solution to the focusing mKdV will consist of solitons moving to the right, breathers traveling
to both directions and a radiation term. Our goal in this paper is to give
detailed asymptotic analysis  for the focusing mKdV with generic data.
To achieve this, we need to understand the interaction among solitons,
breathers and radiation in different regions precisely. In the generic
setting, finitely many breathers and solitons can appear and they
interact with the radiation. One might expect that a consequence of
the integrability, these nonlinear modes interact elastically but the way they influence the radiation are remarkably
different. To illustrate the complicated behaivor of the solution to the mKdV, we compare the dynamics here with the cubic NLS and the KdV equation.
\begin{enumerate}
	\item[(i)] The KdV  equation has solitons but no breathers.  Like  the KdV solitons,  mKdV solitons travel in the opposite direction of the radiation. So one might expect the interactions between them are weak. Interactions among solitons cause the shift of centers of solitons and the soliton influence on the radiation can be seen from matrix conjugation.  Meanwhile, mKdV breathers can travel in both directions. For those traveling in the same direction with solitons, again, the results of interactions
are similar to that of solitons. More importantly, there are breathers traveling in the same direction with radiation.  As for the behavior of the radiation, both the KdV and the mKdV have soliton region, Painlev\'e region and the oscillatory region but the KdV has one extra part, the collisionless shock region.

\item[(ii)] For those breathers traveling in the same direction with the radiation,  the interactions are strong. They are always coupled with the radiations like the NLS.
If the stationary phase point we choose is close to
the velocity of some breather, the model Riemann-Hilbert problem is
significantly different from the defocusing problem. In particular, there
are be eigenvalues of the direct scattering transform located on
the critical curve with respect to the stationary point. This is the place where the interactions among breathers and radiation is seen from matrix conjugation. More explicitly, we are conjugating matrices obtained from solving a one-breather Riemann-Hilbert problem with matrices resulting from solving a parabolic cylinder model problem. 
\end{enumerate}

\smallskip

Our long-time asymptotics will also provide a proof for
soliton resolution conjecture for the mKdV with generic data. This
conjecture asserts, roughly speaking, that any reasonable solution
 eventually resolves into a superposition of a radiation
component plus a finite number of \textquotedblleft nonlinear bound
states\textquotedblright{} or \textquotedblleft solitons\textquotedblright .
Without using integrability, in Duyckaerts-Kenig-Merle \cite{DKM} and Duyckaerts-Jia-Kenig-Merle
\cite{DJKM} establish this conjecture for the energy critical wave
equation in high dimensions (along a sequence of time for the non-radial
case). For integrable systems, this resolution phenomenon
is studied by Borghese, Jenkins and McLaughlin in \cite{BJM16} for the cubic NLS
and by Jenkins, Liu, Perry and Sulem in \cite{JLPS18} for the  the  derivative NLS and more recently by Saalmann \cite{SA18} for the massive Thirring model.  For the mKdV equation  we mention that
in \cite{Schu}, the author only allows solitons to appear and the
analysis is also restricted to the soliton region for smooth initial condition.
Our analysis include \emph{all kinds} of solitons and breathers and establish
the long-time asymptotics on the \emph{full line}. We also lower the regularity
condition to be \emph{almost optimal}.  We will show that for any generic data in the sense of Definition \ref{genericity}, the solution $u$ to the focusing mKdV \eqref{eq:fMKdV}, can be written as  a superposition of solitons, breathers and a radiation term:
\[
u(x,t)=\sum_{\ell=1}^{N_{2}}u_{\ell}^{\left(br\right)}\left(x,t\right)+\sum_{\ell=1}^{N_{1}}u_{\ell}^{\left(so\right)}\left(x,t\right)+R\left(x,t\right).
\]
For the detailed description of the formula above,  we refer the reader to Theorem \ref{main1}.

\smallskip

As a byproduct of our analysis, we obtain the asymptotic stability
of some nonlinear structures. More precisely, we obtain the \emph{full asymptotic
	stability} of soliton, muti-soliton, breather, multi-breathers and
the combination of them.

\smallskip

Without trying to be exhaustive, we discuss the historical
progress of the stability analysis. Indeed, $H^{1}$-stability of
mKdV solitons and multi-solitons have been considered
e.g. in Bona-Souganidis-Strauss \cite{BSS}, Pego-Weinstein \cite{PW},
Martel-Merle-Tsai \cite{MMT} and Martel-Merle \cite{MM1,MM2}. For
the stability of breathers, see Alejo-Mu\~noz \cite{AM1,AM2}. To understand
the (asymptotic) stability of soliton or breathers, for those traveling
in different direction with radiation, one can use the energy method
with the Lyapunov functional as in, for example Martel-Merle \cite{MM1,MM2}
and Alejo-Mu\~noz \cite{AM1,AM2} after restricting to the soliton region.
Understanding the radiation requires more refined analysis,
see Germain-Pusateri-Rousset \cite{GPR} and Mizumachi \cite{Miz1}.
In particular, in \cite{GPR}, for the perturbation
of the mKdV soliton, the authors give detailed descriptions of the
radiation in terms of Painlev\'e function and modified scattering. Here,
we illustrate explicitly the influence of solitons/breathers on the
radiation. More importantly, in the context of mKdV, there are breathers
traveling alongside the radiation to the \emph{left}. As we point out above, the interaction
here behaves like the interaction between solitons and radiation in
NLS. To understand the asymptotic stability of them, one can always attempt
to linearize the equation near breathers. But the spectral analysis
here is much more involved compared with the NLS equation since breathers
oscillates periodically in time. Moreover, even for the integrable
cubic NLS, to the best of our knowledge, there is no PDE proof the
asymptotic stability of the soliton without invoking the inverse scattering
transform.

\smallskip

To study the long-time asymptotics of integrable system, in the pioneering work of Deift-Zhou \cite{DZ93}, a key step in the nonlinear steepest descent method consists of  deforming the contour
associated to  the RHP in such a way that  the phase function  with oscillatory 
dependence on parameters become exponential decay.
In general the entries of the jump matrix are not analytic, so direct analytic extension off the real axis is not possible. Instead
they must be approximated by rational functions and this results in some error term in the recovered solution. Therefore, in the context of nonlinear steepest descent, most works are carried out
under the assumptions that the initial data belong to the Schwartz
space.

\smallskip

In \cite{Zhou98},  Xin Zhou developed a rigorous analysis of the direct and inverse scattering transform of the AKNS system
for a class of initial conditions $u_0(x)=u(x,t=0)$ belonging to the space  $H^{i,j}(\bbR)$.
Here,  $H^{i,j}(\bbR)$  denotes  
the completion of $C_0^\infty(\bbR)$ in the norm
$$
\norm{u}{H^{i,j}(\bbR)}
= \left( \norm{(1+|x|^j)u}{2}^2 + \norm{u^{(i)}}{2}^2 \right)^{1/2}. 
$$
Recently, much effort has been devoted to relax the regularities
of the initial data. In particular, among the most celebrated results
concerning nonlinear Schr\"odinger equations, we point out the work
of Deift-Zhou \cite{DZ03} where they provide the asymptotics for
the NLS in the weighted space $H^{1,1}$. Dieng and McLaughlin in \cite{DM08} (see also an extended version \cite{DMM18}) developed a variant of Deift-Zhou method. In their approach 
rational approximation of the reflection coefficient is replaced by some 
non-analytic extension of the jump matrices off the real axis, which  leads to a $\bar{\partial}$-problem to 
be solved in some regions of the complex plane. The
new  $\bar{\partial}$-problem can be reduced to an integral equation and is solvable through
Neumann series.   
These ideas were originally implemented by Miller and McLaughlin \cite{MM08} to the 
study the  asymptotic stability of orthogonal polynomials. This method has shown its robustness in its application to other integrable models. Notably, for focusing NLS and derivative NLS, they were 
successfully applied to address the soliton resolution in \cite{BJM16} and \cite{JLPS18} respectively. In this paper, we incorporate this approach into the framework of \cite{DZ93}  to calculate the long time behavior of the focusing mKdV equation in weighted Sobolev spaces. 

\smallskip

Also in Deift-Zhou \cite{DZ03}, they apply an approximation argument to extend the long-time asymptotics of the cubic NLS to the weighted space $L^{2,1}$.  This topology is more or
less optimal from the views of PDE and inverse scattering transformations.
The global $L^{2}$ existence of the cubic NLS can be carried out
by the $L_{t}^{4}L_{x}^{\infty}$ Strichartz estimate and the conservation
of the $L^{2}$ norm. But in order to obtain the precise asymptotics,
one needs to \textquotedblleft pay the price of weights\textquotedblright{}, i.e. working with the weighted space $L^{2,1}$. Recently, in our earlier work Chen-Liu \cite{CL19}, we establish the long-time asymptotics for the defocusing mKdV in $H^{1/4,s},\,s>1/2$ using a global approximation argument based on contractions in the spirit of Kenig-Ponce-Vega \cite{KPV}. In Deift-Zhou \cite{DZ03}, due to the $L_{t}^{4}L_{x}^{\infty}$
Strichartz estimates for the linear Schr\"odinger equation and the
conservation of the $L^{2}$ norm, the authors can globally approximate
the solution to the nonlinear Schr\"odinger equation with data in
$L^{2,1}$ using the Beals-Coifman representation of solutions
directly. Unlike the Schr\"odinger equation, the smoothing estimates and
Strichartz estimates for the Airy equation and the mKdV are much more
involved. For example,  one needs $L_{x}^{4}L_{t}^{\infty}$ which behaves
like a maximal operator. To directly work on the Beals-Coifman
solution to the mKdV to establish the smoothing
estimates and Strichartz estimates, one needs estimates for pesudo-differential
operators with very rough symbols. To avoid these technicalities, we first identify the Beals-Coifman solution with the solution
given by the Duhamel formula  which we call a strong solution for smooth data. Since the strong solutions by construction enjoy Strichartz estimates and
smoothing estimates, by our identification, the Beals-Coifman solutions
also satisfy these estimates. Then we combine Strichartz estimates
and smoothing estimates with the $L^2$ Sobolev bijectivity result by Zhou \cite{Zhou98} to pass limits of Beals-Coifman solutions
to obtain the asymptotics for rougher initial data in  $H^{\frac s,1}\left(\mathbb{R}\right)$ with $s\geq 1/4$. In contrast to our earlier work \cite{CL19}, in this paper, we use the recent work on low regularity conservation law due to Kilip-Visan-Zhang \cite{KiViZh} and Koch-Tataru \cite{KoTa}  to perform the approximation argument for $H^{\frac s,1}\left(\mathbb{R}\right)$ with $s\geq 1/4$ in the unified manner.
To deal with the focusing problem here, we need some refined analysis on the discrete scattering data since the Beals-Coifman representation is more complicated. Then again, via the approximation argument adapted to the focusing problem, we extend the soliton resolution to generic data in $H^{1/4,s},\,s>1/2$.

\smallskip

Finally, we would like point out that similar to Deift-Zhou \cite{DZ93}, our
method is general and algorithmic and does not require an a priori
ansatz for the form of the solution of the asymptotic problem. We
only assume the number of zeros of $a\left(z\right)$ and $\breve{a}\left(z\right)$
are finite, see Section \ref{subsec:eigenvales}  for the definition. This condition is
generic which means that the initial data satisfying this condition
is an open dense set in the space of the initial data. For the KdV
problem, if certain norms of the initial data is bounded, then automatically,
this spectral condition holds, see Deift-Trubowitz \cite{DT}. If
the reflection coefficient is zero, these finite number of zeros will correspond to a
pure muli-soliton solution. When the radiation appears, with the interaction
of reflection coefficients, the resulting phenomenon is more delicate
and complicated. A-priori, just knowing there are finitely many number
of zeros, it is not clear at all that under the influence of the
radiation, the initial data will evolve into a sequence of solitons. To establish
the soliton resolution, we go through reductions step by step via
$\overline{\partial}$-derivatives analysis and nonlinear steepest
descent  to reduce our Riemann-Hilbert
problems (RHPs) to some solvable models. We make sure that only controllable error terms are introduced through these reduction. It is from these exactly solvable model problems that we are going to illustrate the interaction between solitary waves and radiation and the leading asymptotics of the solution.

We begin with some notations:

\subsection{Notations}
Let $\sigma_3$ be the third Pauli matrix:
$$\sigma_3=\twomat{1}{0}{0}{-1}$$
and define the matrix operation 
$$e^{\ad\sigma_3}A=\twomat{a}{e^{2} b}{e^{-2}c}{d}.$$
Given any contour $\Sigma$, $C^\pm$ is the Cauchy projection:
\begin{equation}
(C^\pm f)(z)= \lim_{z\to \Sigma_\pm}\dfrac{1}{2\pi i} \int_{\Sigma} \dfrac{f(s)}{s-z}ds.
\end{equation}
Here $+(-)$ denotes taking limit from the positive (negative) side of the oriented contour.\\
\smallskip
We define Fourier transform as 
	\begin{equation}
	\hat{h}\left(\xi\right)=\mathcal{F}\left[h\right]\left(\xi\right)=\frac{1}{2\pi}\int_\bbR e^{-ix\xi}h\left(x\right)\,dx.\label{eq:FT}
	\end{equation}
	Using the Fourier transform, one can define the fractional weighted
	Sobolev spaces:
	\begin{equation}
	H^{k,s}\left(\mathbb{R}\right):=\left\{ h:\,\left\langle 1+\left|\xi\right|^{2}\right\rangle ^{\frac{k}{2}}\hat{h}\left(\xi\right)\in L^{2}\left(\mathbb{R}\right),\:\left\langle 1+x^{2}\right\rangle ^{\frac{s}{2}}h\in L^{2}\left(\mathbb{R}\right)\right\} .\label{eq:weight}
	\end{equation}
	
As usual, $"A:=B"$
or $"B=:A"$
is the definition of $A$ by means of the expression $B$. We use
the notation $\langle x\rangle=\left(1+|x|^{2}\right)^{\frac{1}{2}}$.
For positive quantities $a$ and $b$, we write $a\lesssim b$ for
$a\leq Cb$ where $C$ is some prescribed constant. Also $a\simeq b$
for $a\lesssim b$ and $b\lesssim a$. Throughout, we use $u_{t}:=\frac{\partial}{\partial_{t}}u$,
$u_{x}:=\frac{\partial}{\partial x}u$.

\subsection{The Riemann--Hilbert problem and inverse scattering} 

To describe our approach, we recall that \eqref{eq:fMKdV} generates an isospectral flow for the problem
\begin{equation}
\label{L}
\frac{d}{dx} \Psi = -iz \sigma_3 \Psi + U(x) \Psi
\end{equation}
where
$$ U(x) = \offdiagmat{iu(x)}{i{u(x)}}.$$
This is a standard AKNS system\cite{AKNS}. If $u \in L^1(\bbR) \cap L^2(\bbR)$, equation \eqref{L} admits bounded 
solutions for $z \in \mathbb{R}$.   There exist unique solutions $\Psi^\pm$ of \eqref{L} obeying the the following space asymptotic conditions
$$\lim_{x \rarr \pm \infty} \Psi^\pm(x,z) e^{-ix z \sigma_3} = \diagmat{1}{1},$$
and there is a matrix $T(z)$, the transition matrix, with 
\begin{equation}
\label{Psi-T}
\Psi^+(x,z)=\Psi^-(x,z) T(z).
\end{equation}
The matrix $T(z)$ takes the form
\begin{equation} \label{matrixT}
 T(z) = \twomat{a(z)}{\bb(z)}{b(z)}{\ba(z)}
 \end{equation}
and  the determinant relation gives
$$ a(z)\ba(z) - b(z)\bb(z) = 1. $$
 By uniqueness we have 
\begin{equation}
\label{symmetry-1}
\psi^\pm_{11}(z)=\overline{\psi^\pm_{22}(\overline{z})}, \quad \psi^\pm_{12}(z)=-\overline{\psi^\pm_{21}(\overline{z})},
\end{equation}
\begin{equation}
\label{symmetry-2}
\psi^\pm_{11}(z)={\psi^\pm_{22}(-{z})}, \quad \psi^\pm_{12}(z)=\psi^\pm_{21}(-z).
\end{equation}
This leads to the symmetry relation of  the entries of $T$:
\begin{align} \label{symmetry}
\ba(z)=\overline{a( \zbar )}, \quad \bb(z) = -\overline{ b(z)}. 
\end{align}
On $\mathbb{R}$, the determinant of $T(z)$ is given by
$$|a(z)|^2+ |b(z)|^2=1.$$
Making the change of variable 
$$\Psi=m e^{ixz\sigma_3}$$
the system \eqref{L} then becomes 
\begin{equation}
\label{AKNS-m}
m_x=-iz\ad \sigma_3~ m+Um.
\end{equation}
The standard AKNS method starts with the following two Volterra integral equations for real $z$:
\begin{equation}
\label{IE-m-pm}
m^{(\pm)}(x, z)=I + \int_{\pm\infty}^x e^{i(y-x)z \ad\sigma_3} U(y)m^{(\pm)}(y,z)dy.
\end{equation}

By the standard inverse scattering theory, we formulate the reflection coefficient:
\begin{equation}
\label{reflection}
r(z)=-b(z)/\ba(z), \quad z\in\bbR.
\end{equation}
Also from the symmetry conditions \eqref{symmetry-1}-\eqref{symmetry-2} we deduce that
\begin{equation}
\label{minus}
r(-z)=-\overline{r( z )}.
\end{equation}
\subsubsection{Eigenvalues}\label{subsec:eigenvales}
It is important to notice that $\ba(z)$ and $a(z)$ has analytic continuation into the $\bbC^+$ and $\bbC^-$ half planes respectively.
From \eqref{Psi-T} we deduce that 
\begin{equation}
\label{ba-det}
\breve{a}(z)=\det\twomat{\psi^-_{11}(x, z )}{\psi^+_{12}(x, z)}{\psi^-_{21}(x, z )}{\psi^+_{22}(x,z)}.
\end{equation}
\begin{equation}
\label{a-det}
{a}(z)=\det\twomat{\psi^+_{11}(x,z)}{\psi^-_{12}(x,z)}{\psi^+_{21}(x,z)}{\psi^-_{22}(x,z)}.
\end{equation}
From \eqref{symmetry-2}-\eqref{symmetry} we read off directly that if $\ba(z_i)=0$ for some $z_i\in \bbC^+$, then $\overline{\ba(-  \overline{ z_i } )}=0$ by symmetry. Thus if $\ba(z_i)=0$, then either
\begin{itemize}
\item[(i)] $z_i$ is purely imaginary;\\
or
\item[(ii)] $-\overline{z_i}$ is also a zero $\ba$.
\end{itemize}
When $r\equiv 0$, Case (i) above corresponds to solitons while Case (ii) introduces breathers.

\begin{remark}
\label{remark-generic}
It is proven in \cite{BC84} that there is  an  open and dense subset $U_0 \subset L^1(\bbR)$ such that if $u \in U_0$ , then the zeros of $\ba$ ($a$) are finite and simple and off the real axis. We restrict the initial data to such set in this paper.
\end{remark}
\begin{figure}[h!]
\caption{Zeros of $\ba$ and $a$ }
\vskip 0.4cm
\begin{center}
\begin{tabular}{ccc}


\setlength{\unitlength}{5.0cm}
\begin{picture}(1,1)

\put(0.5,0.5){\vector(1,0){0.25}}
\put(0.75,0.5){\line(1,0){0.25}}

\put(0.5,1){\line(0,-1){0.25}}
\put(0.5,0.5){\line(0,1){0.25}}

\put(0.5,0.5){\line(-1,0){0.25}}
\put(0.25,0.5){\line(-1,0){0.25}}

\put(0.5,0.25){\line(0,1){0.25}}
\put(0.5,0){\line(0,1){0.25}}

\put(0.5,0.5){\circle{0.025}}


\put(0.9,0.55){$\bbC^+$}
\put(0.9, 0.4){$\bbC^-$}
\put(0.55,0.9){$i\bbR$}


\put(0.472,0.65){\color{red}$\times$}
\put(0.472,0.35){\color{blue}$\times$}
\put(0.472,0.4){\color{blue}$\times$}
\put(0.472,0.6){\color{red}$\times$}


\put(0.7,0.8){{\color{red}\circle*{0.025}}}
\put(0.7,0.2){{\color{blue}\circle*{0.025}}}
\put(0.3,0.8){{\color{red}\circle*{0.025}}}
\put(0.3,0.2){{\color{blue}\circle*{0.025}}}

\put(0.6,0.7){{\color{red}\circle*{0.025}}}
\put(0.6,0.3){{\color{blue}\circle*{0.025}}}
\put(0.4,0.7){{\color{red}\circle*{0.025}}}
\put(0.4,0.3){{\color{blue}\circle*{0.025}}}

\end{picture}

 \\[0.2cm]
\end{tabular}

\vskip 0.2cm

\begin{tabular}{ccc}
Origin ({\color{black} $\circ$}) &
zeros of $\ba$ ({\color{red} $\times$} {\color{red} $\bullet$})	&	
zeros of $a$  ({\color{blue} $\times$} {\color{blue} $\bullet$}) 
\end{tabular}
\end{center}

\label{fig:spectra}
\end{figure}
Suppose that $\ba(z_i)=0$ for some $z_i\in\bbC^+$,  $i=1,2,..., N$, then we have the linear dependence of the columns :
\begin{align}
\label{b_i}
 \begin{bmatrix}
           \psi^-_{11}(x,z_i) \\
           \psi^-_{21}(x,z_i) \\
        \end{bmatrix}=b_i\begin{bmatrix}
           \psi^+_{12}(x, z_i) \\
           \psi^+_{22}(x,z_i) \\
        \end{bmatrix}
    \end{align}
  \begin{align} 
\label{b_0}
   \begin{bmatrix}
           m^-_{11}(x, z_i) \\
           m^-_{21}(x, z_i) \\
        \end{bmatrix}=b_i\begin{bmatrix}
           m^+_{12}(x, z_i) \\
           m^+_{22}(x, z_i) \\
        \end{bmatrix}e^{2ix{z_i}}.
  \end{align}
\begin{remark}
As the zeros of $\ba$ are of order one, $\ba'(z_i)\neq 0$.
\end{remark}  

\subsubsection{Inverse Problem}
\label{subsec:inverse}
In this subsection we construct the Beals-Coifman solutions needed for the RHP.  We need to find certain piecewise analytic matrix functions. An obvious choice is 
\begin{equation}
 \begin{cases}
(m^{(-)}_1, m^{(+)}_2), \qquad \Imag z>0\\
(m^{(+)}_1, m^{(-)}_2), \qquad \Imag z<0.
\end{cases}
\end{equation}
We want the solution to the RHP normalized as $x\rightarrow +\infty$, so we set
\begin{equation}
\label{BC}
M(x,z)= \begin{cases}
(m^{(-)}_1, m^{(+)}_2)\twomat{\ba^{-1}}{0}{0}{1}, \qquad \Imag z>0\\
(m^{(+)}_1, m^{(-)}_2)\twomat{1}{0}{0}{a^{-1}}, \qquad \Imag z<0.
\end{cases}
\end{equation}
We assume $a(z)\neq 0$ for all $z\in\mathbb{R}$ and recall
\begin{equation}
{r}(z)=-\dfrac{b(z)}{\ba(z)}
\end{equation}
and by symmetry
$$\dfrac{\bb(z)}{a(z)}=\overline{r(z)}.$$
Using the asymptotic condition of $m^\pm$ and  \eqref{matrixT}, we conclude that for $z\in \bbR$
\begin{subequations}
\begin{equation}
\label{M+}
\lim_{x\to +\infty}(m^{(-)}_1, m^{(+)}_2)\twomat{\ba^{-1}}{0}{0}{1}= \twomat{1}{0}{-e^{2ixz}  \dfrac{b(z)}{\ba(z)} }{1},
\end{equation}
\begin{equation}
\label{M-}
\lim_{x\to +\infty}(m^{(+)}_1, m^{(-)}_2)\twomat{1}{0}{0}{a^{-1}}=\twomat{1}{-e^{-2ixz}\dfrac{\bb(z)}{a(z)} }{0}{1}.
\end{equation}
\end{subequations}
Setting $M_{\pm}(x,z)=\lim_{\epsilon\to 0}M(x, z\pm i\epsilon)$, then $M_{\pm}$ satisfy the following jump condition on $\bbR$:
$$M_+(x, z)=M_-(x,z)\twomat{1+|r(z)|^2}{e^{-2ixz} \overline{r(z)} }{e^{2ixz} r(z) }{1}.$$

We now calculate the residue at the pole $z_i$:
\begin{align}
\label{residue1}
\textrm{Res}_{z =z_i}M_{+,1}(x, z)&=\frac{1}{\breve{a}'(z_i)}\Twomat{m^-_{11}(x,z_i)}{0}{m^-_{21}(x,z_i)}{0}\\
\nonumber
                                    &=\frac{e^{2ix{z_i}}b_i}{\breve{a}'(z_i)}\Twomat{m^+_{12}(x,z_i)}{0}{m^+_{22}(x,z_i)}{0}.
\end{align}

Similarly we can calculate the residues at the pole $\overline{z_i}$:
\begin{align}
\label{residue2}
\textrm{Res}_{z =\overline{z_i}}\textrm{M}_{-,2}(x,z)
=-\frac{e^{-2ix\overline{z_i}}\overline{b_i}}{{a}'(\overline{z_i})}\Twomat{0}{m^+_{11}(x,\overline{z_i} ) }{0}{m^+_{21}(x,\overline{z_i})}.
\end{align}
If $z_i$ is not purely imaginary, we also have 
\begin{align}
\label{residue3}
\textrm{Res}_{z =-\overline{z_i} }M_{+,1}(x, z)&=\frac{1}{\breve{a}'(-\overline{z_i} )}\Twomat{m^-_{11}(x,-\overline{z_i})}{0}{m^-_{21}(x,-\overline{z_i})}{0}\\
\nonumber
                                    &=-\frac{e^{-2ix{\overline{z_i}}}  \overline{b_i}   }{\breve{a}'(-\overline{z_i} )}\Twomat{m^+_{12}(x,-\overline{z_i})}{0}{m^+_{22}(x,-\overline{z_i})}{0}
\end{align}
and
\begin{align}
\label{residue4}
\textrm{Res}_{z =- {z_i}}\textrm{M}_{-,2}(x,z)
= \frac{e^{2ix {z_i}} {b_i}}{{a}'(-{z_i})}\Twomat{0}{m^+_{11}(x, -{z_i} ) }{0}{m^+_{21}(x, -{z_i})}.
\end{align}
Using symmetry reduction we have that $\breve{a}'(z_i)=\overline{a'(\overline{z_i})}$
so we can define norming constant
$$c_i=\frac{b_i}{\breve{a}'(z_i)}.$$
The following result is proven in \cite{Zhou98}:
\begin{proposition}
\label{D1}
If $u_0 \in H^{2,1}(\bbR)$,  then $r(z)\in H^{1,2}(\bbR)$.
\end{proposition}
Thus we arrive at the following set of  scattering data
\begin{equation}
\label{scattering}
\mathcal{S}=\lbrace r(z),\lbrace z_k, c_k \rbrace_{k=1}^{N_1}, \lbrace z_j, c_j \rbrace_{j=1}^{N_2} \rbrace \subset H^{1,2}(\bbR) \oplus \mathbb{C}^{2 N_1} \oplus \mathbb{C}^{ 2 N_2} .
\end{equation}
Here $ z_k=i\zeta_k$ for $\zeta_k>0$  while $z_j=\xi_j+i\eta_j$ with $\xi_j> 0$ and $\eta_j>0$.

It is well-known that $r(z)$, $c_j$ and $c_k$ have linear time evolution:
$$r(z, t)=e^{8itz^3}r(z), \quad c_j(t)=e^{8itz_j^3} c_j, \quad c_k(t)=e^{8itz_k^3} c_k .$$
In Appendix \ref{discrete}, we will show that the maps $u_0 \mapsto \mathcal{S}$ is Lipschitz continuous from $H^{2,1}(\bbR)$ 
into a subset of 
 $H^{1,2}(\bbR) \oplus \mathbb{C}^{2 N_1} \oplus \mathbb{C}^{ 2 N_2} $.   The long time asymptotics of mKdV is obtained through a sequence of transformations of the following RHP:
\begin{problem}
\label{RHP-1}
For fixed $x\in\bbR$ and  $r(z)\in H^{1,2}(\bbR)$, find a meromorphic matrix $M(x,t; z)$ satisfying the following conditions:
\begin{enumerate}
\item[(i)] (Normalization) $M(x,t; z)\to I+\mathcal{O}(z^{-1})$ as $z\to \infty$.
\item[(ii)] (Jump relation) For each $z \in \bbR$, $M(x,t; z)$ has continuous non-tangential boundary value $M_\pm(x, t; z)$ as $z$ approaches $\bbR$ from $\bbC^\pm$ and the following jump relation holds
\begin{align}
\label{jump}
M_+(x,t; z) &=M_-(x,t; z)e^{-i\theta(x,t; z)\ad\sigma_3}v(z)\\
                   &=M_-(x,t; z)v_\theta(z)
\end{align}
where
$$v(z)=\twomat{1+|r(z)|^2}{\overline{r(z)}} {r(z) }{1}$$
and 
\begin{equation}
\label{theta}
\theta(x,t; z)=4t(z^3-3z_0^2 z)=4tz^3+xz
\end{equation}
where 
\begin{equation}
\label{stationary}
\pm z_0=\pm\sqrt{\dfrac{-x}{12t}}
\end{equation}
are the two stationary points.
\item[(iii)] (Residue condition) For $k=1,2..., N_1$, $M(x, t; z)$ has simple poles at each $z_k, \overline{z_k}$ with
\begin{equation}
\label{res-1}
\Res_{z_i}M=\lim_{z\to z_k}M\twomat{0}{0}{e^{2i\theta} c_k}{0}
\end{equation}
\begin{equation}
\label{res-2}
\Res_{\overline{z_k} }M=\lim_{z\to \overline{z_k} }M\twomat{0}{-e^{-2i\theta} \overline{c_k}}{0}{0}.
\end{equation}
For $j=1,2,..., N_2$, $M(x, t; z)$ has simple poles at each $\pm z_j, \pm\overline{z_j}$ with
\begin{equation}
\label{res-3}
\Res_{z_j}M=\lim_{z\to z_j}M\twomat{0}{0}{e^{2i\theta} c_j}{0},
\end{equation}
\begin{equation}
\label{res-4}
\Res_{\overline{z_j} }M=\lim_{z\to \overline{z_j} }M\twomat{0}{-e^{-2i\theta} \overline{c_j}}{0}{0},
\end{equation}
\begin{equation}
\label{res-5}
\Res_{-z_j}M=\lim_{z\to -z_j}M\twomat{0}{e^{-2i\theta} c_j}{0} {0},
\end{equation}
\begin{equation}
\label{res-6}
\Res_{-\overline{z_j} }M=\lim_{z\to -\overline{z_j} }M\twomat{0}{0}{-e^{2i\theta} \overline{c_j}}{0}.
\end{equation}
\end{enumerate}
\end{problem}
\begin{definition}
\label{genericity}
We say that the initial condition $u_0$ is \emph{generic} if 
\begin{enumerate}
\item[1.] $\ba(z)$ ($a(z)$) associated to $u_0$ satisfies the simpleness and finiteness  assumptions stated in Remark \ref{remark-generic}.
\item[2.] For all  $\lbrace z_k \rbrace_{k=1}^{N_1}$ and  $\lbrace z_j \rbrace_{j=1}^{N_2}$ where $z_k=i\zeta_k$ and $z_j=\xi_j+i\eta_j$, 
$$4\zeta_k^2\neq 4\eta_j^2-12\xi_j^2, \quad 4\eta_{j_1}^2-12\xi_{j_2}^2\neq 4\eta_{j_2}^2-12\xi_{j_2}^2 $$
for all $j$, $k$ and $z_{j_1}\neq z_{j_2}$.
\end{enumerate}
\end{definition}
\begin{remark}
\label{re-arrange}
We arrange eigenvalues $\lbrace z_k \rbrace_{k=1}^{N_1}$ and  $\lbrace z_j \rbrace_{j=1}^{N_2}$ in the following way:
\begin{enumerate}
\item For $z_k=i\zeta_k$, $\zeta_k>0$, we have $\zeta_1<\zeta_2<...<\zeta_k<...<\zeta_{N_1}$.
\item For $z_j=\xi_j+i\eta_j$, $\xi_j, \eta_j>0$, we have 
$$4\eta_1^2-12\xi_1^2<...<4\eta_j^2-12\xi_j^2<...<4\eta_{\footnotesize{ N_2}}^2-12\xi_{N_2}^2.$$
\end{enumerate}
\end{remark}

\begin{figure}[h]
\caption{solitons and breathers}
\begin{tikzpicture}[scale=0.8]
\draw [->] (-4,0)--(3,0);
\draw (4,0)--(3,0);
\draw [->] (0,-4)--(0,3);
\draw (0,3)--(0,4);
 \pgfmathsetmacro{\a}{1}
    \pgfmathsetmacro{\b}{1.7320} 
   \draw plot[domain=-1.2:1.2] ({\a*cosh(\x)},{\b*sinh(\x)});
    \draw plot[domain=-1.2:1.2] ({-\a*cosh(\x)},{\b*sinh(\x)});
     \draw[dashed] (-2 , -3.464)--(2 , 3.464);
   \draw[dashed] (-2 , 3.464)--(2 , -3.464);
    \pgfmathsetmacro{\c}{1.414}
    \pgfmathsetmacro{\d}{2.449} 
   \draw plot[red, domain=-1:1] ({\c*cosh(\x)},{\d*sinh(\x)});
    \draw plot[domain=-1:1] ({-\c*cosh(\x)},{\d*sinh(\x)});
    \pgfmathsetmacro{\e}{0.5}
    \pgfmathsetmacro{\f}{0.866} 
   \draw plot[domain=-1.8:1.8] ({\e*cosh(\x)},{\f*sinh(\x)});
   \draw plot[domain=-1.8:1.8] ({-\e*cosh(\x)},{\f*sinh(\x)});
    \draw plot[domain=-1.2:1.2] ( {\a*sinh(\x)}, {\b*cosh(\x)} );
     \draw plot[red, domain=-1:1] ({\c*sinh(\x)},{\d*cosh(\x)});
     \draw plot[domain=-1.8:1.8] ({\e*sinh(\x)},{\f*cosh(\x)});
   \draw plot[domain=-1.2:1.2] ( {-\a*sinh(\x)}, {-\b*cosh(\x)} );
     \draw plot[red, domain=-1:1] ({-\c*sinh(\x)},{-\d*cosh(\x)});
     \draw plot[domain=-1.8:1.8] ({-\e*sinh(\x)},{-\f*cosh(\x)});
 \draw	[fill, green]  (0, 1.1)		circle[radius=0.05];	    
 \draw	[fill, green]  (0, 2)		circle[radius=0.05];	    
 \draw	[fill, green]  (0, 2.6)		circle[radius=0.05];	    
  \draw	[fill, green]  (0, -1.1)		circle[radius=0.05];	    
 \draw	[fill, green]  (0, -2)		circle[radius=0.05];	    
 \draw	[fill, green]  (0, -2.6)		circle[radius=0.05];	    
  \draw	[fill, blue]  (0.5211, 1.953)		circle[radius=0.05];	  
   \draw	[fill, blue]  (-0.5211, 1.953)		circle[radius=0.05];	   
    \draw[fill, blue]  (0.5211, 1.953)		circle[radius=0.05];	  
   \draw	[fill, blue]  (-0.5211, 1.953)		circle[radius=0.05];	   
    \draw[fill, blue]  (0.379, 1.087)		circle[radius=0.05];	  
   \draw	[fill, blue]  (-0.379, 1.087)		circle[radius=0.05];	   
    \draw[fill, blue]  (1.073, 3.074)		circle[radius=0.05];	  
   \draw	[fill, blue]  (-1.073, 3.074)		circle[radius=0.05];	   
   \draw	[fill, blue]  (0.5211, -1.953)		circle[radius=0.05];	  
   \draw	[fill, blue]  (-0.5211, -1.953)		circle[radius=0.05];	   
    \draw[fill, blue]  (0.5211, -1.953)		circle[radius=0.05];	  
   \draw	[fill, blue]  (-0.5211, -1.953)		circle[radius=0.05];	   
    \draw[fill, blue]  (0.379, -1.087)		circle[radius=0.05];	  
   \draw	[fill, blue]  (-0.379, -1.087)		circle[radius=0.05];	   
    \draw[fill, blue]  (1.073, -3.074)		circle[radius=0.05];	  
   \draw	[fill, blue]  (-1.073, -3.074)		circle[radius=0.05];	   
   \draw	[fill, red]  (1.1855, 1.10268)		circle[radius=0.05];	  
   \draw	[fill, red]  (-1.1855, 1.10268)	circle[radius=0.05];	   
    \draw	[fill, red]  (1.529, 1.006)		circle[radius=0.05];	  
   \draw	[fill, red]  (-1.529, 1.006)	circle[radius=0.05];	   
   \draw	[fill, red]  (0.523, 0.2637)		circle[radius=0.05];	  
   \draw	[fill, red]  (-0.523, 0.2637)	circle[radius=0.05];	   
    \draw	[fill, red]  (1.1855, -1.10268)		circle[radius=0.05];	  
   \draw	[fill, red]  (-1.1855, -1.10268)	circle[radius=0.05];	   
    \draw	[fill, red]  (1.529, -1.006)		circle[radius=0.05];	  
   \draw	[fill, red]  (-1.529, -1.006)	circle[radius=0.05];	   
   \draw	[fill, red]  (0.523, -0.2637)		circle[radius=0.05];	  
   \draw	[fill, red]  (-0.523, -0.2637)	circle[radius=0.05];	   
    \end{tikzpicture}
 \begin{center}
  \begin{tabular}{ccc}
Soliton ({\color{green} $\bullet$})	&	
Breather ({\color{red} $\bullet$} {\color{blue} $\bullet$} ) 
\end{tabular}
 \end{center}
\end{figure}
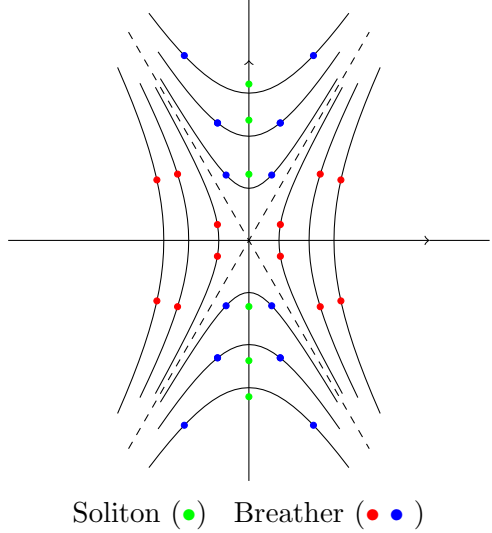

\begin{remark}
For each pole $z_k (z_j) \in \mathbb{C}^+$, let $\gamma_k, (\gamma_j)$ be a circle centered at $z_k (z_j)$ of sufficiently small radius to be lie in the open upper half-plane and to be disjoint from all other circles. By doing so we replace the residue conditions \eqref{res-1}-\eqref{res-6} of the Riemann-Hilbert problem with Schwarz invariant jump conditions across closed contours (see Figure \ref{figure-zeta}). The equivalence of this new RHP on augmented contours with the original one is a well-established result (see \cite{Zhou89} Sec 6). The purpose of this replacement is to
\begin{enumerate}
\item make use of the \textit{vanishing lemma} from \cite[Theorem 9.3]{Zhou89} .
\item Formulate the Beals-Coifman representation of the solution of \eqref{eq:fMKdV}.
\end{enumerate}
We now rewrite the the jump conditions of Problem \ref{RHP-1}:
$M(x,z)$ is analytic in $\mathbb{C}\setminus \Sigma$ and has continuous boundary values $M_\pm$ on $\Sigma$
and  $M_\pm$ satisfy
$$ M_+(x,t; z)=M_-(x,t; z)e^{-i\theta(x,t; z)\ad\sigma_3}v(z)$$
where 
\begin{align*}
v(z)=\twomat{1+|r(z)|^2}{\overline{r(z)}} {r(z) }{1}, \quad z\in \bbR
\end{align*} 
and 
$$
v(z) = 	\begin{cases}
						\twomat{1}{0}{\dfrac{c_k }{z-z_k}}{1}	&	z\in \gamma_k, \\
						\\
						\twomat{1}{\dfrac{\overline{c_k}}{z -\overline{z_k }}}{0}{1}
							&	z \in \gamma_k^*
					\end{cases}
$$
and
$$
v(z) = 	\begin{cases}
						\twomat{1}{0}{\dfrac{c_j}{z-z_j}}{1}	&	z\in \gamma_j, \\
						\\
						\twomat{1}{\dfrac{\overline{c_j}}{z -\overline{z_j}}}{0}{1}
							&	z \in \gamma_j^*\\
							\\
							\twomat{1}{\dfrac{-c_j}{z + z_j}}{0}{1}	&	z\in -\gamma_j, \\
						\\
						\twomat{1}{0}{\dfrac{-\overline{c_j}}{z +\overline{z_j }}}{1}
							&	z \in -\gamma_j^*
					\end{cases}
$$					
\begin{figure}
\caption{The Augmented Contour $\Sigma$}
\vspace{.5cm}
\label{figure-zeta}
\begin{tikzpicture}[scale=0.75]
\draw[ thick] (0,0) -- (-3,0);
\draw[ thick] (-3,0) -- (-5,0);
\draw[thick,->,>=stealth] (0,0) -- (3,0);
\draw[ thick] (3,0) -- (5,0);
\node[above] at 		(2.5,0) {$+$};
\node[below] at 		(2.5,0) {$-$};
\node[right] at (3.5 , 2) {$\gamma_j$};
\node[right] at (3.5 , -2) {$\gamma_j^*$};
\node[left] at (-3.5 , 2) {$-\gamma_j^*$};
\node[left] at (-3.5 , -2) {$-\gamma_j$};
\draw[->,>=stealth] (-2.6,2) arc(360:0:0.4);
\draw[->,>=stealth] (3.4,2) arc(360:0:0.4);
\draw[->,>=stealth] (-2.6,-2) arc(0:360:0.4);
\draw[->,>=stealth] (3.4,-2) arc(0:360:0.4);
\draw [red, fill=red] (-3,2) circle [radius=0.05];
\draw [red, fill=red] (3,2) circle [radius=0.05];
\draw [red, fill=red] (-3,-2) circle [radius=0.05];
\draw [red, fill=red] (3,-2) circle [radius=0.05];
\node[right] at (5 , 0) {$\bbR$};
\draw [green, fill=green] (0, 1) circle [radius=0.05];
\draw[->,>=stealth] (0.3, 1) arc(360:0:0.3);
\draw [green, fill=green] (0, -1) circle [radius=0.05];
\draw[->,>=stealth] (0.3, -1) arc(0:360:0.3);
 \draw[dashed] (-2 , -3.464)--(2 , 3.464);
   \draw[dashed] (-2 , 3.464)--(2 , -3.464);
\draw [blue, fill=blue] (1,3) circle [radius=0.05];
\draw[->,>=stealth] (1.2, 3) arc(360:0:0.2);
\draw [blue, fill=blue] (-1,3) circle [radius=0.05];
\draw[->,>=stealth] (-0.8, 3) arc(360:0:0.2);
\draw [blue, fill=blue] (1,-3) circle [radius=0.05];
\draw[->,>=stealth] (1.2, -3) arc(0:360:0.2);
\draw [blue, fill=blue] (-1,-3) circle [radius=0.05];
\draw[->,>=stealth] (-0.8, -3) arc(0:360:0.2);
\draw  (0.5, 0) arc(0:60:0.5);
\node[right] at (0.5 , 0.2) {\footnotesize ${\pi}/{3}$};
\node[above] at 		(0, 1.3) {$\gamma_k$};
\node[below] at 		(0, -1.3) {$\gamma_k^*$};
\end{tikzpicture}
\begin{center}
\begin{tabular}{ccc}
Soliton ({\color{green} $\bullet$})	&	
Breather ({\color{red} $\bullet$} {\color{blue} $\bullet$} ) 
\end{tabular}
\end{center}
\end{figure}
It is well-known that $v_\theta$ admits triangular factorization:
$$v_\theta=(1-w_{\theta-})^{-1}(1+w_{\theta+}). $$
We define
\begin{align*}
\mu= m_+(1-w_\theta^- )^{-1}=m_-(1+w_\theta^+)
\end{align*}
 then the solvability of the RHP above is equivalent to the solvability of the following Beals-Coifman integral equation:
\begin{align}
\label{BC-int}
\mu(z; x,t) &= I+C^+_\Sigma\mu w_\theta^- +C^-_\Sigma\mu w_\theta^+\\
                \nonumber
        &=I+C^+_\bbR\mu w_\theta^- +C^-_\bbR\mu w_\theta^+\\
        \nonumber
        &\qquad+ 
        \Twomat{\sum_{k=1}^{N_1} \dfrac{\mu_{12}(z_k )c_k e^{2i\theta(z_k)} }{z-z_k}  }
        { -\sum_{k=1}^{N_1} \dfrac{\mu_{11}( \overline{z_k}) {\overline{c_k}} e^{-2i\theta(   \overline{z_k }  )} }{z- \overline{z_k }}   }
        {\sum_{k=1}^{N_1} \dfrac{\mu_{22}(z_k )c_k e^{2i\theta(z_k )} }{z-z_k}  }
        {-\sum_{k=1}^{N_1} \dfrac{\mu_{21}( \overline{z_k }) {\overline{c_k}} e^{-2i\theta(   \overline{z_k}  )} }{z- \overline{z_i}} }\\
         \nonumber
        &\qquad + \Twomat{\sum_{j=1}^{N_2} \dfrac{\mu_{12}(z_j)c_i e^{2i\theta(z_j)} }{z-z_j}  }
        { -\sum_{j=1}^{N_2} \dfrac{\mu_{11}( \overline{z_j }) {\overline{c_j }} e^{-2i\theta(   \overline{z_j }  )} }{z- \overline{z_j }}   }
        {\sum_{j=1}^{N_2} \dfrac{\mu_{22}(z_j)c_j e^{2i\theta(z_j )} }{z-z_j }  }
        {-\sum_{j=1}^{N_2} \dfrac{\mu_{21}( \overline{z_j }) {\overline{c_j}} e^{-2i\theta(   \overline{z_j}  )} }{z- \overline{z_j}} }\\
         \nonumber
         &\qquad +\Twomat
        { -\sum_{j=1}^{N_2} \dfrac {\mu_{12}( -\overline{z_j }) {\overline{c_j }} e^{2i\theta(  - \overline{z_j }  )} }{z+ \overline{z_j }}   }
        {\sum_{j=1}^{N_2} \dfrac{\mu_{11}(-z_j)c_j e^{-2i\theta(-z_j)} }{z+z_j}  }
       {-\sum_{j=1}^{N_2} \dfrac{\mu_{22}( -\overline{z_j }) {\overline{c_j}} e^{2i\theta(  - \overline{z_j}  )} }{z + \overline{z_j}} }
        {\sum_{j=1}^{N_2} \dfrac{\mu_{21}(-z_j)c_j e^{-2i\theta(-z_j )} }{z+z_j }  }.
\end{align}
From the solution of Problem \ref{RHP-1}, we recover
\begin{align}
\label{mkdv.u}
u(x,t) &= \lim_{z \rarr \infty} 2 z m_{12}(x,t,z)\\
\label{mkdv.BC}
        &= \left( \dfrac{1}{\pi} \int_\Sigma \mu (w_\theta^-+w_\theta^+) \right)_{12}\\
        &=\dfrac{1}{\pi} \int_\bbR \mu_{11}(x,t;z) \overline{r}(z) e^{-2i\theta} dz +\sum_{k=1}^{N_1} {\mu_{11}( \overline{z_k}) {\overline{c_k }} e^{-2i\theta(   \overline{z_k }  )} }  \\
        & \quad -\sum_{j=1}^{N_2}{\mu_{11}( \overline{z_j }) {\overline{c_j }} e^{-2i\theta(   \overline{z_j }  )} } +\sum_{j=1}^{N_2} {\mu_{11}(-z_j)c_j e^{-2i\theta(-z_j)} } 
\end{align}
where the limit is taken in $\bbC\setminus \Sigma$ along any direction not tangent to $\Sigma$.  
\end{remark}

\subsection{Single soliton and single breather solution}
If we assume $r=0$ and $\ba$ has exactly one simple zero at $z=i\zeta$, $\zeta>0$ and let $c$ be the norming constant. Notice that $c$ is purely imaginary, then we let 
\begin{align*}
\eps_\pm =\begin{cases}
1, &\quad \Imag c>0\\
-1, & \quad  \Imag c<0
\end{cases}
\end{align*}
then equation \eqref{eq:fMKdV} admits the following single-soliton solution \cite{Wa} :
\begin{equation}
\label{1-soliton}
u(x,t)=2 \zeta \eps_\pm \sech(-2\zeta(x-4\zeta^2 t)+\omega ).
\end{equation}
where 
$$\omega=\log \left(  \dfrac{ \left\vert c \right\vert  }{2\zeta} \right) $$
If we assume $r=(0)$ and $\ba$ has exactly two simple zeros at $z=\pm\xi+i\eta$, $\eta>0$ and let $c=A+iB$ be the norming constant, then Equation \eqref{eq:fMKdV} admits the following one-breather solution \cite{Wa} :
\begin{equation}
\label{1-breather}
u(x,t)=-4\dfrac{\eta}{\xi}\dfrac{ \xi \cosh(\nu_2+\omega_2 ) \sin(\nu_1+\omega_1 )+\eta \sinh(\nu_2+\omega_2  ) \cos(\nu_1+\omega_1 )}{\cosh^2(\nu_2+\omega_2  ) + (\eta/\xi)^2 \cos^2(\nu_1+\omega_1 )}
\end{equation}
with
\begin{align*}
\nu_1 &=2\xi (x+4(\xi^2-3\eta^2)t )\\
\nu_2 &=2\eta (x -4(\eta^2-3\xi^2)t )
\end{align*}
and
\begin{align}
\label{omega-1}
\tan \omega_1&= \dfrac{B\xi-A\eta}{A\xi+B\eta}\\
e^{-\omega_2} &= \left\vert \dfrac{\xi }{ 2\eta } \right\vert \sqrt{\dfrac{A^2+B^2}{\xi^2+\eta^2}  }.
\end{align}
From above we observe that soliton has velocity $\textrm{v}_s=4\eta^2$, always traveling in the positive direction and breather has velocity $\textrm{v}_b=4\eta^2-12\xi^2$, which means breather can travel in both directions. Also notice that 
\begin{equation}
\label{theta-re}
\text{Re}i\theta=(4\eta^2-12\xi^2+12z_0^2)\eta t
\end{equation}
If we fix the velocity $\textrm{v}_b=x/t$, then $ 4\eta^2-12\xi^2=\textrm{v}_b$ implies that the hyperbola pass through the stationary points 
\begin{equation}
\label{stationary}
\pm z_0=\pm \sqrt{\dfrac{-x}{12 t}  }
\end{equation}
Conversely, if $a(z)$ has zeros on the hyperbola  $ 4\eta^2-12\xi^2=x/t$, we expect  breathers moving with velocity $x/t$.

\begin{remark}
\label{painleve-soliton}
Rewrite \eqref{theta-re} as
$$\text{Re} i\theta(x,t; z)=t^{1/3}\left( 4(-3u^2v+v^3)t^{2/3} -\dfrac{x}{t^{1/3}} v  \right). $$
In the Painlev\'e  region where we set $|x/t^{1/3}|\leq c$, it is easy to see that for $\sqrt{3}u >v$ $( \sqrt{3}u <v)$, we have $\text{Re} i\theta< 0$ (  $\text{Re} i\theta> 0$). In the soliton region where $x>0$, $x/t=\mathcal{O}(1)$ we write 
$$ \text{Re} i\theta(x,t; z)=t\left( 4(-3u^2v+v^3)-\dfrac{x}{t} v  \right) .$$ 
It is now clear that if we set $x/t=\textrm{v}_{b_j}=4\eta_j^2-12\xi_j^2$, then $ \text{Re} i\theta(x,t; z_j)=0$.
\end{remark}

\subsection{Main results}
The central result of this paper is to describe the long-time behavior of  the  solutions $u$
 of \eqref{eq:fMKdV} in different regions respectively.
 
%
%

\begin{figure}[h!]
\caption{Three Regions}
\vskip 15pt

\begin{tikzpicture}[scale=0.7]
\draw[fill] (0,0) circle[radius=0.075];
\draw[thick] 	(-8,0) -- (-2,0);
\draw[thick]		(-2,0) -- (0,0);
\draw[->,thick,>=stealth]	(0,0) -- (6,0);
\draw[thick]		(2,0) -- (7,0);
\draw		(-7.2, -0.2) -- (-1, -0.2); 
\node[above] at (-3.5, 0.3) {I};
\node[below] at (-5, -0.7) {$x<0$\,, $|x/t|=\mathcal{O}(1)$};
\node[above] at (0,0.6){II};
\node[above] at (0,.1){0};
\node[below] at (0, -0.2){$ |x/t^{1/3}|=\mathcal{O}(1) $};
\draw	(-2, 0.2) -- (2, 0.2); 
\node[right] at (7, 0) {$x$-axis};
\draw (1,-0.2)--(7, -0.2);
\node [above] at (4, 0.3) {III};
\node[below] at (5.5, -0.4){$x>0$\,, $|x/t|=\mathcal{O}(1)$};
\end{tikzpicture}
\label{fig: regions}
\end{figure}
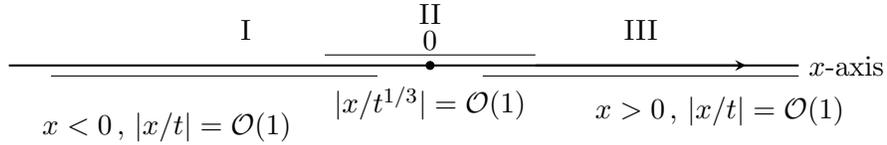

%
%
We are mainly interested in the long time asymptotics of mKdV in the following three regions:
\begin{itemize}
\item oscillatory region:  $x<0$, $|x/t|=\mathcal{O}(1)$ as $t\to \infty$. In this region, we can observe breathers traveling in the left direction.
\item self-similar region: $|x/t^{1/3}| \leq c$ as $t\to \infty$. This region does not have breathers and solitons as $t\to\infty$.
\item soliton region: $x>0$, $|x/t|=\mathcal{O}(1)$ as $t\to \infty$. In this region, we can observe breathers and solitons traveling in the right direction.
\end{itemize}

\begin{remark}
The long time asymptotics for overlaps of the regions have been studied in the previous paper \cite[Theorem 1.6]{CL19}. There are no solitons and breathers in those overlap regions.
\end{remark}

\subsubsection{Long-time asymptotics}
Our main results is the following detailed long-time asymptotics of the solution to the focusing mKdV. This also verifies the soliton resolution for  generic data.
\begin{theorem}\label{main1}
	Given initial the data $u_{0}\in H^{2,1}\left(\mathbb{R}\right)$
	and assume $u_{0}$ is generic in the sense of Definition  \ref{genericity}. Suppose the initial data produce the scattering data\[
	S=\left\{ r\left(z\right),\left\{ z_{k},c_{k}\right\} _{k=1}^{N_{1}},\left\{ z_{j},c_{j}\right\} _{j=1}^{N_{2}}\right\} \in H^{1,2}\left(\mathbb{R}\right)\oplus\mathbb{C}^{2N_{1}}\oplus\mathbb{C}^{2N_{2}}
	\] as in Subsection \ref{subsec:inverse}. We first arrange $z_{j}=\xi_{j}+i\eta_{j},\ \xi_{j},\eta_{j}>0$
	and suppose for some $1\leq\ell_{0}\leq N_{2}$, one has
	\[
	4\eta_{1}^{2}-12\xi_{1}^{2}<\ldots<4\eta_{\ell_{0}}^{2}-12\xi_{\ell_{0}}^{2}<0<4\eta_{\ell_{0}+1}^{2}-12\xi_{\ell_{0}+1}^{2}<\ldots<4\eta_{N_{2}}^{2}-12\xi_{N_{2}}^{2}.
	\]
	Secondly, we list $z_{k}=i\zeta_{k},\,\zeta_{k}>0$ as
	\[
	0 <\zeta_{1}<\ldots<\zeta_{N_{1}}.
	\]
	
	Let $u$ be the solution the focusing mKdV
	\[
	\partial_{t}u+\partial_{xxx}u+6u^{2}\partial_{x}u=0,\ \left(x,t\right)\in\mathbb{R}\times\mathbb{R}^{+}
	\]
	with initial data $u_{0}$ given by the reconstruction formula (the
	Beals-Coifman solution). Denote
	\[
	\tau=z_{0}^{3}t,\,\pm z_{0}=\pm\sqrt{\frac{-x}{12t}}.\]
	Then the solution $u$ can be written as
	the superposition of breathers, solitons and the radiation as following:
	\[
	u(x,t)=\sum_{\ell=1}^{N_{2}}u_{\ell}^{\left(br\right)}\left(x,t\right)+\sum_{\ell=1}^{N_{1}}u_{\ell}^{\left(so\right)}\left(x,t\right)+R\left(x,t\right).
	\]
	\begin{enumerate}
\item[(1).]	For the breather part,
\begin{itemize}
\item[(i)] if $\ell\leq\ell_{0}$,
	\begin{equation}
	\label{eq:thmbre}
	u_{\ell}^{\left(br\right)}\left(x,t\right)=-4\frac{\eta_{\ell}}{\xi_{\ell}}\frac{\xi_{\ell}\cosh\left(\nu_{2}+\tilde{\omega}_{2}\right)\sin\left(\nu_{1}+\tilde{\omega}_{1}\right)+\eta_{\ell}\sinh\left(\nu_{2}+\tilde{\omega}_{2}\right)\cos\left(\nu_{1}+\tilde{\omega}_{1}\right)}{\cosh^{2}\left(\nu_{2}+\tilde{\omega}_{2}\right)+\left(\eta_{\ell}/\xi_{\ell}\right)^{2}\cos^{2}\left(\nu_{1}+\tilde{\omega}_{1}\right)}
	\end{equation}
	where
	\begin{equation}
	\nu_{1}=2\xi_{\ell}\left(x+4\left(\xi_{\ell}^{2}-3\eta_{\ell}^{2}\right)t\right),\label{eq:thmbre1}
	\end{equation}
	\begin{equation}
	\nu_{2}=2\eta_{\ell}\left(x-4\left(\eta_{\ell}^{2}-3\xi_{\ell}^{2}\right)t\right),\label{eq:thmbre2}
	\end{equation}
	and
	\begin{equation}
	\tan\left(\tilde{\omega}_{1}\right)=\frac{\tilde{B}\xi_{\ell}-\tilde{A}\eta_{\ell}}{\tilde{A}\xi_{\ell}+\tilde{B}\eta_{\ell}},\ e^{-\tilde{\omega}_{2}}=\left|\frac{\xi_{\ell}}{2\eta_{\ell}}\right|\sqrt{\frac{\tilde{A}^{2}+\tilde{B}^{2}}{\xi_{\ell}^{2}+\eta_{\ell}^{2}}}\label{eq:thmbre3}
	\end{equation}
	here $\tilde{A}$ and $\tilde{B}$ are given as
	\begin{equation}
	\tilde{c}_{\ell}=c_{\ell}\delta\left(z_{\ell}\right)^{-2}=\tilde{A}+i\tilde{B}.\label{eq:thmbre4}
	\end{equation}
	where the scalar function $\delta(z)$ is given by \[
	\delta\left(z\right)=\left(\prod_{k=1}^{N_{1}}\frac{z-\overline{z_{k}}}{z-z_{k}}\right)\left(\prod_{z_{j}\in\mathcal{B}_{\ell}}\frac{z-\overline{z_{j}}}{z-z_{j}}\right)\left(\prod_{z_{j}\in\mathcal{B}_{\ell,}}\frac{z+z_{j}}{z+\overline{z_{j}}}\right)\left(\frac{z-z_{0}}{z+z_{0}}\right)^{i\kappa}e^{\chi\left(z\right)}
	\]with
	\begin{align*}
	\chi\left(z\right) &=\frac{1}{2\pi i}\int_{-z_{0}}^{z_{0}}\log\left(\frac{1+\left|r\left(\zeta\right)\right|^{2}}{1+\left|r\left(z_{0}\right)\right|^{2}}\right)\frac{d\zeta}{\zeta-z},\\
	\kappa &=-\frac{1}{2\pi}\log\left(1+\left|r\left(z_{0}\right)\right|^{2}\right),\\
	\end{align*}and
	\begin{equation}
	\mathcal{B}_{\ell} =\left\{ z_{j}=\xi_{j}+i\eta_{j}:\ 4\eta_{j}^{2}-12\xi_{j}^{2}>4\eta_{\ell}^{2}-12\xi_{\ell}^{2}\right\}.\label{eq:B_l}
	\end{equation}
	\item[(ii)] If $\ell_{0}+1\leq\ell$, we have the same expressions as \eqref{eq:thmbre},
	\eqref{eq:thmbre1}, \eqref{eq:thmbre2}, \eqref{eq:thmbre3} and \eqref{eq:thmbre4}
	but with $\tilde{A}$ and $\tilde{B}$ are given as
	\[
	\tilde{c}_{\ell}=c_{\ell}\psi\left(z_{\ell}\right)^{-2}=\tilde{A}+i\tilde{B}
	\]
	where the scalar function $\psi$ is defined by
	\[
	\psi\left(z\right)=\left(\prod_{z_{k}\in\mathcal{B}_{\ell,s}}\frac{z-\overline{z_{k}}}{z-z_{k}}\right)\left(\prod_{z_{j}\in\mathcal{B}_{\ell,b}}\frac{z-\overline{z_{j}}}{z-z_{j}}\right)\left(\prod_{z_{j}\in\mathcal{B}_{\ell,b}}\frac{z_{\ell}+z_{j}}{z_{\ell}+\overline{z_{j}}}\right)
	\]
 with
	\begin{align*}
		\mathcal{B}_{\ell} & =\left\{ z_{j}=\xi_{j}+i\eta_{j}:\ 4\eta_{j}^{2}-12\xi_{j}^{2}>4\eta_{\ell}^{2}-12\xi_{\ell}^{2}\right\} \bigcup\left\{ z_{k}=i\zeta_{k}:\ 4\zeta_{k}^{2}>4\eta_{\ell}^{2}-12\xi_{\ell}^{2}\right\} \\
		& =:\mathcal{B}_{\ell,b}\bigcup\mathcal{B}_{\ell,s}.
	\end{align*}
	\end{itemize}
\item[(2).]	For the soliton part, we have
	\begin{equation}
	u_{\ell}^{\left(so\right)}\left(x,t\right)=2\zeta_{\ell} \eps_{\pm, \ell} \sech\left(-2\zeta_{\ell}\left(x-4\zeta_{\ell}^{2}t\right)+\omega_{\ell}\right)\label{eq:soli}
	\end{equation}
	with
	\[
	\omega_{\ell}=\log\left(\frac{\left|c_{\ell}\right|}{2\zeta_{\ell}}\right)+2\sum_{z_{k}\in S_{\ell,s}}\log\left\vert \dfrac  {z_{\ell}-z_{k}}{z_{\ell}-\overline{z_{k}}}\right\vert+2\sum_{z_{j}\in\mathcal{S}_{\ell,b}}\log\left\vert \dfrac {z_{\ell}-z_{j}}{z_{\ell}-\overline{z_{j}}}\right\vert+2\sum_{z_{j}\in\mathcal{S}_{\ell,b}}\log\left\vert   \dfrac {z_{\ell}+\overline{z_{j}} }{z_{\ell}+z_{j}}\right\vert.
	\]
	where
	\begin{align*}
		\mathcal{S}_{\ell} & =\left\{ z_{j}=\xi_{j}+i\eta_{j}:\ 4\eta_{j}^{2}-12\xi_{j}^{2}>4\zeta_{\ell}^{2}\right\} \bigcup\left\{ z_{k}=i\zeta_{k}:\ 4\zeta_{k}^{2}>4\zeta_{\ell}^{2}\right\} \\
		& =:\mathcal{S}_{\ell,b}\bigcup\mathcal{S}_{\ell,s}.
	\end{align*}
\item[(3).]	Finally, the radiation term, we have the following asymptotics.
	\begin{itemize}
\item[(i)]	In the soliton region, i.e., Region $\text{III}$, we have
	\begin{equation}
	\left|R\left(x,t\right)\right|\lesssim\frac{1}{t}.\label{eq:solirad}
	\end{equation}
\item[(ii)]		In the self-similar region, i.e., Region $\text{II}$, for $4<p<\infty$,
	one has
	\begin{equation}
	R\left(x,t\right)=\frac{1}{\left(3t\right)^{\frac{1}{3}}}P\left(\frac{x}{\left(3t\right)^{\frac{1}{3}}}\right)+\mathcal{O}\left(t^{\frac{2}{3p}-\frac{1}{2}}\right)\label{eq:self}
	\end{equation}
	where $P$ is a solution to the Painlev\'e II equation
	\[
	P''\left(s\right)-sP'\left(s\right)+2P^{3}\left(s\right)=0
	\]
	determined by $r\left(0\right)$.
	
\item[(iii)]	In the oscillatory region, i.e. Region $\text{I}$, there are two separate cases.
\begin{itemize}
\item[(a)] If we choose the frame
	$x=\mathrm{v}t$ with $\mathrm{v}=4\eta_{\ell}^{2}-12\xi_{\ell}^{2}<0$,
	then one can see the influence of the breather on the radiation strongly
	and explicitly as the following:
	\begin{equation}
	R\left(x,t\right)=u_{as}\left(x,t\right)+\mathcal{O}\left(\left(z_{0}t\right)^{-\frac{3}{4}}\right)\label{eq:oscasymp2}
	\end{equation}
	where
	\begin{align*}
		u_{as}\left(x,t\right) & =\frac{1}{\sqrt{48tz_{0}}}\left(\left(m_{11}^{\left(br\right)}\left(-z_{0}\right)^{2}\left(i\delta_{A}^{0}\right)^{2}\overline{\beta}_{12}\right)+m_{12}^{\left(br\right)}\left(-z_{0}\right)^{2}\left(\left(i\delta_{A}^{0}\right)^{2}\overline{\beta}_{21}\right)\right)\\
		& +\frac{1}{\sqrt{48tz_{0}}}\left(\left(m_{11}^{\left(br\right)}\left(-z_{0}\right)^{2}\left(i\delta_{B}^{0}\right)^{2}\beta_{12}\right)-m_{12}^{\left(br\right)}\left(-z_{0}\right)^{2}\left(\left(i\delta_{B}^{0}\right)^{2}\beta_{21}\right)\right)
	\end{align*}
	with some explicit constants $m_{11}^{\left(br\right)}\left(-z_{0}\right)$,
	$m_{12}^{\left(br\right)}\left(-z_{0}\right)$, $m_{11}^{\left(br\right)}\left(z_{0}\right)$,
	$m_{12}^{\left(br\right)}\left(z_{0}\right)$ from the breather matrix, see Section \ref{sec:local},
	\[
	\beta_{12}=\frac{\sqrt{2\pi}e^{i\pi/4}e^{-\pi\kappa}}{r\left(z_{0}\right)\Gamma\left(-i\kappa\right)},\ \beta_{21}=\frac{-\sqrt{2\pi}e^{-i\pi/4}e^{-\pi\kappa}}{\overline{r\left(z_{0}\right)}\Gamma\left(i\kappa\right)},
	\]
	\[
	\delta_{A}^{0}=\left(192\tau\right)^{i\kappa/2}e^{-8i\tau}e^{\chi\left(-z_{0}\right)}\eta_{0}\left(-z_{0}\right),
	\]
	\[
	\delta_{B}^{0}=\left(192\tau\right)^{-i\kappa/2}e^{8i\tau}e^{\chi\left(z_{0}\right)}\eta_{0}\left(z_{0}\right)
	\]
	and
	\[
	\eta_{0}\left(\pm z_{0}\right)=\left(\prod_{z_{k}=1}^{N_{1}}\frac{\pm z_{0}-\overline{z_{k}}}{\pm z_{0}-z_{k}}\right)\left(\prod_{z_{j}\in\mathcal{B}_{\ell}}\frac{\pm z_{0}-\overline{z_{j}}}{\pm z_{0}-z_{j}}\right)\left(\prod_{z_{j}\in\mathcal{B}_{\ell}}\frac{\pm z_{0}+z_{j}}{\pm z_{0}+\overline{z_{j}}}\right)
	\]
	where $\mathcal{B}_{\ell}$ as \eqref{eq:B_l}.
\item[(b)]	If $\mathrm{v}\neq4\eta_{j}^{2}-12\xi_{j}^{2}$ for $1\leq j\leq N_{2}$, then
	we have
	\begin{equation}
	R\left(x,t\right)=u_{as}\left(x,t\right)+\mathcal{O}\left(\left(z_{0}t\right)^{-\frac{3}{4}}\right)
	\label{eq:Oscasymp2}
	\end{equation}
	where
	\begin{equation*}
	u_{as}\left(x,t\right)=\left(\frac{\kappa}{3tz_{0}}\right)^{\frac{1}{2}}\cos\left(16tz_{0}^{3}-\kappa\log\left(192tz_{0}^{3}\right)+\phi\left(z_{0}\right)\right)
	\end{equation*}
	with
	\begin{align*}
		\phi\left(z_{0}\right) & =\arg\Gamma\left(i\kappa\right)-\frac{\pi}{4}-\arg r\left(z_{0}\right)+\frac{1}{\pi}\int_{-z_{0}}^{z_{0}}\log\left(\frac{1+\left|r\left(\zeta\right)\right|^{2}}{1+\left|r\left(z_{0}\right)\right|^{2}}\right)\frac{d\zeta}{\zeta-z_{0}}\\
		& -4\left(\sum_{k=1}^{N_{1}}\arg\left(z_{0}-z_{k}\right)+\sum_{z_{j}\in B_{\ell}}\arg\left(z_{0}-z_{j}\right)+\sum_{z_{j}\in B_{\ell}}\arg\left(z_{0}+\overline{z_{j}}\right)\right).
	\end{align*}
	\end{itemize}
	\end{itemize}
	\end{enumerate}
\end{theorem}

\begin{remark}
	The two expressions \eqref{eq:oscasymp2} and  \eqref{eq:Oscasymp2}  above match each other since as the velocity
	of the frame moving away from the the velocity of the breather, $m_{12}^{\left(br\right)}\left(z_{0}\right)$
	provides the exponential decay in time and the remain terms combined
	together  give  the same asymptotics as the later expression
	up to terms exponential decay in time. These exponential
	decay rates depend on the gap between the velocity of the frame the
	the velocities of breathers.
\end{remark}

One can trace all the details in our analysis and notice that actually
it suffices to require the weights in $x$ to be $\left\langle x\right\rangle ^{s}$
with $s>\frac{1}{2}$. More precisely, note that first $s=1$ is used in the construction of Jost functions.
But actually, in that construction, we just need $L^{2,s}\left(\mathbb{R}\right),\,s>\frac{1}{2}$
and the potential in $L^{1}\left(\mathbb{R}\right)$. One can simply
check that $L^{1}\left(\mathbb{R}\right)\subset L^{2,s}\left(\mathbb{R}\right)$
for $s>\frac{1}{2}$. Secondly, $s=1$ is used in the analysis of
asymptotics of the Riemann-Hilbert problem but note that reflection coefficient in $H^{s}$
for $s>\frac{1}{2}$ is sufficient for us due to Sobolev's embedding
and the estimate of modulus of continuity. To estimate the $H^{s}$
norm of the reflection coefficient, by bijectivity, in terms of the
initial data, $L^{2,s}\left(\mathbb{R}\right)$ for $s>\frac{1}{2}$
is enough for us. Although in Zhou's work, he only deals with $s\in\mathbb{N}$,
the fractional results can be obtained simply by  interpolation.

\smallskip
Then using the lowest regularity for the local
well-posedness in $H^{k}\left(\mathbb{R}\right)$ with $k\ge\frac{1}{4}$ via contraction
obtained by Kenig-Ponce-Vega \cite{KPV} and the recent low regularity conservation laws due to Killip-Visan-Zhang \cite{KiViZh} and Koch-Tataru \cite{KoTa}, we can use
a global approximation argument to extend our long-time asymptotics
to $H^{k,s}$ with $k\ge\frac{1}{4}$ and $\ell>\frac{1}{2}$.

\begin{theorem}
	\label{thm:main2}Suppose that $u_{0}\in H^{k,s}\left(\mathbb{R}\right)$
	with $s>\frac{1}{2}$ and $k\geq\frac{1}{4}$ is generic in the sense of  Definition  \ref{genericity}, the long-time asymptotics
	as in Theorem \ref{main1} hold for the solution to the focusing
	mKdV
		\begin{equation}
	\partial_{t}u+\partial_{xxx}u+6u^{2}\partial_{x}u=0,\,u\left(0\right)=u_{0}\in H^{k,s}\left(\mathbb{R}\right)\label{eq:fmkdv1}
	\end{equation}
	given by the Duhemel representation
	\begin{equation}
		u=W\left(t\right)u_{0}-\int_{0}^{t}W\left(t-s\right)\left(6u^{2}\partial_{x}u\left(s\right)\right)\,ds\label{eq:mild-1}
	\end{equation}
	where
	\[
	W\left(t\right)u_{0}=e^{-t\partial_{xxx}}u_{0}\ \text{and}\ \mathcal{F}_{x}\left[W\left(t\right)u_{0}\right]\left(\xi\right)=e^{it\xi^{3}}\hat{u}_{0}\left(\xi\right)
	\]
\end{theorem}

The key point is that Beals-Coifman solutions have asymptotics and strong solutions in the sense of Duhamel can be used to pass to limits. For smooth data, Beals-Coifman solutions and strong solutions are the same. Our computations for Beals-Coifman solutions show that the error estimates only depend on weights but  not  the regularity of initial data. To illustrate this philosophy, we have the follow diagram:
\begin{displaymath}
    \xymatrix{ \mathcal{S}  \ni u_0(x) \ar[r]^{IST} \ar[d]_{Duhamel} & \text{asymptotics}\\
          \mathcal{S}      \ni u(x,t)  \ar[r]_{Approx}  & u(x,t)\in H^{1/4, 1} \ar[u]}
\end{displaymath}
\begin{remark}
	In order to get precise behavior of the radiation, the
	weights in the Sobolev norms are necessary. From the inverse scattering point of view, these
	weights are used to construct Jost functions and in the $\overline{\partial}$-interpolation
	argument. On the other hand, from the stationary phase point of view, to
	obtain the precise asymptotics of the oscillatory integral, we need
	the function which is multiplied by an oscillatory factor to be defined pointwise so that we can localize the leading order behavior
	to the stationary point. The weights
	precisely give us the pointwise meaning of the function which is integrated
	again an oscillatory factor via Sobolev embedding. For the linear scattering theory, one can probably conclude the long-time behavior of the nonlinear equation matches a linear flow using unweighted norms. But in our setting, the scattering behavior is
	nonlinear, so we have to carry out the precise asymptotics and hence the weights can not be avoided.
\end{remark}

Hereinafter, for the sake of simplicity, we focus on the case $s=1$.

\smallskip

\subsubsection{Asymptotic stability}

As by products of our long-time asymptotics,
the full asymptotic stability of solitons/breathers of the mKdV follows
naturally.  First of all, we state the asymptotic the stability of a breather traveling to the left separately. Recall that the stability of a breather traveling
to the right restricted to the solitary region by energy method is analyzed
in Alejo-Mu\~noz \cite{AM1,AM2}. The stability of a breather traveling to the right via our approach
is given in Corollary \ref{cor:stabN} as a special case.
\begin{corollary}
	\label{cor:stabreather}Let $u^{\left(br\right)}\left(x,t;z_{0}\right)$ be a breather
	with discrete scattering data $(z_{0},c_{0})$ such that $z_{0}=\xi_{0}+i\eta_{0},\ \xi_{0},\eta_{0}>0$ with $4\eta_{0}^{2}-12\xi_{0}^{2}<0$. Suppose $\left\Vert \mathrm{R}\left(0\right)\right\Vert _{H^{\frac{1}{4},1}\left(\mathbb{R}\right)}<\epsilon$
	for $0\leq\epsilon\ll1$ small enough, consider the solution $u$
	to the focusing mKdV \eqref{eq:fmkdv1} with the initial data
	\[
	u_{0}=u^{\left(br\right)}\left(x,0;z_{0},c_{0}\right)+\mathrm{R}\left(0\right).
	\]
	Then there exist $z_{1}=\xi_{1}+i\eta_{1}$ and the norming constant
	$c_{1}$  such that
	\begin{equation}
	\left|z_{1}-z_{0}\right|+\left|c_{1}-c_{0}\right|\lesssim\epsilon.\label{eq:param}
	\end{equation}
	Let $r\left(z\right)$ be the reflection coefficient computed from
	$u_{0}$. Then, as $t\to\infty$
	\[
	u=u^{\left(br\right)}\left(x,t;z_{1},c_{1}\right)+R\left(x,t\right)
	\]
	where the radiation term $R\left(x,t\right)$ has the following asymptotics:
	\begin{enumerate}
\item[(1).]	In the soliton region, i.e., Region $\text{III}$, we have
	\begin{equation}
		\left|R\left(x,t\right)\right|\lesssim\frac{1}{t}.\label{eq:solirad-1}
	\end{equation}
\item[(2).]	In the self-similar region, ( Region $\text{II}$ ), for $4<p<\infty$,
	one has
	\begin{equation}
		R\left(x,t\right)=\frac{1}{\left(3t\right)^{\frac{1}{3}}}P\left(\frac{x}{\left(3t\right)^{\frac{1}{3}}}\right)+\mathcal{O}\left(t^{\frac{2}{3p}-\frac{1}{2}}\right)\label{eq:self-1}
	\end{equation}
	where $P$ is a solution to the Painlev\'e II equation
	\[
	P''\left(s\right)-sP'\left(s\right)+2P^{3}\left(s\right)=0
	\]
	determined by $r\left(0\right)$.
	
\item[(3).]	In the oscillatory region (Region $\text{I}$ ), the asymptotics
	for $R\left(x,t\right)$ are more involved.
	\begin{itemize}
	\item[(i)] If we choose the frame
	$x=\mathrm{v} t$ with $\mathrm{v}=4\eta_{1}^{2}-12\xi_{1}^{2}<0$, one has
	\[
	R\left(x,t\right)=u_{as}\left(x,t\right)+\mathcal{O}\left(\left(z_{0}t\right)^{-\frac{3}{4}}\right)
	\]
	where
	\begin{align}
		u_{as}\left(x,t\right) & =\frac{1}{\sqrt{48tz_{0}}}\left(\left(m_{11}^{\left(br\right)}\left(-z_{0}\right)^{2}\left(i\delta_{A}^{0}\right)^{2}\overline{\beta}_{12}\right)+m_{12}^{\left(br\right)}\left(-z_{0}\right)^{2}\left(\left(i\delta_{A}^{0}\right)^{2}\overline{\beta}_{21}\right)\right)\nonumber \\
		& +\frac{1}{\sqrt{48tz_{0}}}\left(\left(m_{11}^{\left(br\right)}\left(z_{0}\right)^{2}\left(i\delta_{B}^{0}\right)^{2}\beta_{12}\right)-m_{12}^{\left(br\right)}\left(z_{0}\right)^{2}\left(\left(i\delta_{B}^{0}\right)^{2}\beta_{21}\right)\right)\label{eq:oscasymp1-1}
	\end{align}
	with some explicit constants $m_{11}^{\left(br\right)}\left(-z_{0}\right)$,
	$m_{12}^{\left(br\right)}\left(-z_{0}\right)$, $m_{11}^{\left(br\right)}\left(z_{0}\right)$,
	$m_{12}^{\left(br\right)}\left(z_{0}\right)$ from the breather matrix,
	\[
	\beta_{12}=\frac{\sqrt{2\pi}e^{i\pi/4}e^{-\pi\kappa}}{r\left(z_{0}\right)\Gamma\left(-i\kappa\right)},\ \beta_{21}=\frac{-\sqrt{2\pi}e^{-i\pi/4}e^{-\pi\kappa}}{\overline{r\left(z_{0}\right)}\Gamma\left(i\kappa\right)},
	\]
	\[
	\delta_{A}^{0}=\left(192\tau\right)^{i\kappa/2}e^{-8i\tau}e^{\chi\left(-z_{0}\right)},
	\]
	\[
	\delta_{B}^{0}=\left(192\tau\right)^{-i\kappa/2}e^{8i\tau}e^{\chi\left(z_{0}\right)}
	\]
\item[(ii)]	If $\mathrm{v}\neq4\eta_{1}^{2}-12\xi_{1}^{2}$, then we have
	\[
	R\left(x,t\right)=u_{as}\left(x,t\right)+\mathcal{O}\left(\left(z_{0}t\right)^{-\frac{3}{4}}\right)
	\]
	where
	\begin{equation}
		u_{as}\left(x,t\right)=\left(\frac{\kappa}{3tz_{0}}\right)^{\frac{1}{2}}\cos\left(16tz_{0}^{3}-\kappa\log\left(192tz_{0}^{3}\right)+\phi\left(z_{0}\right)\right)\label{eq:Oscasymp2-1}
	\end{equation}
	with
	\begin{align*}
		\phi\left(z_{0}\right) & =\arg\Gamma\left(i\kappa\right)-\frac{\pi}{4}-\arg r\left(z_{0}\right)+\frac{1}{\pi}\int_{-z_{0}}^{z_{0}}\log\left(\frac{1+\left|r\left(\zeta\right)\right|^{2}}{1+\left|r\left(z_{0}\right)\right|^{2}}\right)\frac{d\zeta}{\zeta-z_{0}}.
	\end{align*}
	\end{itemize}
\end{enumerate}
\end{corollary}

\begin{proof}
	The above results follow from Theorem \ref{thm:main2} and the Lipshitiz
	continuity of the direct scattering map see the Appendix A and Zhou \cite{Zhou98}.
\end{proof}
To conclude our stability discussion, one can also consider the
full asymptotic stability of a complicated radiationless nonlinear structure. To construct the reflectionless
solution, suppose we have the following discrete scattering data
\[
S_{D}=\left\{ \left\{ z_{0,k},c_{0,k}\right\} _{k=1}^{N_{1}},\left\{ z_{0,j},c_{0,j}\right\} _{j=1}^{N_{2}}\right\} \in\mathbb{C}^{2N_{1}}\oplus\mathbb{C}^{2N_{2}}.
\]
Assume that $z_{0,j}=\xi_{0,j}+i\eta_{0,j}$, $\xi_{0,j},\eta_{0,j}>0$
and $z_{0,k}=\zeta_{0,k}i$ and for some $1\leq\ell_{0}\leq N_{2}$,
one has
\[
4\eta_{0,1}^{2}-12\xi_{0,1}^{2}<\ldots<4\eta_{0,\ell_{0}}^{2}-12\xi_{0,\ell_{0}}^{2}<0<4\eta_{0,\ell_{0}+1}^{2}-12\xi_{0,\ell_{0}+1}^{2}<\ldots<4\eta_{0,N_{2}}^{2}-12\xi_{0,N_{2}}^{2}.
\]
Secondly, we list the eigenvalues of $\breve{a}\left(z\right)$ on
the upper-half imaginary axis as $z_{0,k}=\zeta_{0,k}i$,
\[
\zeta_{0,1}<\ldots<\zeta_{0,N_{1}}.
\]
Then one can construct a reflectionless solution $u_{N}\left(x,t\right)$
using
\[
S_{D}=\left\{ \left\{ z_{0,k},c_{0,k}\right\} _{k=1}^{N_{1}},\left\{ z_{0,j},c_{0,j}\right\} _{j=1}^{N_{2}}\right\} 
\]
as

\[
u_{N}\left(x,t\right)=\sum_{\ell=1}^{N_{2}}u_{\ell}^{\left(br\right)}\left(x,t;z_{0,\ell},c_{0,\ell}\right)+\sum_{\ell=1}^{N_{1}}u_{\ell}^{\left(so\right)}\left(x,t;z_{0,\ell},c_{0,\ell}\right).
\]
where
\[
u_{\ell}^{\left(br\right)}\left(x,t;z_{0,\ell}\right)=-4\frac{\eta_{0,\ell}}{\xi_{0,\ell}}\frac{\xi_{0,\ell}\cosh\left(\nu_{2}+\tilde{\omega}_{2}\right)\sin\left(\nu_{1}+\tilde{\omega}_{1}\right)+\eta_{0,\ell}\sinh\left(\nu_{2}+\tilde{\omega}_{2}\right)\cos\left(\nu_{1}+\tilde{\omega}_{1}\right)}{\cosh^{2}\left(\nu_{2}+\tilde{\omega}_{2}\right)+\left(\eta_{0,\ell}/\xi_{0,\ell}\right)^{2}\cos^{2}\left(\nu_{1}+\tilde{\omega}_{1}\right)}
\]
where
\[
\nu_{1}=2\xi_{0,\ell}\left(x+4\left(\xi_{0,\ell}^{2}-3\eta_{0,\ell}^{2}\right)t\right),
\]
\[
\nu_{2}=2\eta_{0,\ell}\left(x-4\left(\eta_{0,\ell}^{2}-3\xi_{0,\ell}^{2}\right)t\right),
\]
and
\[
\tan\left(\tilde{\omega}_{1}\right)=\frac{\tilde{B}\xi_{0,\ell}-\tilde{A}\eta_{0,\ell}}{\tilde{A}\xi_{0,\ell}+\tilde{B}\eta_{0,\ell}},\ e^{-\tilde{\omega}_{2}}=\left|\frac{\xi_{0,\ell}}{2\eta_{0,\ell}}\right|\sqrt{\frac{\tilde{A}^{2}+\tilde{B}^{2}}{\xi_{0,\ell}^{2}+\eta_{0,\ell}^{2}}}
\]
here $\tilde{A}$ and $\tilde{B}$ are given as
\[
\tilde{c}_{0,\ell}=c_{0,\ell}\psi\left(z_{0,\ell}\right)^{-2}=\tilde{A}+i\tilde{B}
\]
where the scalar function as before is given has
\[
\psi\left(z\right)=\left(\prod_{z_{0,k}\in\mathcal{B}_{\ell,s}}\dfrac{z-\overline{z_{0,k}}}{z-z_{0,k}}\right)\left(\prod_{z_{0,j}\in\mathcal{B}_{\ell,b}}\dfrac{z-\overline{z_{0, j}}}{z-z_{0, j}}\right)\left(\prod_{z_{0,j}\in\mathcal{B}_{\ell,b}}\dfrac{z+z_{0,j}}{z+\overline{z_{0, j}}}\right).
\]
We also define
\begin{align*}
	\mathcal{B}_{\ell} & =\left\{ z_{0,j}=\xi_{0,j}+i\eta_{0,j}:\ 4\eta_{0,j}^{2}-12\xi_{0,j}^{2}>4\eta_{0,\ell}^{2}-12\xi_{0,\ell}^{2}\right\} \bigcup\left\{ z_{0,k}=i\zeta_{0,k}:\ 4\zeta_{0,k}^{2}>4\eta_{0,\ell}^{2}-12\xi_{0,\ell}^{2}\right\} \\
	& =:\mathcal{B}_{\ell,b}\bigcup\mathcal{B}_{\ell,s}.
\end{align*}
For the soliton part,
\[
u_{\ell}^{\left(so\right)}\left(x,t;z_{0,\ell}\right)=2\zeta_{0,\ell}  \eps_{\pm, \ell}  \text{sech}\left(-2\zeta_{0,\ell}\left(x-4\zeta_{0,\ell}^{2}t\right)+\omega_{0,\ell}\right)
\]
where 
\[
\omega_{0,\ell}=\log\left(\dfrac {\left|c_{0,\ell}\right|}{2\zeta_{0,\ell}}\right)+2\sum_{z_{0,k}\in S_{\ell,s}}\log\left\vert\frac{z_{0,\ell}-z_{0,k}} {z_{0,\ell}-\overline{z_{0,k}}}\right\vert+2\sum_{z_{0,j}\in\mathcal{S}_{\ell,b}}\log\left\vert\frac{z_{0,\ell}-z_{0,j}} {z_{0,\ell}-\overline{z_{0,j}}}\right\vert+2\sum_{z_{0,j}\in\mathcal{S}_{\ell,b}}\log\left\vert\frac{z_{0,\ell}+\overline{z_{0,j}}}{z_{0,\ell}+z_{0,j}}\right\vert.
\]
We also define
\begin{align*}
	\mathcal{S}_{\ell} & =\left\{ z_{0,j}=\xi_{0,j}+i\eta_{0,j}:\ 4\eta_{0,j}^{2}-12\xi_{0,j}^{2}>4\zeta_{0,\ell}^{2}\right\} \bigcup\left\{ z_{0,k}=i\zeta_{0,k}:\ 4\zeta_{0,k}^{2}>4\zeta_{0,\ell}^{2}\right\} \\
	& =:\mathcal{S}_{\ell,b}\bigcup\mathcal{S}_{\ell,s}.
\end{align*}
Finally, we state a corollary regarding the full asymptotic stability of $u_{N}\left(x,t\right).$

\begin{corollary}
	\label{cor:stabN}Consider the reflectionless solution $u_{N}\left(x,t\right)$
	to the focusing mKdV \eqref{eq:fmkdv1}. Suppose $\left\Vert \mathrm{R}\left(0\right)\right\Vert _{H^{\frac{1}{4},1}\left(\mathbb{R}\right)}<\epsilon$
	for $0\leq\epsilon\ll1$ small enough, the consider the solution $u$
	to the focusing mKdV \eqref{eq:fmkdv1} with the initial data
	\[
	u_{0}=u_{N}\left(x,t\right)+\mathrm{R}\left(0\right),
	\]
	then there exist scattering data
	\[
	S=\left\{ r\left(z\right),\left\{ z_{1,k},c_{1,k}\right\} _{k=1}^{N_{1}},\left\{ z_{1,j},c_{1,j}\right\} _{j=1}^{N_{2}}\right\} \in H^{1,\frac{1}{4}}\left(\mathbb{R}\right)\oplus\mathbb{C}^{2N_{1}}\oplus\mathbb{C}^{2N_{2}}
	\]
	computed in terms of $u_{0}$ such that
	\[
\sum_{k=1}^{N_{1}}\left(\left|z_{0,k}-z_{1,k}\right|+\left|c_{0,k}-c_{1,k}\right|\right)+\sum_{j=1}^{N_{2}}\left(\left|z_{0,j}-z_{1,j}\right|+\left|c_{0,j}-c_{1,j}\right|\right)\lesssim\epsilon.
	\]
	Then with the scattering data $S$, one can write the solution $u$
	to the focusing mKdV with the intial data $u_{0}$ as
	\[
	u=\sum_{\ell=1}^{N_{2}}u_{\ell}^{\left(br\right)}\left(x,t;z_{1,\ell},c_{1,\ell}\right)+\sum_{\ell=1}^{N_{1}}u_{\ell}^{\left(so\right)}\left(x,t;z_{1,\ell},c_{1,\ell}\right)+R\left(x,t\right)
	\]
	where the radiation term $R\left(x,t\right)$ has the asymptotics in Theorem \ref{thm:main2} and Theorem \ref{main1}
	using the scattering data $S$.
\end{corollary}
\begin{remark}
	Notice that for $N_{2}=0$ and $N_{1}=1$, $u_{N}\left(x,t\right)$
	is simply a solitary wave and for $N_{2}=1$ and $N_{1}=0$, $u_{N}\left(x,t\right)$
	is a breather. Corollary \ref{cor:stabN} in particular gives the full asymptotic stability of soliton and breather. Also this corollary covers the full asymptotic stability of multi-soliton solution, multi-breather solution and the mixed structure of them.
\end{remark}

\subsection{Acknowledgement}
We thank Prof. Catherine Sulem for her detailed comments and helpful remarks.


\section{Conjugation}
\label{sec:prep}
Along a characteristic line $x=vt$ for $v<0$  we have the following signature table:
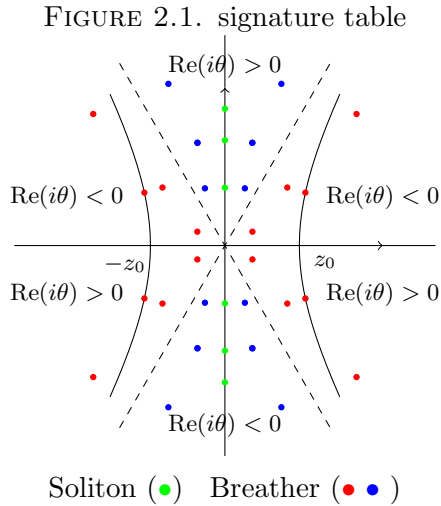
\begin{figure}[H]
\label{sig-table}
\caption{signature table}
\begin{tikzpicture}[scale=0.7]
\draw [->] (-4,0)--(3,0);
\draw (4,0)--(3,0);
\draw [->] (0,-4)--(0,3);
\draw (0,3)--(0,4);
 \pgfmathsetmacro{\a}{1}
    \pgfmathsetmacro{\b}{1.7320} 
     \draw[dashed] (-2 , -3.464)--(2 , 3.464);
   \draw[dashed] (-2 , 3.464)--(2 , -3.464);
    \pgfmathsetmacro{\c}{1.414}
    \pgfmathsetmacro{\d}{2.449} 
   \draw plot[red, domain=-1:1] ({\c*cosh(\x)},{\d*sinh(\x)});
    \draw plot[domain=-1:1] ({-\c*cosh(\x)},{\d*sinh(\x)});
    \pgfmathsetmacro{\e}{0.5}
    \pgfmathsetmacro{\f}{0.866} 
 \draw	[fill, green]  (0, 1.1)		circle[radius=0.05];	    
 \draw	[fill, green]  (0, 2)		circle[radius=0.05];	    
 \draw	[fill, green]  (0, 2.6)		circle[radius=0.05];	    
  \draw	[fill, green]  (0, -1.1)		circle[radius=0.05];	    
 \draw	[fill, green]  (0, -2)		circle[radius=0.05];	    
 \draw	[fill, green]  (0, -2.6)		circle[radius=0.05];	    
  \draw	[fill, blue]  (0.5211, 1.953)		circle[radius=0.05];	  
   \draw	[fill, blue]  (-0.5211, 1.953)		circle[radius=0.05];	   
    \draw[fill, blue]  (0.5211, 1.953)		circle[radius=0.05];	  
   \draw	[fill, blue]  (-0.5211, 1.953)		circle[radius=0.05];	   
    \draw[fill, blue]  (0.379, 1.087)		circle[radius=0.05];	  
   \draw	[fill, blue]  (-0.379, 1.087)		circle[radius=0.05];	   
    \draw[fill, blue]  (1.073, 3.074)		circle[radius=0.05];	  
   \draw	[fill, blue]  (-1.073, 3.074)		circle[radius=0.05];	   
   \draw	[fill, blue]  (0.5211, -1.953)		circle[radius=0.05];	  
   \draw	[fill, blue]  (-0.5211, -1.953)		circle[radius=0.05];	   
    \draw[fill, blue]  (0.5211, -1.953)		circle[radius=0.05];	  
   \draw	[fill, blue]  (-0.5211, -1.953)		circle[radius=0.05];	   
    \draw[fill, blue]  (0.379, -1.087)		circle[radius=0.05];	  
   \draw	[fill, blue]  (-0.379, -1.087)		circle[radius=0.05];	   
    \draw[fill, blue]  (1.073, -3.074)		circle[radius=0.05];	  
   \draw	[fill, blue]  (-1.073, -3.074)		circle[radius=0.05];	   
   \draw	[fill, red]  (1.1855, 1.10268)		circle[radius=0.05];	  
   \draw	[fill, red]  (-1.1855, 1.10268)	circle[radius=0.05];	   
    \draw	[fill, red]  (1.529, 1.006)		circle[radius=0.05];	  
   \draw	[fill, red]  (-1.529, 1.006)	circle[radius=0.05];	   
   \draw	[fill, red]  (0.523, 0.2637)		circle[radius=0.05];	  
   \draw	[fill, red]  (-0.523, 0.2637)	circle[radius=0.05];	   
    \draw	[fill, red]  (1.1855, -1.10268)		circle[radius=0.05];	  
   \draw	[fill, red]  (-1.1855, -1.10268)	circle[radius=0.05];	   
    \draw	[fill, red]  (1.529, -1.006)		circle[radius=0.05];	  
   \draw	[fill, red]  (-1.529, -1.006)	circle[radius=0.05];	   
   \draw	[fill, red]  (0.523, -0.2637)		circle[radius=0.05];	  
   \draw	[fill, red]  (-0.523, -0.2637)	circle[radius=0.05];	  
     \draw	[fill, red]  (2.5, 2.5)		circle[radius=0.05];	  
    \draw	[fill, red]  (2.5, -2.5)		circle[radius=0.05];	  
     \draw	[fill, red]  (-2.5, 2.5)		circle[radius=0.05];	  
      \draw	[fill, red]  (-2.5, -2.5)		circle[radius=0.05];	  
   \node [below] at (1.9,0) {\footnotesize $z_0$};
    \node [below] at (-1.9,0) {\footnotesize $-z_0$};
    \node[above]  at (-3, 0.5) {\footnotesize $\text{Re}(i\theta)<0$};
     \node[above]  at (3, 0.5) {\footnotesize $\text{Re}(i\theta)<0$};
      \node[below]  at (-3, -0.5) {\footnotesize $\text{Re}(i\theta)>0$};
     \node[below]  at (3, -0.5) {\footnotesize $\text{Re}(i\theta)>0$};
     \node[above]  at (0, 3) {\footnotesize $\text{Re}(i\theta)>0$};
     \node[below]  at (0, -3) {\footnotesize $\text{Re}(i\theta)<0$};
    \end{tikzpicture}
 \begin{center}
  \begin{tabular}{ccc}
Soliton ({\color{green} $\bullet$})	&	
Breather ({\color{red} $\bullet$} {\color{blue} $\bullet$} ) 
\end{tabular}
 \end{center}
\end{figure}
In the figure above, we have chosen 
$$v=\dfrac{x}{t}= 4\eta^2_\ell-12\xi^2_\ell$$
where $  \lbrace  z_j \rbrace_{j=1}^{N_2} \ni z_\ell=\xi_\ell+i\eta_\ell$ with $1\leq \ell\leq N_2$. Define the following sets:
\begin{equation}
\label{B-set}
\mathcal{B}_\ell=\lbrace z_j = \xi_j+i\eta_j: ~  4\eta^2_j-12\xi^2_j  >4\eta^2_\ell-12\xi^2_\ell\rbrace .
\end{equation}
and
\begin{equation}
\label{mathcal-z}
\mathcal{Z}_k=\lbrace z_k \rbrace_{k=1}^{N_1}, \quad \mathcal{Z}_j=\lbrace z_j \rbrace_{k=1}^{N_2}, \quad  \mathcal{Z}=\calZ_j\cup \calZ_k.
\end{equation}
Also define
\begin{equation}
\label{lambda}
\lambda=\text{min}\lbrace \text{min}_{z,z'\in \calZ} |z-z'|, \quad \text{dist}(\calZ, \bbR)   \rbrace.
\end{equation}
We observe that for all $z_k \in \mathcal{Z}_k$ and $z_j\in \mathcal{B}_\ell $, 
$$\text{Re}(i\theta(z_k))>0,\, \text{Re}(i\theta(z_j))>0 .$$
Then we introduce a new matrix-valued function
\begin{equation}
\label{m1}
m^{(1)}(z;x,t) = m(z;x,t) \delta(z)^{-\sigma_3} 
\end{equation}
where $\delta(z)$  solves 
the scalar RHP 
Problem \ref{prob:RH.delta} below:

\begin{problem}
\label{prob:RH.delta}
Given $\pm z_0 \in \bbR$ and $r \in H^{1}(\bbR)$, find a scalar function 
$\delta(z) = \delta(z; z_0)$, meromorphic for
$z \in \bbC \setminus [-z_0, z_0]$ with the following properties:
\begin{enumerate}
\item		$\delta(z) \rarr 1$ as $z \rarr \infty$,
\item		$\delta(z)$ has continuous boundary values $\delta_\pm(z) =\lim_{\eps \darr 0} \delta(z \pm i\eps)$ for $z \in (-z_0, z_0)$,
\item		$\delta_\pm$ obey the jump relation
			$$ \delta_+(z) = \begin{cases}
											\delta_-(z)  \left(1 + \left| r(z) \right|^2 \right),	&	 z\in (-z_0, z_0),\\
											\delta_-(z), &	z \in \bbR\setminus (-z_0, z_0),
										\end{cases}
			$$
\item		$\delta(z)$ has simple pole at $z_k$ for $k=1... N_1$ and at $z_j, -\overline{z_j}$ for $j\in \mathcal{B}_\ell$.
\end{enumerate}
\end{problem}

\begin{lemma}
\label{lemma:delta}
Suppose $r \in H^{1}(\bbR)$ and that $\kappa(s)$ is given by \begin{equation}
\kappa=-\frac{1}{2\pi}\log\left(1+\left|r\left(z_{0}\right)\right|^{2}\right), \label{kappa}
\end{equation}Then
\begin{itemize}
\item[(i)]		Problem \ref{prob:RH.delta} has the unique solution
\begin{equation}
\label{RH.delta.sol}
\delta(z) =\left( \prod_{k=1}^{N_1} \dfrac{z- \overline {z_k }}{ z-z_k}\right) \left(\prod_{z_j\in B_\ell} \dfrac{z-\overline{z_j}}{z-z_j}\right) \left(\prod_{z_j\in B_\ell} \dfrac{z+z_j}{z+\overline{z_j}}\right) \left( \dfrac{z-z_0}{z+z_0} \right)^{i\kappa} e^{\chi(z)}  
\end{equation}
where $\kappa$ is given by equation \eqref{kappa}
and
\begin{equation}
\label{chi}
\chi(z)=\dfrac{1}{2\pi i}\int_{-z_0}^{z_0}\log\left( \dfrac{1+|r(\zeta)|^2}{1+|r(z_0)|^2} \right)\dfrac{d\zeta}{\zeta-z}
\end{equation}
$$ \left( \dfrac{z-z_0}{z+z_0} \right)^{i\kappa}=\exp\left( i\kappa \left( \log\left\vert \dfrac{z-z_0}{z+z_0} \right\vert +i\arg(z-z_0)-i\arg(z+z_0) \right) \right).$$
Here we have chosen the branch of the logarithm with $-\pi  < \arg(z) < \pi$. 
\bigskip
\item[(ii)] For $z\in \bbC\setminus [-z_0, z_0]$
\begin{equation*}
\delta(z) =(\overline{\delta(\zbar)})^{-1}
\end{equation*}
\bigskip

\item[(iii)] As $z\to \infty$, 
$$\delta(z)=1+\dfrac{\delta_1}{z}+\mathcal{O} \left( z^{-2} \right). $$

\item[(iv)]Along any ray of the form $\pm z_0+ e^{i\phi}\bbR^+$ with $0<\phi<\pi$ or $\pi < \phi < 2\pi$, 
				
				$$ 
						 \left| \delta(z) - \left( \dfrac{z-z_0}{z+z_0} \right)^{i\kappa}  \delta_0(\pm z_0)\right| 
						 		\leq C_r
						 |z \mp z_0|^{1/2} $$
						 where
						 $$\delta_0(\pm z_0)=\left( \prod_{k=1}^{N_1} \dfrac{\pm z_0- \overline {z_k }}{ \pm z_0-z_k}\right) \left(\prod_{z_j\in B_\ell} \dfrac{\pm z_0-\overline{z_j}}{\pm z_0-z_j}\right)  \left(\prod_{z_j\in B_\ell} \dfrac{\pm z_0+z_j}{\pm z_0+\overline{z_j}}\right)e^{\chi(\pm z_0)} $$
and the implied constant depends on $r$ through its $H^{1}(\bbR)$-norm  
				and is independent of $\pm z_0\in \bbR$.
				\end{itemize}
\end{lemma}

\begin{proof}
The proofs of (i)-(ii) can be found in  \cite{DZ93}. For (iii), we use the fact that as $z\to\infty$
\begin{align*}
\dfrac{z- \overline{z_k}}{z-z_k} &=\dfrac{z-z_k+z_k-\overline{z_k} }{z-z_k}\\
      &=1+\dfrac{2i\Imag (z_k)}{z}+\mathcal{O} \left( z^{-2} \right)
\end{align*}
and
\begin{align*}
\exp \left( \dfrac{1}{2\pi i}\int_{-z_0}^{z_0} \dfrac{\log ( 1+|r(\zeta)|^2 )}{\zeta-z} {d\zeta} \right)=1-\dfrac{1}{2\pi i z} \int_{-z_0}^{z_0} {\log ( 1+|r(\zeta)|^2 )} {d\zeta}+\mathcal{O} \left( z^{-2} \right).
\end{align*}
 To establish (iv), we first note that 
$$ \left\vert \left( \dfrac{z-z_0}{z+z_0} \right)^{i\kappa}\right\vert \leq e^{\pi \kappa}.$$
To bound the difference $e^{\chi(z)}-e^{\chi(\pm z_0)}$, notice that
\begin{align*}
\left\vert e^{\chi(z)}-e^{\chi(\pm z_0)}\right\vert &\leq\left\vert e^{\chi(\pm z_0)}\right\vert 
\left\vert e^{\chi(z)-\chi(\pm z_0)}-1 \right\vert\\
     &\lesssim \left\vert \int_0^1 \dfrac{d}{ds} e^{s(  \chi(z)-\chi(\pm z_0) )ds } \right\vert\\
     &\lesssim  |z \mp z_0|^{1/2} \sup_{0\leq s\leq 1} \left\vert e^{s(  \chi(z)-\chi(\pm z_0) )}\right\vert\\
     &\lesssim |z \mp z_0|^{1/2}
    \end{align*}
    where the third inequality follows from \cite[Lemma 23]{BDT88}.
\end{proof}

It is straightforward to check that if $m(z;x,t)$ solves Problem \ref{RHP-1}, then the new matrix-valued function $m^{(1)}(z;x,t)=m(z;x,t)\delta(z)^{\sigma_3}$ is the solution to the  following RHP.  

\begin{problem}
\label{prob:mkdv.RHP1}
Given 
$$\mathcal{S}=\lbrace r(z),\lbrace z_k, c_k \rbrace_{k=1}^{N_1}, \lbrace z_j, c_j \rbrace_{j=1}^{N_2} \rbrace \subset H^{1}(\bbR) \oplus \mathbb{C}^{2 N_1} \oplus \mathbb{C}^{ 2 N_2}$$ 
and the augmented contour $\Sigma$ in Figure \ref{figure-Sigma}
and set
$$\dfrac{x}{t}=4\eta^2_\ell-12\xi^2_\ell$$
where $ \lbrace z_j\rbrace_{j=1}^{N_2} \ni z_\ell=\xi_\ell+i\eta_\ell$, 
find a matrix-valued function $m^{(1)}(z;x,t)$ on $\bbC \setminus\Sigma$ with the following properties:
\begin{enumerate}
\item		$m^{(1)}(z;x,t) \rarr I$ as $|z| \rarr \infty$,
\item		$m^{(1)}(z;x,t)$ is analytic for $z \in  \bbC \setminus\Sigma$
			with continuous boundary values
			$m^{(1)}_\pm(z;x,t)$.
\item		On $\bbR$, the jump relation $$m^{(1)}_+(z;x,t)=m^{(1)}_-(z;x,t)	
			e^{-i\theta\ad\sigma_3}v^{(1)}(z)$$
			holds,
			 where $$v^{(1)}(z) = \delta_-(z)^{\sigma_3} v(z) \delta_+(z)^{-\sigma_3}.$$
			 \noindent
			The jump matrix $e^{-i\theta\ad\sigma_3} v^{(1)} $ is factorized as 
			\begin{align}
			\label{mkdv.V1}
			e^{-i\theta\ad\sigma_3}v^{(1)}(z)	=
			\begin{cases}
					\Twomat{1}{0}{\dfrac{\delta_-^{-2}  r}{1+|r|^2}  e^{2i\theta}}{1}
					\Twomat{1}{\dfrac{\delta_+^2 \rbar}{1+|r|^2} e^{-2i\theta}}{0}{1},
						& z \in (-z_0, z_0),\\
						\\
						\Twomat{1}{\rbar \delta^2 e^{-2i\theta}}{0}{1}
					\Twomat{1}{0}{r \delta^{-2} e^{2i\theta}}{1},
						& z \in(-\infty, -z_0)\cup (z_0,\infty) .
			\end{cases}
			\end{align}
\item On $ \left( \bigcup_{k=1}^{N_1} \gamma_k  \right) \cup\left( \bigcup_{j=1}^{N_2} \gamma_j  \right) $, let $\delta(z)$ be the solution to Problem \ref{prob:RH.delta}  we have the following jump conditions
$m^{(1)}_+(z;x,t)=m^{(1)}_-(z;x,t)	
			e^{-i\theta\ad\sigma_3}v^{(1)}(z)$
			where
$$
e^{-i\theta\ad\sigma_3}v^{(1)}(z) = 	\begin{cases}
						\twomat{1}{\dfrac{(1/\delta)'(z_k)^{-2} }{c_k  (z-z_k)}e^{-2i\theta} }{0}{1}	&	z\in \gamma_k, \\
						\\
						\twomat{1}{0}{\dfrac{\delta'( \overline{z_k } )^{-2}}{\overline{c_k}  (z -\overline{z_k })} e^{2i\theta}}{1}
							&	z \in \gamma_k^*
					\end{cases}
$$
and for $z_j\in \mathcal{B}_\ell $
$$
e^{-i\theta\ad\sigma_3}v^{(1)}(z)= 	\begin{cases}
						\twomat{1}{\dfrac{(1/\delta)'(z_j)^{-2} }{c_j  (z-z_j)}e^{-2i\theta} }{0}{1}	&	z\in \gamma_j, \\
						\\
						\twomat{1}{0}{\dfrac{\delta'( \overline{z_j } )^{-2}}{\overline{c_j}  (z -\overline{z_j })} e^{2i\theta}}{1}
							&	z \in \gamma_j^*\\
							\\
							\twomat{1}{0}{-\dfrac{\delta'( -{z_j } )^{-2}}{ c_j  (z + z_j ) } e^{2i\theta}}{1}	&	z\in -\gamma_j, \\
						\\
						\twomat{1} {-\dfrac{  (1/\delta)'(  -\overline{z_j } )^{-2} }{ \overline{c_j}  (z +\overline{z_j } ) }  e^{-2i\theta} }{0}{1}
							&	z \in -\gamma_j^*
					\end{cases}
$$								
and for $z_j \in \lbrace z_j \rbrace_{j=1}^{N_2} \setminus \mathcal{B}_\ell $	
$$
e^{-i\theta\ad\sigma_3}v^{(1)}(z)= 	\begin{cases}
						\twomat{1}{0}{\dfrac{c_j \delta(z_j)^{-2} }{z-z_j} e^{2i\theta}}{1}	&	z\in \gamma_j, \\
						\\
						\twomat{1}{\dfrac{\overline{c_j} \delta(\overline{z_j})^2 }{z -\overline{z_j}} e^{-2i\theta} }{0}{1}
							&	z \in \gamma_j^*\\
							\\
							\twomat{1}{\dfrac{-c_j \, \delta(-{z_j})^{2} }{z + z_j} e^{-2i\theta}  }{0}{1}	&	z\in -\gamma_j, \\
						\\
						\twomat{1}{0}{\dfrac{-\overline{c_j} \, \delta(-\overline{z_j})^{-2}e^{2i\theta} }{z +\overline{z_j }}}{1}
							&	z \in -\gamma_j^*
					\end{cases}
$$					
\end{enumerate}
\end{problem}
\begin{remark}
\label{circles}
We set
\begin{equation}
\label{Gamma}
\Gamma=  \left( \bigcup_{k=1}^{N_1} \gamma_k  \right) \cup \left( \bigcup_{k=1}^{N_1} \gamma_k^*  \right) \cup \left( \bigcup_{j=1}^{N_2} \pm \gamma_j  \right)\cup \left( \bigcup_{j=1}^{N_2} \pm \gamma_j^*  \right).
\end{equation}
From the signature table Figure \ref{sig-table} and the triangularities of the jump matrices, we observe that along the characteristic line $x=\mathrm{v}t$ where $\mathrm{v}=4\eta^2_\ell-12\xi^2_\ell$,  by choosing the radius of each element of $\Gamma$ small enough, we have for $z\in \Gamma \setminus \left( \pm \gamma_\ell \cup \pm \gamma^*_\ell   \right)   $ 
$$ e^{-i\theta\ad\sigma_3}v^{(1)}(z) \lesssim e^{-ct},  \quad t\to \infty.$$
For technical purpose which will become clear later, we want that the radius of each element of $\Gamma$ less than $\lambda/3$ where $\lambda$ is given by \eqref{lambda}. Also we make each element of $\Gamma$ is invariant under Schwarz reflection.
\end{remark}

\begin{figure}
\caption{The Augmented Contour $\Sigma$}
\vspace{.5cm}
\label{figure-Sigma}
\begin{tikzpicture}[scale=0.75]
 \pgfmathsetmacro{\c}{1.414}
    \pgfmathsetmacro{\d}{2.449} 
   \draw plot[red, domain=-1:1] ({\c*cosh(\x)},{\d*sinh(\x)});
    \draw plot[domain=-1:1] ({-\c*cosh(\x)},{\d*sinh(\x)});
\draw[ thick] (0,0) -- (-3,0);
\draw[ thick] (-3,0) -- (-5,0);
\draw[thick,->,>=stealth] (0,0) -- (3,0);
\draw[ thick] (3,0) -- (5,0);
\node[above] at 		(2.5,0) {$+$};
\node[below] at 		(2.5,0) {$-$};
\node[right] at (3.5 , 2) {$\gamma_j$};
\node[right] at (3.5 , -2) {$\gamma_j^*$};
\node[left] at (-3.5 , 2) {$-\gamma_j^*$};
\node[left] at (-3.5 , -2) {$-\gamma_j$};
\draw[->,>=stealth] (-2.6,2) arc(360:0:0.4);
\draw[->,>=stealth] (3.4,2) arc(360:0:0.4);
\draw[->,>=stealth] (-2.6,-2) arc(0:360:0.4);
\draw[->,>=stealth] (3.4,-2) arc(0:360:0.4);
\draw [red, fill=red] (-3,2) circle [radius=0.05];
\draw [red, fill=red] (3,2) circle [radius=0.05];
\draw [red, fill=red] (-3,-2) circle [radius=0.05];
\draw [red, fill=red] (3,-2) circle [radius=0.05];
\node[right] at (5 , 0) {$\bbR$};
\draw [green, fill=green] (0, 1) circle [radius=0.05];
\draw[->,>=stealth] (0.3, 1) arc(360:0:0.3);
\draw [green, fill=green] (0, -1) circle [radius=0.05];
\draw[->,>=stealth] (0.3, -1) arc(0:360:0.3);
 \draw[dashed] (-2 , -3.464)--(2 , 3.464);
  \draw[dashed] (-2 , 3.464)--(2 , -3.464);
\draw [blue, fill=blue] (1,3) circle [radius=0.05];
\draw[->,>=stealth] (1.2, 3) arc(360:0:0.2);
\draw [blue, fill=blue] (-1,3) circle [radius=0.05];
\draw[->,>=stealth] (-0.8, 3) arc(360:0:0.2);
\draw [blue, fill=blue] (1,-3) circle [radius=0.05];
\draw[->,>=stealth] (1.2, -3) arc(0:360:0.2);
\draw [blue, fill=blue] (-1,-3) circle [radius=0.05];
\draw[->,>=stealth] (-0.8, -3) arc(0:360:0.2);
\draw  (0.5, 0) arc(0:60:0.5);
\node[right] at (0.5 , 0.2) {\footnotesize ${\pi}/{3}$};
\node[above] at 		(0, 1.3) {$\gamma_k$};
\node[below] at 		(0, -1.3) {$\gamma_k^*$};
 \node [below] at (1.9,0) {\footnotesize $z_0$};
    \node [below] at (-1.9,0) {\footnotesize $-z_0$};
     \draw	[fill, red]  (1.529, 1.006)		circle[radius=0.05];	  
     \draw[->,>=stealth] (1.729, 1.006) arc(360:0:0.2);
   \draw	[fill, red]  (-1.529, 1.006)	circle[radius=0.05];	  
    \draw[->,>=stealth] (-1.329, 1.006) arc(360:0:0.2);
     \draw	[fill, red]  (1.529, -1.006)		circle[radius=0.05];	  
     \draw[->,>=stealth] (1.729, -1.006) arc(0:360:0.2);
   \draw	[fill, red]  (-1.529, -1.006)	circle[radius=0.05];	  
   \draw[->,>=stealth] (-1.329, -1.006) arc(0:360:0.2); 
    \node[above]  at (-3, 0.5) {\footnotesize $\text{Re}(i\theta)<0$};
     \node[above]  at (3, 0.5) {\footnotesize $\text{Re}(i\theta)<0$};
      \node[below]  at (-3, -0.5) {\footnotesize $\text{Re}(i\theta)>0$};
     \node[below]  at (3, -0.5) {\footnotesize $\text{Re}(i\theta)>0$};
     \node[above]  at (0, 3) {\footnotesize $\text{Re}(i\theta)>0$};
     \node[below]  at (0, -3) {\footnotesize $\text{Re}(i\theta)<0$};
     \node[above] at 		(1.9, 1.1) {$\gamma_\ell$};
\node[above] at 		(-2.1, 1.1) {$-\gamma_\ell^*$};
\node[below] at 		(-2.1, -1.0) {$-\gamma_\ell$};
\node[below] at 		(2.1, -1.0) {$\gamma_\ell^*$};
\end{tikzpicture}
\begin{center}
\begin{tabular}{ccc}
Soliton ({\color{green} $\bullet$})	&	
Breather ({\color{red} $\bullet$} {\color{blue} $\bullet$} ) 
\end{tabular}
\end{center}
\end{figure}

\section{Contour deformation}
\label{sec:mixed}

We now perform contour deformation on Problem \ref{prob:mkdv.RHP1}, following the standard procedure outlined in \cite[Section 4]{LPS}.
Since the phase function \eqref{theta} has two critical points
at $\pm z_0$, our new contour is chosen to be
\begin{equation}
\label{new-contour}
\Sigma^{(2)} = \Sigma_1 \cup \Sigma_2 \cup \Sigma_3 \cup \Sigma_4\cup  \Sigma_5 \cup \Sigma_6 \cup \Sigma_7 \cup \Sigma_8
\end{equation}
shown in Figure \ref{new-contour} and consists of rays of the form $\pm z_0+ e^{i\phi}\bbR^+$
where $\phi = \pi/4, 3\pi/4,5\pi/4, 7\pi/4$. 

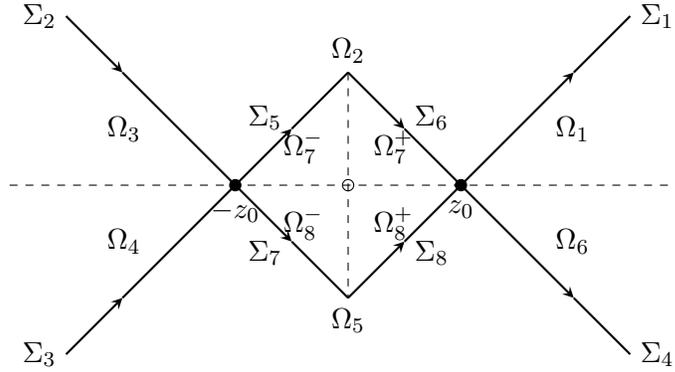
\begin{figure}[H]
\caption{Deformation from $\mathbb{R}$ to $\Sigma^{(2)}$}
\vskip 15pt
\begin{tikzpicture}[scale=0.75]
\draw[dashed] 					(0, 2) -- (0,-2);	
\draw[dashed] 					(-6,0) -- (6,0);								
\draw[->,thick,>=stealth] 			(2,0) -- (4,2);								
\draw	[thick] (5, 3) -- (4,2);
\draw [->,thick,>=stealth]  	(-5,3) -- (-4,2);							
\draw[thick]	 (-2,0) -- (-4,2);
\draw[->,thick,>=stealth]		(-5,-3) -- (-4,-2);							
\draw[thick]						(-4,-2) -- (-2,0);
\draw[thick,->,>=stealth]		(2,0) -- (4,-2);								
\draw[thick]						(4,-2) -- (5,-3);
\draw	  [thick,->,>=stealth]  (-2,0) -- (-1,1);								
\draw   [thick]  (0,2) -- (-1, 1);
\draw[thick,->,>=stealth]		(-2,0) -- (-1,-1);								
\draw[thick]						(-1,-1) -- (0, -2);
\draw[thick,->,>=stealth]		(0,2) -- (1,1);								
\draw [thick]	(2,0) -- (1, 1);
\draw[thick,->,>=stealth]		(0,-2) -- (1,-1);								
\draw[thick]						(1,-1) -- (2, -0);
\draw	[fill]							(-2,0)		circle[radius=0.1];	
\draw	[fill]							(2,0)		circle[radius=0.1];
\draw							(0,0)		circle[radius=0.1];
\node[below] at (-2,-0.1)			{$-z_0$};
\node[below] at (2,-0.1)			{$z_0$};
\node[right] at (5,3)					{$\Sigma_1$};
\node[left] at (-5,3)					{$\Sigma_2$};
\node[left] at (-5,-3)					{$\Sigma_3$};
\node[right] at (5,-3)				{$\Sigma_4$};
\node[left] at (-1,1.2)					{$\Sigma_5$};
\node[left] at (-1,-1.2)					{$\Sigma_7$};
\node[right] at (1,1.2)					{$\Sigma_6$};
\node[right] at (1,-1.2)					{$\Sigma_8$};
\node[right] at (3.5,1)				{$\Omega_1$};
\node[above] at (0,2)			{$\Omega_2$};
\node[left] at (-3.5,1)				{$\Omega_3$};
\node[left] at (-3.5,-1)				{$\Omega_4$};
\node[below] at (0,-2)			{$\Omega_5$};
\node[right] at (3.5,-1)				{$\Omega_6$};
\node[above] at (0.8, 0.2)			{$\Omega_7^+$};
\node[below] at (0.8,-0.2)				{$\Omega_8^+$};
\node[above] at (-0.8, 0.2)			{$\Omega_7^-$};
\node[below] at (-0.8,-0.2)				{$\Omega_8^-$};
\end{tikzpicture}
\label{fig:contour-2}
\end{figure}
For technical reasons (see Remark \ref{radius}), we define the following smooth cutoff function:
\begin{equation}
\label{chi-z}
\Xi_{\calZ}(z)=\begin{cases}
1 &\quad \text{dist}(z, \calZ\cup \calZ^*)\leq\lambda/3\\
0 &\quad \text{dist}(z, \calZ\cup \calZ^*)>2\lambda/3 .
\end{cases}
\end{equation}
Here recall that $\calZ$ is given by \eqref{mathcal-z} and $\lambda$ is defined in \eqref{lambda}.
We now introduce another matrix-valued function $m^{(2)}$:
$$ m^{(2)}(z) = m^{(1)}(z)  \calR^{(2)}(z). $$
Here $\calR^{(2)}$ will be chosen to remove the jump on the real axis and bring about new analytic jump matrices with the desired exponential decay 
along the contour $\Sigma^{(2)}$. Straight forward computation gives
\begin{align*}
m^{(2)}_+	&=m^{(1)}_+ \calR^{(2)}_+ \\
				&= m^{(1)}_- \left( e^{-i\theta\ad\sigma_3} v^{(1)} \right) \calR^{(2)}_+ \\
				&= m^{(2)}_- \left(\calR^{(2)}_-\right)^{-1}
						\left( e^{-i\theta\ad\sigma_3} v^{(1)} \right) \calR^{(2)}_+.
\end{align*}
We want to make sure that the following condition is satisfied
$$ 
(\calR^{(2)}_-)^{-1} \left( e^{-i\theta\ad\sigma_3} v^{(1)} \right) \calR^{(2)}_+ = I
$$
where $\calR_\pm^{(2)}$ are the boundary values of $\calR^{(2)}(z)$ as $\pm \Imag(z) \darr 0$. In this case the jump matrix associated to $m^{(2)}_\pm$ will be the identity matrix on $\bbR$ .

From the signature table \cite[Figure 0.1]{DZ93} we find that the function $e^{2i\theta}$ is exponentially decreasing on $\Sigma_3$  $\Sigma_4$, $\Sigma_5$, $\Sigma_6$ and increasing on $\Sigma_1$, $\Sigma_2$, $\Sigma_7$, $\Sigma_8$ while the reverse is true for $e^{-2i\theta}$. 
Letting
\begin{equation} \label{eta}
\eta(z; \pm z_0) = \left( \prod_{k=1}^{N_1} \dfrac{\pm z_0- \overline {z_k }}{ \pm z_0-z_k}\right) \left(\prod_{z_j\in B_\ell} \dfrac{\pm z_0-\overline{z_j}}{ \pm z_0-z_j}\right) \left(\prod_{z_j\in B_\ell} \dfrac{\pm z_0+z_j}{\pm z_0+\overline{z_j}}\right) \left( \dfrac{ z-z_0}{z+z_0} \right)^{i\kappa}
\end{equation}
\begin{equation} \label{eta-0}
\eta_0(\pm z_0) = \left( \prod_{k=1}^{N_1} \dfrac{\pm z_0- \overline {z_k }}{ \pm z_0-z_k}\right) \left(\prod_{z_j\in B_\ell} \dfrac{\pm z_0-\overline{z_j}}{ \pm z_0-z_j}\right) \left(\prod_{z_j\in B_\ell} \dfrac{\pm z_0+z_j}{\pm z_0+\overline{z_j}}\right) 
\end{equation}
and we define $\calR^{(2)}$ as follows (Figure \ref{fig R-2+}-\ref{fig R-2-}): 
 the functions $R_1$, $R_3$, $R_4$, $R_6$, $R_7^+$, $R_8^+$, $R_7^-$, $R_{8}^-$ satisfy 
\begin{align}
\label{R1}
R_1(z)	&=	\begin{cases}
							-{r(z)} \delta(z)^{-2}			
								&	z \in (z_0,\infty)\\[10pt]
						-{r(z_0 )} e^{-2\chi(z_0)} \eta(z; z_0)^{-2}(1-\Xi_\calZ)
								&	z	\in \Sigma_1,
					\end{cases}\\[10pt]
\label{R3}
R_3(z)	&=	\begin{cases}
						-{r(z)} \delta(z)^{-2}		
								&	z \in (-\infty, -z_0)\\[10pt]
						-{r(-z_0 )} e^{-2\chi(-z_0)} \eta(z; -z_0)^{-2} (1-\Xi_\calZ)
								&	z	\in \Sigma_2,
					\end{cases}\\[10pt]
\label{R4}
R_4(z)	&=	\begin{cases}
						\overline{r(z)} \delta(z)^{2}			
								&	z \in (-\infty, -z_0)\\[10pt]
						\overline{r(-z_0)} e^{2\chi(-z_0)} \eta(z; z_0)^{2} (1-\Xi_\calZ)
								&	z	\in \Sigma_3,
					\end{cases}\\[10pt]
\label{R6}
R_6(z)	&=	\begin{cases}
						\overline{r(z)} \delta(z)^{2}			
								&	z \in (z_0, \infty)\\[10pt]
						\overline{r(z_0)} e^{2\chi(z_0)} \eta(z; z_0)^{2} (1-\Xi_\calZ)
								&	z	\in \Sigma_4,
					\end{cases}
\end{align}

\begin{align}
\label{R7+}
R_7^+(z)	&=	\begin{cases}
						-\dfrac{\delta_+^{2}(z)  \overline{r(z)} }{1+ |r(z)|^2}		
								& z \in (-z_0, z_0)\\[10pt]
					{-\dfrac{e^{2\chi(z_0)} \eta(z; z_0)^{2} \overline{r(z_0)} }{1+ |r(z_0)|^2}} (1-\Xi_\calZ)
						& z \in \Sigma_6,
					\end{cases}
					\\[10pt]
\label{R8+}
R_8^+(z)	&=	\begin{cases}
						\dfrac{\delta_-^{-2}(z)r(z)}{1+|r(z)|^2}		
								& z \in (-z_0, z_0)\\[10pt]
					\dfrac{e^{-2\chi(z_0)} \eta(z; z_0)^{-2} r(z_0)}{1+|r(z_0)|^2} (1-\Xi_\calZ) \quad
                                  & z \in \Sigma_8,
					\end{cases}
					\\[10pt]
\label{R7-}
R_{7}^-(z)	&=	\begin{cases}
						-\dfrac{\delta_+^{2}(z)  \overline{r(z)} }{1+|r(z)|^2}		
								& z \in (-z_0, z_0)\\[10pt]
					{-\dfrac{e^{2\chi(-z_0)} \eta(z; z_0)^{2} \overline{r(-z_0)} }{1+|r(-z_0)|^2}} (1-\Xi_\calZ)
						& z \in \Sigma_5
						\end{cases}
						\\[10pt]
\label{R8-}
R_8^-(z)	&=	\begin{cases}
						\dfrac{\delta_-^{-2}(z)r(z)}{1+|r(z)|^2}		
								& z \in (-z_0, z_0)\\[10pt]
					\dfrac{e^{-2\chi(-z_0)} \eta(z; z_0)^{-2} r(-z_0)}{1+|r(-z_0)|^2}  (1-\Xi_\calZ)\quad
                                  & z \in \Sigma_7.
					\end{cases}
\end{align}

{
\SixMatrix{The Matrix  $\calR^{(2)}$ for Region I, near $z_0$}{fig R-2+}
	{\twomat{1}{0}{R_1 e^{2i\theta}}{1}}
	{\twomat{1}{R_7^+ e^{-2i\theta}}{0}{1}}
	{\twomat{1}{0}{R_8^+ e^{2i\theta}}{1}}
	{\twomat{1}{R_6 e^{-2i\theta}}{0}{1}}
}

{
\sixmatrix{The Matrix  $\calR^{(2)}$ for Region I, near $-z_0$}{fig R-2-}
	{\twomat{1}{R_{7}^- e^{-2i\theta}}{0}{1}}
	{\twomat{1}{0}{R_3e^{2i\theta}}{1}}
	{\twomat{1}{R_4e^{-2i\theta}}{0}{1}}
	{\twomat{1}{0} {R_{8}^- e^{2i\theta}}{1}}
}

Each $R_i(z)$ in $\Omega_i$ is constructed in such a way that the jump matrices on the contour and $\dbar R_i(z)$ enjoys the property of exponential decay as $t\to \infty$.
We formulate Problem \ref{prob:mkdv.RHP1} into a mixed RHP-$\dbar$ problem. In the following sections we will separate this mixed problem into a localized RHP and a pure $\dbar$ problem whose long-time contribution to the asymptotics of $u(x,t)$ is of higher order than the leading term.

The following lemma (\cite[Proposition 2.1]{DM08}) will be used in the error estimates of 
$\bar \partial$-problem in Section \ref{sec:dbar}.

We first denote the entries that appear in \eqref{R1}--\eqref{R8-} by
\begin{align*}
p_1(z)=p_3(z)	&=	r(z).	&
p_4(z)=p_6(z)	&=	- \overline{r(z)},&\\
p_{7^-}(z)=p_{7^+}(z)	&=- \dfrac{ \overline{r(z)}}{1 + |r(z)|^2},& p_{8^-}(z)=p_{8^+}(z)	&= \dfrac{r(z)}{1+|r(z)|^2}.
\end{align*}

\begin{lemma}
\label{lemma:dbar.Ri}
Suppose $r \in H^{1}(\bbR)$. There exist functions $R_i$ on $\Omega_i$, $i=1,3,4,6,7^\pm,8^\pm$ satisfying \eqref{R1}--\eqref{R8-}, so that
$$ 
|\dbar R_i(z)| \lesssim 
	 |p_i'(\Real(z))| + |z-\xi|^{-1/2}  +\dbar \left( \Xi_\calZ(z) \right) , 	
				z \in \Omega_i
			$$ 
where $\xi=\pm z_0$  and the implied constants are uniform for $r $ in a bounded subset of $H^{1}(\bbR)$.
\end{lemma}

\begin{proof}
We only prove the lemma for $R_1$. Define $f_1(z)$ on $\Omega_1$ by
$$ f_1(z) = p_1(z_0) e^{-2\chi(z_0)} \eta(z; z_0)^{-2} \delta(z)^{2} $$
and let
\begin{equation}
\label{interpol}
\ R_1(z) = \left( f_1(z) + \left[ p_1(\Real(z)) - f_1(z) \right] \mathcal{K}(\phi) \right) \delta(z)^{-2} 
(1-\Xi_\calZ)
\end{equation}
where $\phi = \arg (z-\xi)$ and $\mathcal{K}$ is a smooth function on $(0, \pi/4)$ with
$$ 
\mathcal{K}(\phi)=
	\begin{cases}
			1
			&	z\in [0, \pi/12], \\
			0
			&	z \in[\pi/6, \pi/4]
	\end{cases}
$$ 
 It is easy to see that $R_1$ as constructed has the boundary values \eqref{R1}.
Writing $z-z_0= \rho e^{i\phi}$, we have
$$ \dbar = \frac{1}{2}\left( \frac{\dee}{\dee x} + i \frac{\dee}{\dee y} \right)
			=	\frac{1}{2} e^{i\phi} \left( \frac{\dee}{\dee \rho} + \frac{i}{\rho} \frac{\dee}{\dee \phi} \right).
$$
We calculate
\begin{align*}
\dbar R_1 (z) & = \left( \frac{1}{2}  p_1'(\Real z) \mathcal{K}(\phi)  ~ \delta(z)^{-2} -
		\left[ p_1(\Real z) - f_1(z) \right]\delta(z)^{-2}  \frac{ie^{i\phi}}{|z-\xi|}  \mathcal{K}'(\phi) \right)\\
		 & \quad \times \left( 1-\Xi_\calZ  \right)-\left( f_1(z) + \left[ p_1(\Real(z)) - f_1(z) \right] \mathcal{K}(\phi) \right) \delta(z)^{-2} \dbar \left( \Xi_\calZ(z) \right) .
\end{align*}
Given that $\Xi(z)$ is infinitely smooth and compactly supported, it follows from Lemma \ref{lemma:delta} (iv) that
$$ 
 \left|\left( \dbar R_1 \right)(z)  \right| \lesssim
|p_1'(\Real z)| + |z-\xi|^{-1/2} +\dbar \left( \Xi_\calZ(z) \right) 
$$
where the implied constants depend on $\norm{r}{H^{1}}$ and the smooth function $\mathcal{K}$. 
The estimates in the remaining sectors are identical.
\end{proof}

The unknown $m^{(2)}$ satisfies a mixed $\dbar$-RHP. We first identify the jumps of $m^{(2)}$ along the contour $\Sigma^{(2)}$. Recall that $m^{(1)}$ is analytic along the contour,  the jumps are determined entirely by 
$\mathcal{R}^{(2)}$, see \eqref{R1}--\eqref{R8-}. Away from $\Sigma^{(2)}$, using the triangularity of $\mathcal{R}^{(2)}$, we   have that 
\begin{equation}
\label{N2.dbar}
 \dbar m^{(2)} = m^{(2)} \left( \calR^{(2)} \right)^{-1} \dbar \calR^{(2)} = m^{(2)} \dbar \calR^{(2)}. 
 \end{equation}
 
 \begin{remark}
 \label{radius}
By construction of $\mathcal{R}^{(2)}$ (see \eqref{R1}-\eqref{R8-} and \eqref{interpol}) and the choice of the radius of the circles in the set $ \Gamma$ (see Remark \ref{circles}), the right multiplication of $\mathcal{R}^{(2)}$ to $m^{(1)}$ will not change the jump conditions on circles in the set $ \Gamma$ . Thus over circles in the set $ \Gamma$, $m^{(2)}$ has the same jump matrices as given by (4) of Problem \ref{prob:mkdv.RHP1}.
\end{remark}

\begin{problem}
\label{prob:DNLS.RHP.dbar}
Given $r \in H^{1}(\bbR)$, find a matrix-valued function $m^{(2)}(z;x,t)$ on $\bbC \setminus \bbR$ with the following properties:
\begin{enumerate}
\item		$m^{(2)}(z;x,t) \rarr I$ as $ z \rarr \infty$ in $ \bbC \setminus \left( \Sigma^{(2)} \cup \Gamma \right),$
\item		$m^{(2)}(z;x,t)$ is continuous for $z \in  \bbC \setminus\left( \Sigma^{(2)} \cup \Gamma \right)$ 
			with continuous boundary values 
			$m^{(2)}_\pm(z;x,t) $
			(where $\pm$ is defined by the orientation in Figure \ref{fig:contour})
\item		The jump relation $m^{(2)}_+(z;x,t)=m^{(2)}_-(z;x,t)	
			e^{-i\theta\ad\sigma_3}v^{(2)}(z)$ holds, where
			$e^{-i\theta\ad\sigma_3}v^{(2)}(z)	$ is given in Figure \ref{fig:jumps-1}-\ref{fig:jumps-2} and part (4) of Problem \ref{prob:mkdv.RHP1}.
\item		The equation 
			$$
			\dbar m^{(2)} = m^{(2)} \, \dbar \calR^{(2)}
			$$ 
			holds in $\bbC \setminus \Sigma^{(2)}$, where
			$$
			\dbar \calR^{(2)}=
				\begin{doublecases}
					\Twomat{0}{0}{(\dbar R_1) e^{2i\theta}}{0}, 	& z \in \Omega_1	&&
					\Twomat{0}{(\dbar R_7^+)e^{-2i\theta}}{0}{0}	,	& z \in \Omega_7^+	\\
					\\
					\Twomat{0}{0}{(\dbar R_8^+)e^{2i\theta}}{0},	&	z \in \Omega_8^+	&&
					\Twomat{0}{(\dbar R_6)e^{-2i\theta}}{0}{0}	,	&	z	\in \Omega_6	 \\
					\\
					\Twomat{0}{0}{(\dbar R_3) e^{2i\theta}}{0}, 	& z \in \Omega_3	&&
					\Twomat{0}{(\dbar R_4)e^{-2i\theta}}{0}{0}	,	& z \in \Omega_4	\\
					\\
					\Twomat{0}{0}{(\dbar R_8^-)e^{2i\theta}}{0},	&	z \in \Omega_8^-	&&
					\Twomat{0}{(\dbar R_7^-)e^{-2i\theta}}{0}{0}	,	&	z	\in \Omega_7^- \\
					\\
					0	&\hspace{-5pt} z\in \Omega_2\cup\Omega_5	
				\end{doublecases}
			$$
\end{enumerate}
\end{problem}

The following picture is an illustration of the jump matrices of RHP Problem \ref{prob:DNLS.RHP.dbar}. For brevity we ignore the discrete scattering data.

\begin{figure}[H]
\caption{Jump Matrices  $v^{(2)}$  for $m^{(2)}$ near $z_0$}
\vskip 15pt
\begin{tikzpicture}[scale=0.7]
\draw[dashed] 				(-6,0) -- (6,0);							
\draw [->,thick,>=stealth] 	(0,0) -- (1.5,1.5);						
\draw  [thick]  (3,3) -- (1.5, 1.5);
\draw 	[->,thick,>=stealth]	  (-3,3) -- (-1.5,1.5);					
\draw	 [thick]	(0,0)--(-1.5,1.5);
\draw[->,thick,>=stealth]	(-3,-3) -- (-1.5,-1.5);					
\draw[thick]					(-1.5,-1.5) -- (0,0);
\draw[->,thick,>=stealth]	(0,0) -- (1.5,-1.5);					
\draw[thick]					(1.5,-1.5) -- (3,-3);
\draw[fill]						(0,0)	circle[radius=0.075];		
\node [below] at  			(0,-0.15)		{$z_0$};
\node[right] at					(3.2,3)		{$\unitlower{R_1 e^{2i\theta}}$};
\node[left] at					(-3.2,3)		{$\unitupper{R_7^+ e^{-2i\theta}}$};
\node[left] at					(-3.2,-3)		{$\unitlower{R_8^+ e^{2i\theta}}$};
\node[right] at					(3.2,-3)		{$\unitupper{R_6 e^{-2i\theta}}$};
\node[left] at					(2.5,3)		{$\Sigma_1$};
\node[right] at					(-2.5,3)		{$\Sigma_6$};
\node[right] at					(-2.5,-3)		{$\Sigma_8$};
\node[left] at					(2.5,-3)		{$\Sigma_4$};
\end{tikzpicture}
\label{fig:jumps-1}
\end{figure}

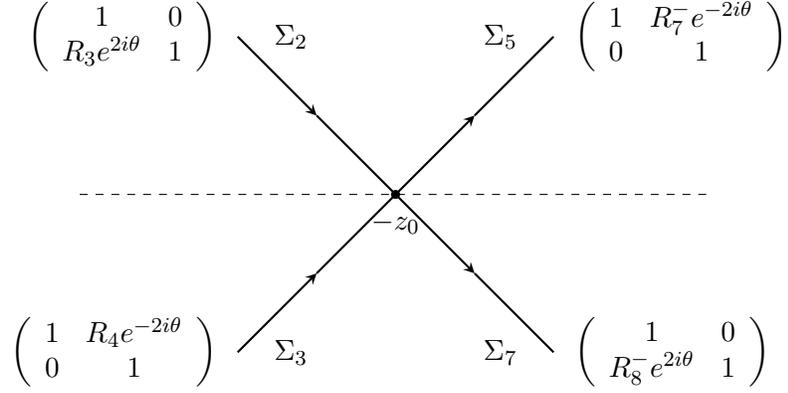
\begin{figure}[H]
\caption{Jump Matrices  $v^{(2)}$  for $m^{(2)}$ near $-z_0$}
\vskip 15pt
\begin{tikzpicture}[scale=0.7]
\draw[dashed] 				(-6,0) -- (6,0);							
\draw [->,thick,>=stealth] 	(0,0) -- (1.5,1.5);						
\draw  [thick]  (3,3) -- (1.5, 1.5);
\draw 	[->,thick,>=stealth]	  (-3,3) -- (-1.5,1.5);					
\draw	 [thick]	(0,0)--(-1.5,1.5);
\draw[->,thick,>=stealth]	(-3,-3) -- (-1.5,-1.5);					
\draw[thick]					(-1.5,-1.5) -- (0,0);
\draw[->,thick,>=stealth]	(0,0) -- (1.5,-1.5);					
\draw[thick]					(1.5,-1.5) -- (3,-3);
\draw[fill]						(0,0)	circle[radius=0.075];		
\node [below] at  			(0,-0.15)		{$-z_0$};
\node[right] at					(3.2,3)		{$\unitupper{R_7^- e^{-2i\theta}}$};
\node[left] at					(-3.2,3)		{$\unitlower{R_3 e^{2i\theta}}$};
\node[left] at					(-3.2,-3)		{$\unitupper{R_4 e^{-2i\theta}}$};
\node[right] at					(3.2,-3)		{$\unitlower{R_{8}^- e^{2i\theta}}$};
\node[left] at					(2.5,3)		{$\Sigma_5$};
\node[right] at					(-2.5,3)		{$\Sigma_2$};
\node[right] at					(-2.5,-3)		{$\Sigma_3$};
\node[left] at					(2.5,-3)		{$\Sigma_7$};
\end{tikzpicture}
\label{fig:jumps-2}
\end{figure}

%
%

\section{The Localized Riemann-Hilbert Problem}
\label{sec:local}
We perform the following factorization of $m^{(2)}$:
\begin{equation}
\label{factor-LC}
m^{(2)} = m^{(3)} m^{\RHP}.
\end{equation}
Here we require that $m^{(3)} $ to be the solution of the pure $\dbar$-problem, hence no jump, and $ m^{\RHP}$ solution of the localized  RHP Problem \ref{MKDV.RHP.local} below
with the jump matrix $v^\RHP=v^{(3)}$. The current section focuses on finding $ m^{\RHP}$.
\begin{problem}
\label{MKDV.RHP.local}
Find a $2\times 2$ matrix-valued function $m^\RHP(z; x,t)$, analytic on $\bbC \setminus \Sigma^{(3)}$,
with the following properties:
\begin{enumerate}
\item	$m^\RHP(z;x,t) \rarr I$ as $|z| \rarr \infty$ in $\bbC \setminus (\Sigma^{(3)}\cup\Gamma)$, where $I$ is the $2\times 2$ identity matrix,
\item	$m^\RHP(z; x,t)$ is analytic for $z \in \bbC \setminus  (\Sigma^{(3)}\cup\Gamma)$ with continuous boundary values $m^\RHP_\pm$
		on $\Sigma^{(3)}\cup\Gamma $,
\item	The jump relation $m^\RHP_+(z;x,t) = m^\RHP_-(z; x,t ) v^\RHP(z)$ holds on $\Sigma^{(3)}\cup\Gamma $, where
		\begin{equation*}	
		v^\RHP(z) =	v^{(3)}(z).
		\end{equation*}
\end{enumerate}
\end{problem}

\begin{remark}
Comparing the jump condition on $\Sigma^{(2)}$ and $\Sigma^{(3)}$, we note that the interpolation defined through \eqref{interpol} introduce new jump on $\Sigma^{(3)}_9$ with jump matrix given by 
\begin{equation}
v_9=\begin{cases} I, & z\in \left(-iz_0 \tan(\pi/12), iz_0 \tan(\pi/12) \right)\\
\\
        \unitupper{ (R_7^--R_7^+)e^{-2i\theta} },  & z\in \left(iz_0 \tan(\pi/12), iz_0 \right) \\
        \\
         \unitlower{ (R_8^--R_8^+)e^{2i\theta} },  & z\in \left(-iz_0 , -iz_0 \tan(\pi/12), \right). 
\end{cases}
\end{equation}
\end{remark}

\begin{figure}[H]
\caption{$\Sigma^{(3)}$}
\vskip 15pt
\begin{tikzpicture}[scale=0.65]
\draw	[->,thick,>=stealth] 	(2,0) -- (4,2);								
\draw	[thick]	(5, 3) -- (4,2);
\draw  [->,thick,>=stealth] 	(-5,3) -- (-4,2);							
\draw [thick]	(-2,0) -- (-4,2);
\draw[->,thick,>=stealth]		(-5,-3) -- (-4,-2);							
\draw[thick]						(-4,-2) -- (-2,0);
\draw[thick,->,>=stealth]		(2,0) -- (4,-2);								
\draw[thick]						(4,-2) -- (5,-3);
\draw[thick,->,>=stealth]	(-2,0) -- (-1,1);								
\draw  [thick]    (0,2) -- (-1, 1);
\draw[thick,->,>=stealth]		(-2,0) -- (-1,-1);								
\draw[thick]						(-1,-1) -- (0, -2);
\draw	[thick,->,>=stealth]		(0,2) -- (1,1);								
\draw	[thick]  (2,0) -- (1, 1);
\draw[thick,->,>=stealth]		(0,-2) -- (1,-1);								
\draw[thick]						(1,-1) -- (2, -0);
\draw	[fill]							(-2,0)		circle[radius=0.1];	
\draw	[fill]							(2,0)		circle[radius=0.1];
\draw[->,thick,>=stealth] 		(0, -2) -- (0,0);
\draw[thick]			(0,0) -- (0,2);		
\node[below] at (-2,-0.1)			{$-z_0$};
\node[below] at (2,-0.1)			{$z_0$};
\node[right] at (5,3)					{$\Sigma^{(3)}_1$};
\node[left] at (-5,3)					{$\Sigma^{(3)}_2$};
\node[left] at (-5,-3)					{$\Sigma^{(3)}_3$};
\node[right] at (5,-3)				{$\Sigma^{(3)}_4$};
\node[left] at (-1,1.2)					{$\Sigma^{(3)}_5$};
\node[left] at (-1,-1.2)					{$\Sigma^{(3)}_7$};
\node[right] at (1,1.2)					{$\Sigma^{(3)}_6$};
\node[right] at (1,-1.2)					{$\Sigma^{(3)}_8$};
\node[right] at (0,0)              {$\Sigma^{(3)}_9$};
\end{tikzpicture}
\label{fig:contour}
\end{figure}

\begin{figure}[H]
\caption{ $\Sigma^{(3)}\cup\Gamma$}
\vskip 15pt
\begin{tikzpicture}[scale=0.7]
 \draw[dashed] (-2 , -3.464)--(2 , 3.464);
   \draw[dashed] (-2 , 3.464)--(2 , -3.464);
\draw[thick]		(2,0) -- (2.5, 0.5);								
\draw[->,thick,>=stealth] [red]	 (2.5, 0.5)--(4, 2) ;
\draw [thick] [red]	 (4, 2)--(5,3) ;
\draw[thick] 		(-2,0) -- (-2.5, 0.5);					
\draw[->,thick,>=stealth]  [red]	  (-5,3) --  (-2.5, 0.5);	
\draw[thick] 		(-2,0) -- (-2.5, -0.5);					
\draw[->,thick,>=stealth]  [red]	  (-5,-3) --  (-2.5, -0.5);	
\draw[thick,->,>=stealth]		(2,0) -- (2.5,-0.5);								
\draw[thick]	[red]					(2.5,-0.5) -- (5,-3);
\draw [thick]	(-1.7, 0.3) --(-1.5, 0.5); 
\draw[thick,->,>=stealth] 	(-2, 0)--(-1.7, 0.3);						
\draw[thick] [red]	 (0,2) -- (-1.5, 0.5);
\draw[thick] 	(-1.5,- 0.5)--(-1.7, -0.3) ; 
\draw[thick,->,>=stealth] 	(-2, 0)--(-1.7, -0.3);						
\draw[thick] [red]	 (0, -2) -- (-1.5, -0.5);
\draw[thick]	[red]			(0,2) -- (1.5, 0.5);			
\draw	[thick]   (2,0) -- (1.8, 0.2);
\draw [thick,->,>=stealth]	(1.5, 0.5)-- (1.8, 0.2);
\draw[thick][red]		(0,-2) -- (1.5,-0.5);				
\draw[thick,->,>=stealth] 					(1.5, -0.5) -- (1.8, -0.2);
\draw[thick] (1.8, -0.2)--(2, 0);
\draw	[fill]							(-2,0)		circle[radius=0.1];	
\draw	[fill]							(2,0)		circle[radius=0.1];
\draw[->,thick,>=stealth] [red]		(0, -2) -- (0,0);
\draw[thick]	[red]		(0,0) -- (0,2);		
\node[below] at (-2,-0.1)			{$-z_0$};
\node[below] at (2,-0.1)			{$z_0$};
\node[right] at (5,3)					{$\Sigma^{(3)}_1$};
\node[left] at (-5,3)					{$\Sigma^{(3)}_2$};
\node[left] at (-5,-3)					{$\Sigma^{(3)}_3$};
\node[right] at (5,-3)				{$\Sigma^{(3)}_4$};
\node[left] at (-1,1.2)					{$\Sigma^{(3)}_5$};
\node[left] at (-1,-1.2)					{$\Sigma^{(3)}_7$};
\node[right] at (1,1.2)					{$\Sigma^{(3)}_6$};
\node[right] at (1,-1.2)					{$\Sigma^{(3)}_8$};
\node[right] at (0, 0)					{$\Sigma^{(3)}_9$};
 \pgfmathsetmacro{\c}{2}
    \pgfmathsetmacro{\d}{3.464} 
   \draw [blue] plot[domain=-1:1] ({\c*cosh(\x)},{\d*sinh(\x)});
    \draw  [blue]  plot[domain=-1:1] ({-\c*cosh(\x)},{\d*sinh(\x)});
    \draw [blue, fill=blue] (1,3) circle [radius=0.05];
\draw[->,>=stealth] [red] (1.2, 3) arc(360:0:0.2);
\draw [blue, fill=blue] (-1,3) circle [radius=0.05];
\draw[->,>=stealth] [red] (-0.8, 3) arc(360:0:0.2);
\draw [blue, fill=blue] (1,-3) circle [radius=0.05];
\draw[->,>=stealth] [red] (1.2, -3) arc(0:360:0.2);
\draw [blue, fill=blue] (-1,-3) circle [radius=0.05];
\draw[->,>=stealth] [red] (-0.8, -3) arc(0:360:0.2);
\draw [green, fill=green]  (0, 2.7) circle [radius=0.05];
\draw[->,>=stealth] [red] (0.3, 2.7) arc(360:0:0.3);
\draw [green, fill=green] (0, -2.7) circle [radius=0.05];
\draw[->,>=stealth] [red]  (0.3, -2.7) arc(0:360:0.3);
\draw[->,>=stealth] [red](-2.6,2) arc(360:0:0.4);
\draw[->,>=stealth] [red] (3.4,2) arc(360:0:0.4);
\draw[->,>=stealth] [red](-2.6,-2) arc(0:360:0.4);
\draw[->,>=stealth] [red] (3.4,-2) arc(0:360:0.4);
\draw [red, fill=red] (-3,2) circle [radius=0.05];
\draw [red, fill=red] (3,2) circle [radius=0.05];
\draw [red, fill=red] (-3,-2) circle [radius=0.05];
\draw [red, fill=red] (3,-2) circle [radius=0.05];
\draw [red, fill=red] (2.6749, 3.0764) circle [radius=0.05];
\draw[->,>=stealth]  (2.9749, 3.0764) arc(360:0:0.3);
\draw [red, fill=red] (-2.6749, 3.0764) circle [radius=0.05];
\draw[->,>=stealth]  (-2.3749, 3.0764) arc(360:0:0.3);
\draw [red, fill=red] (-2.6749, -3.0764) circle [radius=0.05];
\draw[->,>=stealth]  (-2.3749, -3.0764) arc(0:360:0.3);
\draw [red, fill=red] (2.6749, -3.0764) circle [radius=0.05];
\draw[->,>=stealth]  (2.9749, -3.0764) arc(0:360:0.3);
\end{tikzpicture}
\label{fig:contour-d-1}
\end{figure}

For some fixed $\eps>0$, we define 
\begin{align*}
L_\eps &=\lbrace z: z=z_0+u z_0 e^{3i\pi/4},  \eps\leq u\leq \sqrt{2} \rbrace\\
           &\cup   \lbrace z: z=z_0+u z_0 e^{i\pi/4},  \eps\leq u\leq +\infty \rbrace \\
           & \cup \lbrace z: z=-z_0+u z_0 e^{i\pi/4},  \eps\leq u\leq \sqrt{2} \rbrace\\
             &\cup   \lbrace z: z=-z_0+u z_0 e^{3i\pi/4},  \eps\leq u\leq +\infty \rbrace \\
           \Sigma'&= \left(\Sigma^{(3)}\setminus (L_\eps\cup {L_\eps^*}\cup \Sigma^{(3)}_9) \right) \cup \left( \pm \gamma_\ell  \right)  \cup \left( \pm \gamma^*_\ell\right) . 
\end{align*}

\begin{figure}[H]
\caption{ $\Sigma'$}
\vskip 15pt
\begin{tikzpicture}[scale=0.6]
 \draw[dashed] (-2 , -3.464)--(2 , 3.464);
   \draw[dashed] (-2 , 3.464)--(2 , -3.464);
\draw[thick]		(2,0) -- (2.5, 0.5);								

\draw[thick] 		(-2,0) -- (-2.5, 0.5);					

\draw[thick] 		(-2,0) -- (-2.5, -0.5);					

\draw[thick]		(2,0) -- (2.5,-0.5);								

\draw[thick] 	(-1.5, 0.5)--(-1.8, 0.2) ; 
\draw[thick]		(-1.8, 0.2)--(-2, 0);						

\draw[thick] 	(-1.5,- 0.5)--(-1.7, -0.3) ; 
\draw[thick] 	(-2, 0)--(-1.7, -0.3);						
			
\draw[thick]	(2,0) -- (1.7, 0.3);
\draw[thick]	(1.7, 0.3)--(1.5, 0.5);
				
\draw[thick] 					(1.5, -0.5) -- (1.8, -0.2);
\draw[thick] (1.8, -0.2)--(2, 0);
\draw	[fill]							(-2,0)		circle[radius=0.1];	
\draw	[fill]							(2,0)		circle[radius=0.1];
\node[below] at (-2,-0.1)			{$-z_0$};
\node[below] at (2,-0.1)			{$z_0$};
\pgfmathsetmacro{\c}{2}
    \pgfmathsetmacro{\d}{3.464} 
   \draw [blue] plot[domain=-1:1] ({\c*cosh(\x)},{\d*sinh(\x)});
    \draw  [blue]  plot[domain=-1:1] ({-\c*cosh(\x)},{\d*sinh(\x)});
\draw [red, fill=red] (2.6749, 3.0764) circle [radius=0.05];
\draw[->,>=stealth]  (2.9749, 3.0764) arc(360:0:0.3);
\draw [red, fill=red] (-2.6749, 3.0764) circle [radius=0.05];
\draw[->,>=stealth]  (-2.3749, 3.0764) arc(360:0:0.3);
\draw [red, fill=red] (-2.6749, -3.0764) circle [radius=0.05];
\draw[->,>=stealth]  (-2.3749, -3.0764) arc(0:360:0.3);
\draw [red, fill=red] (2.6749, -3.0764) circle [radius=0.05];
\draw[->,>=stealth]  (2.9749, -3.0764) arc(0:360:0.3);
\node[above]  at (-5, 0.5) {\footnotesize $\text{Re}(i\theta)<0$};
     \node[above]  at (5, 0.5) {\footnotesize $\text{Re}(i\theta)<0$};
      \node[below]  at (-5, -0.5) {\footnotesize $\text{Re}(i\theta)>0$};
     \node[below]  at (5, -0.5) {\footnotesize $\text{Re}(i\theta)>0$};
     \node[above]  at (0, 3) {\footnotesize $\text{Re}(i\theta)>0$};
     \node[below]  at (0, -3) {\footnotesize $\text{Re}(i\theta)<0$};
     \draw[ ] (0,0) -- (-3,0);
\draw[ ] (-3,0) -- (-5,0);
\draw[->,>=stealth] (0,0) -- (3,0);
 \draw[ ] (3,0) -- (5,0);
\end{tikzpicture}
\label{fig:sigma'}
\end{figure}

Here $\Sigma'$ is the black portion of the contour $\Sigma^{(3)}\cup\Gamma $ given in Figure \ref{fig:contour-d-1}. Now we decompose $w^{ ( 3 )}_\theta = v^{ ( 3 )}_\theta -I $ into two parts:
\begin{equation}
w^{ ( 3 )}_\theta=w^e+w'
\end{equation}
where $w'=w^{ ( 3 )}_\theta\restriction_{  \Sigma' }$ and $w^e=w^{ ( 3 )}_\theta\restriction_{   \left( \Sigma^{(3)}\cup\Gamma \right) \setminus\Sigma'   }$. 

Near $\pm z_0$, we write
$$i\theta(z; x, t)=4it \left( (z \mp z_0)^3\pm 3z_0 (z \mp z_0)^2 \pm 2z_0^3  \right)$$
and on $L_\eps$, away from $\pm z_0$, we estimate:
\begin{equation}
\label{R1-decay}
\left\vert R_1 e^{2i\theta} \right\vert \leq C_r e^{-24t z_0^3 u^2} \leq C_r e^{-24\eps^2\tau},
\end{equation}
\begin{equation}
\label{R3-decay}
\left\vert R_3 e^{2i\theta} \right\vert \leq C_r e^{-24t z_0^3 u^2} \leq C_r e^{-24\eps^2\tau},
\end{equation}
\begin{equation}
\label{R7-decay}
\left\vert R^\pm_7 e^{-2i\theta} \right\vert \leq C_r e^{-16t z_0^3 u^2} \leq C_r e^{-16\eps^2\tau}
\end{equation}
where the constant $C_r$ depends on the $H^1$ norm of $r$.
Similarly, on ${L_\eps^*}$
\begin{equation}
\label{R4-decay}
\left\vert R_4 e^{-2i\theta} \right\vert \leq C_r e^{-24t z_0^3 u^2} \leq C_r e^{-24\eps^2\tau},
\end{equation}
\begin{equation}
\label{R6-decay}
\left\vert R_6 e^{-2i\theta} \right\vert \leq C_r e^{-24t z_0^3 u^2} \leq C_r e^{-24\eps^2\tau},
\end{equation}
\begin{align}
\label{R8-decay}
\left\vert R^\pm_8 e^{2i\theta} \right\vert \leq C_r e^{-16t z_0^3 u^2} \leq C_r e^{-16\eps^2\tau}.
\end{align}
Also notice that
On $\Sigma^{(3)}_9$, by the construction of  $\mathcal{K}(\phi)$ and $v_9$, one obtains
\begin{equation}
\label{R9-decay}
\left\vert v_9 -I \right\vert \lesssim e^{-ct}.
\end{equation}
Combining Remark \ref{circles} with the discussion above we conclude that 
\begin{equation}
\label{expo}
|w^e|\lesssim e^{-ct}
\end{equation}

\begin{proposition}
\label{Prop:expo}
There exists a $2\times 2$ matrix $E_1(x,t; z)$ with 
$$E_1(x,t; z)=I+\mathcal{O}\left(  \dfrac{ e^{-ct}}{z} \right),$$
such that 
\begin{equation}
\label{E_1}
m^\RHP(x,t;z)=E_1(x,t; z)m^\RHP_*(x,t;z)
\end{equation}
where $m^\RHP_*(x,t;z)$ solves the RHP with jump contour $\Sigma'$ given in Figure \ref{fig:sigma'} and jump matrices
$$v'=I+w'.$$ 
\end{proposition}
\begin{proof}
We will later show the existence of $m^\RHP_*(x,t;z)$ and $ \norm{m^\RHP_*(x,t;z)}{L^\infty}$ is finite. Assuming this, it is easy to see that on $(\Sigma^{(3)} \cup \Gamma)\setminus \Sigma'$, $E_1$ satisfies the following jump condition:
$$E_{1+}=E_{1-} \left( m^\RHP_* (1+w^{e}) \left( m^{\RHP}_* \right)^{-1} \right).$$
Using \eqref{expo} the conclusion follows from solving a small norm Riemann-Hilbert problem (see the solution to Problem \ref{prob: E-2} for detail). 
\end{proof}
\subsection{Construction of parametrix}
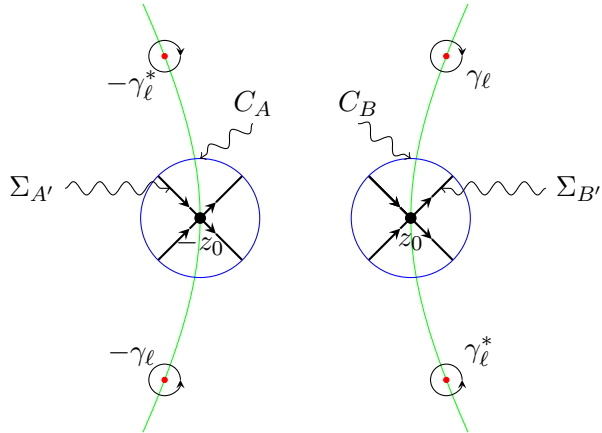
\begin{figure}[H]
\caption{ $\Sigma'=\Sigma_{A'}\cup \Sigma_{B'}\cup \pm \gamma_\ell \cup \gamma^*_\ell$}
\vskip 15pt
\begin{tikzpicture}[scale=0.7, photon/.style={decorate,decoration={snake,post length=0.8mm}} ]
\pgfmathsetmacro{\c}{2}
    \pgfmathsetmacro{\d}{3.464} 
   \draw [green] plot[domain=-1:1] ({\c*cosh(\x)},{\d*sinh(\x)});
    \draw  [green]  plot[domain=-1:1] ({-\c*cosh(\x)},{\d*sinh(\x)});
\draw[thick ,->,>=stealth]		(2,0) -- (2.4,0.4);								
\draw[thick] (2.4, 0.4)--(2.8, 0.8);
\draw[thick,->,>=stealth] 					(1.2, 0.8) -- (1.8, 0.2);
\draw[thick] (1.8, 0.2)--(2, 0);
\draw[thick ,->,>=stealth]		(2,0) -- (2.4, -0.4);							
\draw[thick] (2.4, -0.4)--(2.8, -0.8);
\draw[thick] 	(-1.2, 0.8)--(-1.7, 0.3) ; 
\draw[thick,->,>=stealth] 	(-2, 0)--(-1.7, 0.3);						
\draw[thick] 	(-1.2,- 0.8)--(-1.7, -0.3) ; 
\draw[thick,->,>=stealth] 	(-2, 0)--(-1.7, -0.3);						
\draw[thick,->,>=stealth] 					(1.2, 0.8) -- (1.8, 0.2);
\draw[thick] (1.8, 0.2)--(2, 0);
\draw[thick,->,>=stealth] 					(1.2, -0.8) -- (1.8, -0.2);
\draw[thick] (1.8, -0.2)--(2, 0);
\draw[thick,->,>=stealth] 					(-2.8, 0.8) -- (-2.2, 0.2);
\draw[thick] (-2.2, 0.2)--(-2, 0);
\draw[thick,->,>=stealth] 					(-2.8, -0.8) -- (-2.2, -0.2);
\draw[thick] (-2.2, -0.2)--(-2, 0);
\draw	[fill]						(-2,0)		circle[radius=0.1];	
\draw	[fill]					(2,0)		circle[radius=0.1];
\draw		[blue]				(-2,0)		circle[radius=1.131];	
\draw		[blue]					(2,0)		circle[radius=1.131];
\node[below] at (-2,-0.1)			{$-z_0$};
\node[below] at (2,-0.1)			{$z_0$};
\node[left] at (-4.56, 0.565)					{$\Sigma_{A'}$};
\node[right] at (4.56, 0.565)				{$\Sigma_{B'}$};
\draw [red, fill=red] (2.6749, 3.0764) circle [radius=0.05];
\draw[->,>=stealth]  (2.9749, 3.0764) arc(360:0:0.3);
\draw [red, fill=red] (-2.6749, 3.0764) circle [radius=0.05];
\draw[->,>=stealth]  (-2.3749, 3.0764) arc(360:0:0.3);
\draw [red, fill=red] (-2.6749, -3.0764) circle [radius=0.05];
\draw[->,>=stealth]  (-2.3749, -3.0764) arc(0:360:0.3);
\draw [red, fill=red] (2.6749, -3.0764) circle [radius=0.05];
\draw[->,>=stealth]  (2.9749, -3.0764) arc(0:360:0.3);
\draw[->,photon] ( 4.565,0.565) --  (2.565, 0.565); 
\draw[->,photon] ( -4.565,0.565) --  (-2.565, 0.565); 
\draw[->,photon] ( 1, 1.8) --  (2, 1.131); 
\draw[->,photon] ( -1, 1.8) --  (-2, 1.131); 
\node [above] at  ( 1, 1.75)  {$C_B$};
\node [above] at  ( -1, 1.75)  {$C_A$};
\node [below] at (3.3, 3) {$\gamma_\ell$};
\node [below] at (-3.3, 3) {$-\gamma_\ell^*$};
\node [above] at (-3.3, -3) {$-\gamma_\ell$};
\node [above] at (3.3, -3) {$\gamma_\ell^*$};
\end{tikzpicture}
\label{fig:contour-2}
\end{figure}
In this subsection we construct $m^\RHP_*$ needed in the proof of Proposition \ref{Prop:expo}. To achieve this, we need the solutions of the following three exactly solvable RHPs:
\begin{problem}
\label{prob:mkdv.br}
Find a matrix-valued function $m^{(br)}(z;x,t)$ on $\bbC \setminus\Sigma$ with the following properties:
\begin{enumerate}
\item		$m^{(br)}(z;x,t) \rarr I$ as $|z| \rarr \infty$,
\item		$m^{(br)}(z;x,t)$ is analytic for $z \in  \bbC \setminus ( \pm \gamma_\ell \cup \pm\gamma^*_\ell ) $
			with continuous boundary values
			$m^{(br)}_\pm(z;x,t)$.
\item On $ \pm \gamma_\ell \cup \pm\gamma^*_\ell$, let $\delta(z)$ be the solution to Problem \ref{prob:RH.delta} and we have the following jump conditions
$m^{(br)}_+(z;x,t)=m^{(br)}_-(z;x,t)	
			e^{-i\theta\ad\sigma_3}v^{(br)}(z)$
			where
$$
e^{-i\theta\ad\sigma_3}v^{(br)}(z)= 	\begin{cases}
						\twomat{1}{0}{\dfrac{c_\ell \delta(z_\ell )^{-2} }{z-z_\ell} e^{2i\theta}}{1}	&	z\in \gamma_\ell, \\
						\\
						\twomat{1}{\dfrac{\overline{c_\ell} \delta(\overline{z_\ell })^2 }{z -\overline{z_\ell }} e^{-2i\theta} }{0}{1}
							&	z \in \gamma_\ell^*\\
							\\
							\twomat{1}{\dfrac{-c_\ell \, \delta(-{z_\ell })^{2} }{z + z_\ell } e^{-2i\theta}  }{0}{1}	&	z\in -\gamma_\ell, \\
						\\
						\twomat{1}{0}{\dfrac{-\overline{c_\ell} \, \delta(-\overline{z_\ell })^{-2}e^{2i\theta} }{z +\overline{z_\ell }}}{1}
							&	z \in -\gamma_\ell^* .
					\end{cases}
$$					
\end{enumerate}
\end{problem}

\begin{problem}
\label{prob:mkdv.A}
Find a matrix-valued function $m^{A'}(z;x,t)$ on $\bbC \setminus\Sigma_A'$ with the following properties:
\begin{enumerate}
\item		$m^{A'}(z;x,t) \rarr I$ as $ z \rarr \infty$.
\item		$m^{A' }(z;x,t)$ is analytic for $z \in  \bbC \setminus \Sigma_A' $
			with continuous boundary values
			$m^{A'}_\pm(z;x,t)$.
\item On $ \Sigma_A'$ we have the following jump conditions
$$m^{A'}_+(z;x,t)=m^{A'}_-(z;x,t)	
			e^{-i\theta\ad\sigma_3}v^{A'}(z)$$
			where $v^{A'}=v^{(2)}\restriction _{\Sigma_A' }$.
\end{enumerate}
\end{problem}			

\begin{problem}
\label{prob:mkdv.B}
Find a matrix-valued function $m^{B'}(z;x,t)$ on $\bbC \setminus\Sigma_B'$ with the following properties:
\begin{enumerate}
\item		$m^{B'}(z;x,t) \rarr I$ as $ z \rarr \infty$.
\item		$m^{B'}(z;x,t)$ is analytic for $z \in  \bbC \setminus \Sigma_B' $
			with continuous boundary values
			$m^{B’}_\pm(z;x,t)$.
\item On $ \Sigma_B'$ we have the following jump conditions
$$m^{B'}_+(z;x,t)=m^{B’}_-(z;x,t)	
			e^{-i\theta\ad\sigma_3}v^{B'}(z)$$
			where $v^{B’}=v^{(2)}\restriction_{\Sigma_B' }$.
\end{enumerate}
\end{problem}		
We first study the solution to Problem \ref{prob:mkdv.br}. Since this problem consists of only discrete data, \eqref{BC-int} reduces to a linear system. More explicitly, we have a closed system:
\begin{align}
\label{BC-int-br}
\twomat{ \mu_{11}(\overline{z_l} ) }{\mu_{12}( {z_l} )}{\mu_{21}(\overline{z_l} )}{\mu_{22}( {z_l} )} &= I +\Twomat { \dfrac{\mu_{12}(z_l )c_l \delta(z_\ell )^{-2} e^{2i\theta(z_l )} }{\overline{z_l} -z_l }  }
        { -\dfrac{\mu_{11}( \overline{z_l }) {\overline{c_l }}  \delta(\overline{z_\ell })^2 e^{-2i\theta(   \overline{z_l }  )} }{z_l- \overline{z_l }}   }
        {\dfrac{\mu_{22}(z_l ) c_l \delta(z_\ell )^{-2} e^{2i\theta(z_l )} }{ \overline{z_l} -z_l }  }
        {- \dfrac{\mu_{21}( \overline{z_l }) {\overline{c_l }}  \delta(\overline{z_\ell })^2 e^{-2i\theta(   \overline{z_l }  )} }{z_l- \overline{z_l }} }\\
         \nonumber
         &\qquad +\Twomat
        { -\dfrac {\mu_{12}( -\overline{z_l }) {\overline{c_l }} \delta(-\overline{z_\ell })^{-2} e^{2i\theta(  - \overline{z_l }  )} }{\overline{z_l} + \overline{z_l }}   }
        {\dfrac{\mu_{11}(-z_l )c_l \delta(-{z_\ell })^{2}  e^{-2i\theta(-z_j)} }{z_l+z_l }  }
       {- \dfrac{\mu_{22}( -\overline{z_l }) {\overline{c_l }}  \delta(-\overline{z_\ell })^{-2} e^{2i\theta(  - \overline{z_l }  )} }{\overline{z_l}  + \overline{z_l }} }
        {\dfrac{\mu_{21}(-z_l )c_l  \delta(-{z_\ell })^{2}  e^{-2i\theta(-z_l )} }{z_l+z_l }  },
\end{align}
\begin{align}
\label{BC-int-br-}
\twomat{ \mu_{11}( -{z_l} ) }{\mu_{12}( -\overline{z_l} )}{\mu_{21}(-{z_l} )}{\mu_{22}(-\overline {z_l} )} &= I +\Twomat { \dfrac{\mu_{12}(z_l )c_l  \delta(z_\ell )^{-2} e^{2i\theta(z_l )} }{-{z_l} -z_l }  }
        { -\dfrac{\mu_{11}( \overline{z_l }) {\overline{c_l }}  \delta(\overline{z_\ell })^2 e^{-2i\theta(   \overline{z_l }  )} }{- \overline{z_l } - \overline{z_l } }   }
        {\dfrac{\mu_{22}(z_l )c_l  \delta(z_\ell )^{-2} e^{2i\theta(z_l )} }{ -{z_l} -z_l }  }
        {- \dfrac{\mu_{21}( \overline{z_l }) {\overline{c_l }}  \delta(\overline{z_\ell })^2 e^{-2i\theta(   \overline{z_l }  )} }{- \overline{z_l } - \overline{z_l }} }\\
         \nonumber
         &\qquad +\Twomat
        { -\dfrac {\mu_{12}( -\overline{z_l }) {\overline{c_l }}  \delta(-\overline{z_\ell })^{-2} e^{2i\theta(  - \overline{z_l }  )} }{-{z_l} + \overline{z_l }}   }
        {\dfrac{\mu_{11}(-z_l )c_l \delta(-{z_\ell })^{2} e^{-2i\theta(-z_j)} }{-\overline{z_l} +z_l }  }
       {- \dfrac{\mu_{22}( -\overline{z_l }) {\overline{c_l }}  \delta(-\overline{z_\ell })^{-2} e^{2i\theta(  - \overline{z_l }  )} }{-{z_l}  + \overline{z_l }} }
        {\dfrac{\mu_{21}(-z_l )c_l \delta(-{z_\ell })^{2} e^{-2i\theta(-z_l )} }{-\overline{z_l} +z_l }  }.
\end{align}

Given that 
$$\delta(z)=\left(  \overline{ \delta (\overline{z}) } \right)^{-1},  $$
the Schwarz invariant condition of the jump matrices $ e^{-i\theta\ad\sigma_3}v^{(br)}(z) $ is satisfied and the solvability of this linear system \eqref{BC-int-br}-\eqref{BC-int-br-} follows. Moreover, we find the single breather solution:
\begin{align}
\nonumber
u^{(br)}(x,t) &=2 z \lim_{z\to \infty} m^{(br)}_{12}\\
\label{u-breather}
                   &=-4\dfrac{\eta_\ell}{\xi_\ell}\dfrac{ \xi_\ell \cosh(\nu_2+\tomega_2 ) \sin(\nu_1+\tomega_1 )+(\eta_\ell ) \sinh(\nu_2+\tomega_2  ) \cos(\nu_1+\tomega_1 )}{\cosh^2(\nu_2+\tomega_2  ) + (\eta_\ell/\xi_\ell)^2 \cos^2(\nu_1+\tomega_1 )}
\end{align}
with
\begin{align*}
\nu_1 &=2\xi_\ell(x+4(\xi^2_\ell -3\eta^2_\ell )t )\\
\nu_2 &=2\eta_\ell (x -4(\eta^2_\ell-3\xi^2_\ell )t ).
\end{align*}
And
\begin{align}
\label{tomega-1}
\tan \tomega_1&= \dfrac{\tB\xi_\ell-\tA\eta_\ell}{\tA\xi_\ell+\tB_\eta}\\
\label{tomega-2}
e^{-\tomega_2} &= \left\vert \dfrac{\xi_\ell }{ 2\eta_\ell } \right\vert \sqrt{\dfrac{\tA^2+\tB^2}{\xi^2_\ell+\eta^2_\ell}  }
\end{align}
where we set 
$$\tilde{c}_\ell=c_\ell \delta(z_\ell )^{-2}=\tA+i\tB .$$
We then study the solution to Problem \ref{prob:mkdv.A} and Problem \ref{prob:mkdv.B}. Extend the contours $\Sigma_{A'}$ and $\Sigma_{B'}$ to
\begin{subequations}
\begin{equation}
\widehat{\Sigma }_{A'}=\lbrace z=-z_0+z_0 u e^{\pm i\pi /4}: -\infty<u<\infty \rbrace,
\end{equation}
\begin{equation}
\widehat{\Sigma }_{B'}=\lbrace z=z_0+z_0 u e^{\pm i3\pi /4}: -\infty<u<\infty \rbrace
\end{equation}
\end{subequations}
respectively and define $\hat{v}^{A'}$, $\hat{v}^{B'}$ on $\widehat{\Sigma }_{A'}$, $\widehat{\Sigma }_{B'}$ through
\begin{subequations}
\begin{equation}
\hat{v}^{A‘} =\begin{cases}
v^{A'}(z), &\quad z\in \Sigma_{A'}\subset \widehat{\Sigma}_{A'}, \\
       0,   &\quad z\in \widehat{\Sigma}_{A'}\setminus {\Sigma}_{A'},
\end{cases}
\end{equation}
\begin{equation}
\hat{v}^{B'} =\begin{cases}
v^{B'}(z), &\quad z\in \Sigma_{B'}\subset \widehat{\Sigma}_{B'} \\
       0,   &\quad z\in \widehat{\Sigma}_{B'}\setminus {\Sigma}_{B'}.
\end{cases}
\end{equation}
\end{subequations}
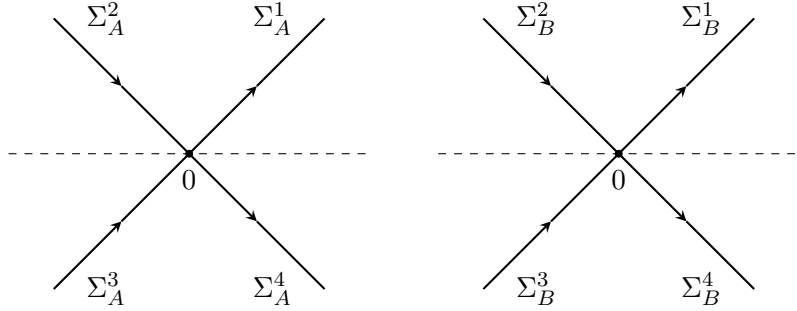
\begin{figure}[H]
\caption{$\Sigma_A,\Sigma_B$}
\vskip 15pt
\begin{tikzpicture}[scale=0.6]
\draw[dashed] 				(-4,0) -- (4,0);							
\draw [->,thick,>=stealth] 	(0,0) -- (1.5,1.5);						
\draw  [thick]  (3,3) -- (1.5, 1.5);
\draw 	[->,thick,>=stealth]	  (-3,3) -- (-1.5,1.5);					
\draw	 [thick]	(0,0)--(-1.5,1.5);
\draw[->,thick,>=stealth]	(-3,-3) -- (-1.5,-1.5);					
\draw[thick]					(-1.5,-1.5) -- (0,0);
\draw[->,thick,>=stealth]	(0,0) -- (1.5,-1.5);					
\draw[thick]					(1.5,-1.5) -- (3,-3);
\draw[fill]						(0,0)	circle[radius=0.075];		
\node [below] at  			(0,-0.15)		{$0$};
\node[left] at					(2.5,3)		{$\Sigma_A^1$};
\node[right] at					(-2.5,3)		{$\Sigma_A^2$};
\node[right] at					(-2.5,-3)		{$\Sigma_A^3$};
\node[left] at					(2.5,-3)		{$\Sigma_A^4$};
\end{tikzpicture}
\qquad
\begin{tikzpicture}[scale=0.6]
\draw[dashed] 				(-4,0) -- (4,0);							
\draw [->,thick,>=stealth] 	(0,0) -- (1.5,1.5);						
\draw  [thick]  (3,3) -- (1.5, 1.5);
\draw 	[->,thick,>=stealth]	  (-3,3) -- (-1.5,1.5);					
\draw	 [thick]	(0,0)--(-1.5,1.5);
\draw[->,thick,>=stealth]	(-3,-3) -- (-1.5,-1.5);					
\draw[thick]					(-1.5,-1.5) -- (0,0);
\draw[->,thick,>=stealth]	(0,0) -- (1.5,-1.5);					
\draw[thick]					(1.5,-1.5) -- (3,-3);
\draw[fill]						(0,0)	circle[radius=0.075];		
\node [below] at  			(0,-0.15)		{$0$};
\node[left] at					(2.5,3)		{$\Sigma_B^1$};
\node[right] at					(-2.5,3)		{$\Sigma_B^2$};
\node[right] at					(-2.5,-3)		{$\Sigma_B^3$};
\node[left] at					(2.5,-3)		{$\Sigma_B^4$};
\end{tikzpicture}
\label{fig:jumps-A-B}
\end{figure}
 Let $\Sigma_A$ and $\Sigma_B$ denote the contours
\begin{equation*}
\lbrace z=z_0 u e^{\pm i\pi/4} : -\infty<u<\infty \rbrace
\end{equation*}
with the same orientation as those of $\Sigma_{A'}$ and $\Sigma_{B'}$ respectively.  On 
$\widehat{\Sigma}_{A'}$ ($\widehat{\Sigma}_{B'}$) we carry out the following change of variable 
$$z\mapsto \zeta=\sqrt{48z_0 t}(z \pm z_0) $$
and introduce the scaling operators 
\begin{subequations}
\begin{equation}
\begin{cases}
N_A:  L^{2} ( \widehat{\Sigma}_{A'} ) \rightarrow L^2(\Sigma_A)\\
f(z)\mapsto (N_A f)(z)=f\left(\dfrac{\zeta}{\sqrt{48z_0 t}} -z_0  \right),
\end{cases}
\end{equation}
\begin{equation}
\begin{cases}
N_B:  L^{2} ( \widehat{\Sigma}_{B'} ) \rightarrow L^2(\Sigma_B)\\
f(z)\mapsto (N_B f)(z)=f\left(\dfrac{\zeta}{\sqrt{48z_0 t}} + z_0  \right).
\end{cases}
\end{equation}
\end{subequations}
We also define
\begin{equation}
\label{operator a-b}
1_A=1\restriction_{\Sigma_A}, \quad 1_B=1\restriction_{\Sigma_B}
\end{equation}
We first consider the case $\Sigma_B$. The rescaling gives 
\begin{equation*}
N_B \left( e^{\chi(z_0)} \eta(z; z_0)  e^{-it\theta}\right)=\delta^0_B\delta^1_B(\zeta)
\end{equation*}
with 
\begin{align*}
\delta^0_B &= (192\tau)^{-i\kappa/2} e^{8i\tau} e^{\chi(z_0)} \eta_0(z_0)\\
\delta^1_B(\zeta) &= \zeta^{i\kappa} \left( \dfrac{2 z_0}{\zeta /\sqrt{48t z_0} +2z_0}\right)^{i\kappa} e^{(-i\zeta^2/4)(1+\zeta(432\tau)^{-1/2} )}.
\end{align*}
Note that $\delta^0_B(z)$ is independent of $z$ and that $|\delta^0_B(z)|=1$. 
Set 
\begin{align*}
\Delta^0_B &=(\delta_B^0)^{\sigma_3}\\
w^B(\zeta)     &=(\Delta^0_B)^{-1}  (N_B \hat{w}^{B'})\Delta^0_B
\end{align*}
and define the operator $B: L^2(\Sigma_B) \to  L^2(\Sigma_B)$
\begin{align*}
B &=C_{ w^B}\\
  &= C^+ \left( \cdot (\Delta ^0_B)^{-1} (N_B \hat{w}^{B'}_- )  \Delta ^0_B \right)+ C^- \left( \cdot (\Delta ^0_B)^{-1} (N_B \hat{w}^{B'}_+)  \Delta ^0_B \right).
\end{align*}
On 
\begin{align*}
L_B\cup \overline{L}_B & =\lbrace z=u z_0 \sqrt{48t z_0}e^{i\pi/4}: -\eps<u< \eps \rbrace\\
       &\quad \cup \lbrace z=u z_0 \sqrt{48t z_0}e^{-i\pi/4}: -\eps<u<\eps \rbrace
\end{align*}
From the list of entries stated in \eqref{R1}, \eqref{R4}, \eqref{R7+} and \eqref{R8+}, we have
\begin{align}
\label{wB-1}
\left(  (\Delta_B^0)^{-1} \left(N_B \hat{w}^{B'}_- \right)  \Delta_B^0 \right)(\zeta) &= \twomat{0}{0}{r(z_0) \delta^1_B( \zeta )^{-2} }{0},\\
\label{wB-2}
\left(  (\Delta_B^0)^{-1} \left(N_B \hat{w}^{B'}_- \right)  \Delta_B^0 \right)(\zeta ) &= \twomat{0}{0}{
\dfrac{r(z_0)}{1 +  |r(z_0)|^2} \delta^1_B( \zeta )^{-2} }{0},\\
\label{wB-3}
\left(  (\Delta_B^0)^{-1} \left(N_B \hat{w}^{B'}_+ \right)  \Delta_B^0 \right)(\zeta) &= \twomat{0}{\overline{ r(z_0)} \delta^1_B(\zeta  )^{2} } {0}{0},\\
\label{wB-4}
\left(  (\Delta_B^0)^{-1} \left(N_B \hat{w}^{B'}_+ \right)  \Delta_B^0 \right)(\zeta) &= \twomat{0}
{\dfrac{ \overline{r(z_0)}}{1 +  |r(z_0)|^2} \delta^1_B(\zeta)^{2} }{0}{0}.
\end{align}
\begin{lemma}
\label{delta-B }
Let $\gamma$ be a small but fixed positive number with $0<2\gamma<1$. Then 
$$\left\vert \delta^1_B(\zeta)^{\pm 2}-\zeta^{\pm 2i\kappa} e^{\mp i \zeta^2/2} \right\vert \leq c |e^{\mp i\gamma \zeta^2/2} |\tau^{-1/2} $$
and as a consequence
\begin{equation}
\norm{\delta^1_B(\zeta)^{\pm 2}-\zeta^{\pm 2i\kappa} e^{\mp i \zeta^2/2} }{L^1\cap L^2 \cap L^\infty} \leq c \tau^{-1/2} 
\end{equation}
where the $\pm$ sign corresponds to $\zeta \in L_B$ and $\zeta\in \overline{L}_B$ respectively. Moreover, 
\begin{equation}
\label{zeta-tau}
\left\vert \zeta^{\pm 2i\kappa} e^{\mp i \zeta^2/2}\right\vert \lesssim \left\vert e^{\mp i\gamma z^2/2} \right\vert e^{-\eps^2(1-\gamma)24\tau}\lesssim  \left\vert e^{\mp i\gamma z^2/2} \right\vert  \tau^{-1/2}
\end{equation}
where the $\pm$ sign corresponds to $\zeta \in  (\Sigma_B^1\cup \Sigma_B^3)\setminus L_B$ and $ \zeta\in (\Sigma_B^2\cup \Sigma_B^4)\setminus \overline{L}_B$ respectively. 
\end{lemma}
\begin{proof}
We only deal with the $-$ sign. One can write
\begin{align*}
&\delta^1_B(\zeta)^{- 2}-\zeta^{- 2i\kappa} e^{ i \zeta^2/2}\\
 & \quad = e^{i\gamma \zeta^2/2}\left( e^{i\gamma \zeta^2/2} \left[  \left( \dfrac{2 z_0}{ \zeta /\sqrt{48t z_0}+2z_0} \right)^{-2i\kappa} \zeta^{-2i\kappa} e^{i(1-2\gamma)(\zeta^2/2)(1+\zeta/[(1-2\gamma)(432\tau)^{1/2}] )} \right. \right. \\
&\qquad \qquad \qquad \qquad \qquad \left. \left.  -\zeta^{-2i\kappa }  e^{i(1-2\gamma)\zeta^2/2}\right]   \right).
 \end{align*}
Each of the terms in the expression above is uniformly bounded for $x<0$ and $t>0$ ( \cite[p 334]{DZ93}). 
Following the proof of \cite[Lemma 3.35]{DZ93}, we estimate 
$$\left\vert e^{i\gamma \zeta^2/2} \left(  \left( \dfrac{2 z_0}{\zeta/\sqrt{48t z_0}+2z_0} \right)^{-2i\kappa} -1  \right) \right\vert \leq c |e^{i\gamma \zeta^2/2} |\tau^{-1/2}$$
and 
\begin{align*}
& \left\vert e^{i\gamma \zeta^2/2} \zeta^{-2i\kappa} \left( e^{i(1-2\gamma)(\zeta^2/2)(1+\zeta/[(1-2\gamma)(432\tau)^{1/2}] )}  -  e^{i(1-2\gamma)\zeta^2/2} \right) \right\vert \\
 &\quad \leq c |e^{i\gamma \zeta^2/2} |\tau^{-1/2} 
\end{align*} as desired. And the inequality in \eqref{zeta-tau} is an easy consequence of \eqref{R1-decay}-\eqref{R8-decay}.
\end{proof}
We then consider the case $\Sigma_A$. Again the rescaling gives 
\begin{equation*}
N_A \left( e^{\chi(z_0)} \eta(z; z_0)  e^{-it\theta}\right)=\delta^0_A \delta^1_A(\zeta)
\end{equation*}
with 
\begin{align*}
\delta^0_A &= (192\tau)^{i\kappa/2} e^{-8i\tau} e^{\chi(-z_0)}\eta_0(-z_0)\\
\delta^1_A(\zeta) &= (-\zeta)^{-i\kappa} \left( \dfrac{-2 z_0}{\zeta /\sqrt{48t z_0} -2z_0}\right)^{-i\kappa} e^{(i\zeta^2/4)(1-\zeta (432\tau)^{-1/2} )}.
\end{align*}
Note that $\delta^0_A$ is independent of $\zeta$ and that $|\delta^0_A|=1$.
Set
\begin{align*}
\Delta^0_A &=(\delta_A^0)^{\sigma_3}\\
w^A(\zeta)     &=(\Delta^0_A)^{-1}  (N_A \hat{w}^{A'})\Delta^0_A
\end{align*}
and define the operator $A: L^2(\Sigma_A) \to  L^2(\Sigma_A)$
\begin{align*}
A &=C_{ (\Delta^0_A)^{-1}  (N_A \hat{w}^{A'})\Delta^0_A }\\
  &= C^+ \left( \cdot (\Delta ^0_A)^{-1} (N_A \hat{w}^{A'}_- )  \Delta ^0_A \right)+ C^- \left( \cdot (\Delta ^0_A)^{-1} (N_A \hat{w}^{A'}_+)  \Delta ^0_A \right).
\end{align*}
On 
\begin{align*}
L_A\cup \overline{L}_A & =\lbrace z=u z_0 \sqrt{48t z_0}e^{-i3\pi/4}: -\eps<u<\eps \rbrace\\
       &\quad \cup \lbrace z=u z_0 \sqrt{48t z_0}e^{-i\pi/4}: -\eps<u<\eps \rbrace
\end{align*}
we have from the list of entries stated in \eqref{R1}, \eqref{R4}, \eqref{R7+} and \eqref{R8+}
\begin{align}
\left(  (\Delta_A^0)^{-1} \left(N_A \hat{w}^{A'}_- \right)  \Delta_A^0 \right)(z) &= \twomat{0}{0}{r(-z_0) \delta^1_A(z)^{-2} }{0},\\
\left(  (\Delta_A^0)^{-1} \left(N_A \hat{w}^{A'}_- \right)  \Delta_A^0 \right)(z ) &= \twomat{0}{0}{
\dfrac{r(-z_0)}{1 + |r(z_0)|^2} \delta^1_A(z)^{-2} }{0},\\
\left(  (\Delta_A^0)^{-1} \left(N_A \hat{w}^{A'}_+ \right)  \Delta_A^0 \right)(z) &= \twomat{0}{\overline{ r(-z_0)} \delta^1_A(z)^{2} } {0}{0},\\
\left(  (\Delta_A^0)^{-1} \left(N_A \hat{w}^{A'}_+ \right)  \Delta_A^0 \right)(z) &= \twomat{0}
{\dfrac{ \overline{r(-z_0)}}{1+ |r(z_0)|^2} \delta^1_A(z)^{2} }{0}{0}.
\end{align}
\begin{lemma}
\label{delta-A}
Let $\gamma$ be a small but fixed positive number with $0<2\gamma<1$. Then 
$$ \left\vert  \delta^1_A(\zeta)^{\pm 2}-(-\zeta)^{\mp 2i\kappa} e^{\pm i \zeta^2/2} \right\vert\leq  c |e^{\pm i\gamma \zeta^2/2} |\tau^{-1/2}  $$
and as a consequence,
\begin{equation}
\norm{ \delta^1_A(\zeta)^{\pm 2}-(-\zeta)^{\mp 2i\kappa} e^{\pm i\zeta^2/2}}{L^1\cap L^2 \cap L^\infty}  \leq c \tau^{-1/2} 
\end{equation}
where the $\pm$ sign corresponds to $\zeta\in L_A$ and $\zeta\in \overline{L}_A$ respectively. Moreover,
\begin{equation}
\label{zeta-tau'}
\left\vert (-\zeta)^{\pm 2i\kappa} e^{\mp i \zeta^2/2}\right\vert \lesssim \left\vert e^{\mp i\gamma z^2/2} \right\vert e^{-\eps^2(1-\gamma)24\tau} \lesssim \left\vert e^{\mp i\gamma z^2/2} \right\vert  \tau^{-1/2}
\end{equation}
where the $\pm$ sign corresponds to $\zeta \in  (\Sigma_A^2\cup \Sigma_A^3)\setminus L_A$ and $ \zeta\in (\Sigma_A^1\cup \Sigma_A^4)\setminus \overline{L}_A$ respectively. 
\end{lemma}
We now define
\begin{align*}
w^{A^0}(\zeta) &=\lim_{\tau\to\infty} (\Delta^0_A)^{-1}  (N_A \hat{w}^{A'})\Delta^0_A(\zeta),\\
w^{B^0}(\zeta) &=\lim_{\tau\to\infty} (\Delta^0_B)^{-1}  (N_B \hat{w}^{B'})\Delta^0_B(\zeta),\\
A^0 &=C^+(\cdot w^{A^0}_-)+C^-(\cdot w^{A^0}_+),\\
B^0 &=C^+(\cdot w^{B^0}_-)+C^-(\cdot w^{B^0}_+).
\end{align*}
\begin{proposition}
\label{prop:resolvent}
\begin{equation}
\norm{(1_A-A)^{-1} }{L^2(\Sigma_A)} , \, \norm{(1_B-B)^{-1}}{L^2(\Sigma_B)} \leq c 
\end{equation}
as $\tau\to\infty$.
\end{proposition}
\begin{proof}
From Lemma \ref{delta-B } and Lemma \ref{delta-A},  it is easily seen that 
\begin{equation}
\label{fix}
\norm{A-A^0}{L^2(\Sigma_A)}, \, \norm{B-B^0}{L^2(\Sigma_B)} \leq c \tau^{-1/2}.
\end{equation}
We will only establish the boundedness of  $(1_B-B)^{-1}$ since the case for  $(1_A-A)^{-1}$ is similar. From Lemma \ref{delta-B } we deduce that  on $\Sigma_B$
\begin{align}
\label{w-B-0}
w^{B^0}(\zeta) =\begin{cases}
\twomat{0}{0}{r(z_0) \zeta^{- 2i\kappa} e^{ i \zeta^2/2}}{0}, \quad  \zeta\in \Sigma_B^1,\\
\\
\twomat{0}{\dfrac{ \overline{r(z_0)}}{1 +  |r(z_0)|^2} \zeta^{ 2i\kappa} e^{ -i \zeta^2/2} }{0}{0}, \quad \zeta\in \Sigma_B^2, \\
\\
\twomat{0}{0}{
\dfrac{r(z_0)}{1 +  |r(z_0)|^2} \zeta^{- 2i\kappa} e^{ i \zeta^2/2} }{0}, \quad  \zeta\in \Sigma_B^3,\\
\\
\twomat{0}{\overline{ r(z_0)} \zeta^{ 2i\kappa} e^{ -i \zeta^2/2}  } {0}{0}, \quad  \zeta\in \Sigma_B^4.
\end{cases}
\end{align}
Setting 
$$v^{B^0}(\zeta) =I+w^{B^0}(\zeta)  $$
we first notice that
$v^{B^0}(\zeta) $ is precisely the jumps of the exactly solvable parabolic cylinder problem.  The solution of this problem is standard and can be found in \cite[Appendix A]{BJM16}. More importantly,  $v^{B^0}(\zeta) $ satisfies the 
Schwarz invariant condition:
$$v^{B^0}(\zeta)= v^{B^0}(\overline{\zeta})^\dagger$$
which will guarantee the uniqueness of the solution. By standard arguments in \cite{Zhou89} and \cite[Sec 7.5]{Deift99}, this implies the existence and boundedness of the resolvent operator $(1_B-B^0)^{-1}$. And the boundedness of $(1_B-B)^{-1}$ is a consequence of \eqref{fix} and the second resolvent identity.
\end{proof}
Indeed, for $\zeta\in\Sigma_B$ we let
\begin{equation}
m^{B^0}(\zeta)=I+\dfrac{1}{2\pi i}\int_{\Sigma_B}\dfrac{\left( (1_B-B^0)^{-1} I\right)(s) w^{B^0}(s)}{s-\zeta} ds
\end{equation}
then $m^{B^0}(\zeta)$ solves the following Riemann-Hilbert problem
\begin{align}
\label{RHP-B0}
\begin{cases}
m^{B^0}_+(\zeta) &=m^{B^0}_-(\zeta)v^{B^0}(\zeta), \quad \zeta\in\Sigma_B\\
m^{B^0}(\zeta) &\to I , \quad \zeta\to \infty
\end{cases}
\end{align}
In the large $\zeta$ expansion, 
$$m^{B^0}(\zeta) =I-\dfrac{m^{B^0}_1}{\zeta}+O(\zeta^{-2}), \quad \zeta\to\infty$$
thus
$$m^{B^0}_1=\dfrac{1}{2\pi i}\int_{\Sigma_B}{\left( (1_B-B)^{-1} I\right)(s) w^{B^0}(s)} ds.$$
Similarly, setting 
\begin{equation}
m^{B}(\zeta)=I+\dfrac{1}{2\pi i}\int_{\Sigma_B}\dfrac{\left( (1_B-B)^{-1} I\right)(s) w^{B}(s)}{s-\zeta} ds
\end{equation}
then $m^{B}(\zeta)$ solves the following Riemann-Hilbert problem
\begin{align}
\label{RHP-B}
\begin{cases}
m^{B}_+(\zeta) &=m^{B}_-(\zeta)v^{B}(\zeta), \quad \zeta\in\Sigma_B\\
m^{B}(\zeta) &\to I , \quad \zeta\to \infty
\end{cases}
\end{align}
Here $v^{B}(\zeta)=I+w^{B}(\zeta)$ where $w^B(\zeta)$ is given by \eqref{wB-1}-\eqref{wB-4}.
In the large $\zeta$ expansion, 
$$m^{B}(\zeta) =I-\dfrac{m^{B}_1}{\zeta}+O(\zeta^{-2}), \quad \zeta\to\infty$$
thus
$$m^{B}_1=\dfrac{1}{2\pi i}\int_{\Sigma_B}{\left( (1_B-B)^{-1} I\right)(s) w^{B}(s)} ds.$$
Setting $w^d=w^{B}-w^{B^0}$, a simple computation shows that
\begin{align*}
\int_{\Sigma_B }  \left( ( 1_B-B)^{-1}I \right)w^{B}  - \int_{\Sigma_B   } { \left( (1_B-B^0 )^{-1}I \right)w^{B^0} } &= \int_{\Sigma^{B} }w^d+  \int_{\Sigma_B  } \left( ( 1_B-B^0   )^{-1} (C_{w^d} I)  \right) w^{B}\\
&\quad+ \int_{\Sigma_B} \left( ( 1_B-B^0 )^{-1} (B^0 I)  \right)w^d \\
&\quad+\int_{\Sigma_B  } \left( (1_B-B^0 )^{-1} C_{w^d}  (1_B-B )^{-1}   \right) \left( B ( I) \right) { w^{B} } \\
&=\text{I} + \text{II}  + \text{III} + \text{IV}. 
\end{align*}
From Lemma \ref{delta-B } and Proposition \ref{prop:resolvent}, it is clear that
\begin{align*}
\left\vert \text{I}\right\vert  &\lesssim \tau^{-1/2},\\
\left\vert \text{II}  \right\vert   & \leq   \norm{( 1_B-B^0 )^{-1} }{L^2(\Sigma_B )}\norm{C_{w^d }I }{L^2(\Sigma_B) } \norm{ w^{B} }{L^2(\Sigma_B ) } \\
                                    & \lesssim \tau^{-1/2},\\                    
\left\vert \text{III}  \right\vert &\leq \norm{( 1_B-B^0 )^{-1} }{L^2(\Sigma_B )}\norm{B^0 I }{L^2(\Sigma_B) } \norm{ w^{d} }{L^2(\Sigma_B ) }\\
                                    & \lesssim \tau^{-1/2}.
\end{align*}
For the last term
\begin{align*}
\left\vert \text{IV} \right\vert & \leq  \norm{ (1_B-B^0 )^{-1} }{L^2(\Sigma_B )} \norm{( 1_B-B )^{-1} }{L^2(\Sigma_B )} \norm{C_{w^d}}{L^2(\Sigma_B )} \\
     &\quad \times  \norm{B( I ) }{L^2(\Sigma_B )} \norm{ w^B  }{L^2(\Sigma_B )} \\
     &\leq c \norm{ w^d}{L^\infty (\Sigma_B ) }  \norm{w^B }{L^2(\Sigma^{(3 )} )}^2\\
     & \lesssim \tau^{-1/2}.
\end{align*}
So we conclude that
\begin{equation}
\label{differ-B}
\left\vert \int_{\Sigma_B }  \left( ( 1_B-B)^{-1}I \right)w^{B}  - \int_{\Sigma_B   } { \left( (1_B-B^0 )^{-1}I \right)w^{B^0} }   \right\vert \lesssim \tau^{-1/2}.
\end{equation}
Clearly there is a parallel case for $\Sigma_A$:
\begin{equation}
\label{differ-A}
\left\vert \int_{\Sigma_A }  \left( ( 1_A-A)^{-1}I \right)w^{A}  - \int_{\Sigma_A  } { \left( (1_A-A^0 )^{-1}I \right)w^{A^0} }   \right\vert \lesssim \tau^{-1/2}.
\end{equation}
The explicit form of $m^{B^0}_1$ is given as follows (see \cite[Appendix A]{BJM16}) :
\begin{equation}
\label{explicit-B0}
m^{B^0}_1=\twomat{0}{-i\beta_{12}}{i\beta_{21}}{0}
\end{equation}
where
$$\beta_{12}=\dfrac{\sqrt{2\pi } e^{i\pi/4} e^{-\pi \kappa } }{r(z_0) \Gamma(-i\kappa)}, \qquad \beta_{21}=\dfrac{-\sqrt{2\pi } e^{-i\pi/4} e^{-\pi \kappa } }{ \overline{ r(z_0) } \Gamma(i\kappa)}$$
and $\Gamma(z)$ is the \textit{Gamma} function.
Recall that on $\Sigma_B$, $\zeta=\sqrt{48z_0 t}(z-z_0)$, thus by \eqref{differ-B}, we have
\begin{equation}
\label{differ-m-B}
\left\vert   \dfrac{m_1^B}{\zeta}-\dfrac{m_1^{B^0}}{\zeta}   \right\vert \lesssim \dfrac{1}{t(z-z_0)}.
\end{equation}
Using the explicit form of $w^{B^0}$ given by \eqref{w-B-0}, symmetry reduction given by \eqref{minus} and their analogue for $w^{A^0}$, we verify that
\begin{equation}
v^{A^0}(z)=  \sigma_3 \overline{ v^{B^0}( - \overline{z} )  } \sigma_3  
\end{equation}
which  in turn implies by uniqueness that
\begin{equation}
m^{A^0}(z)= \sigma_3 \overline{ m^{B^0}( - \overline{z} )  } \sigma_3  
\end{equation}
and from this we deduce that
\begin{align}
\label{mA-mB}
m^{A^0}_1 &=-\sigma_3 \overline{ m^{B^0}_1  } \sigma_3  \\
\nonumber
&=\twomat{0}{i \overline{\beta}_{12}}{-i \overline{\beta}_{21} }{0}.
\end{align}
We also have an analogue of \eqref{differ-m-B} for $m_1^{A^0}$: 
\begin{equation}
\label{differ-m-A}
\left\vert   \dfrac{m_1^A}{\zeta}-\dfrac{m_1^{A^0}}{\zeta}   \right\vert \lesssim \dfrac{1}{t(z+z_0)}.
\end{equation}
Collecting all the computations above, we write down the asymptotic expansions of solutions to Problem \ref{prob:mkdv.A} and Problem \ref{prob:mkdv.B} respectively.
\begin{proposition}
\label{solution-A-B}
Setting $\zeta=\sqrt{48z_0 t}(z + z_0) $, the solution to RHP Problem  \ref{prob:mkdv.A}  $m^{A'}$ admits the following expansion:
\begin{equation}
\label{expansion-A}
m^{A'}(z(\zeta) ;x,t)=I +\dfrac{1}{\zeta}\twomat{0}{i ( \delta^0_A)^2 \overline{\beta}_{12}}{-i ( \delta^0_A)^{-2}\overline{\beta}_{21} }{0} +\mathcal{O}(t^{-1}).
\end{equation}
Similarly, setting $\zeta=\sqrt{48z_0 t}(z - z_0) $, the solution to RHP Problem  \ref{prob:mkdv.B}  $m^{B'}$ admits the following expansion:
\begin{equation}
\label{expansion-B}
m^{B'}(z(\zeta) ;x,t)=I +\dfrac{1}{\zeta}\twomat{0}{-i ( \delta^0_B)^2 {\beta}_{12}}{i ( \delta^0_B)^{-2}{\beta}_{21} }{0} +\mathcal{O}(t^{-1}).
\end{equation}
\end{proposition}
Now we construct $m^\RHP_1$ needed in the proof of Proposition \ref{Prop:expo}. 
In Figure \ref{fig:contour-2}, we let $\rho$ be the radius of the circle $C_A$ ($C_B$) centered at $z_0$ ($-z_0$). We seek a solution of the form
\begin{equation}
\label{parametrix}
m^\RHP_*(z)=\begin{cases}
E_2(z)m^{(br)}(z) \quad  &\left\vert z\pm z_0 \right\vert>\rho \\
E_2(z)m^{(br)}(z)  m^{A'}(z) \quad &\left\vert z + z_0 \right\vert\leq\rho \\
E_2(z)m^{(br)}(z)  m^{B'}(z) \quad & \left\vert z - z_0 \right\vert\leq \rho 
\end{cases}
\end{equation}
Since $m^{(br)}$, $ m^{A'}$ and $ m^{B'}$ solve Problem \ref{prob:mkdv.br}, Problem \ref{prob:mkdv.A} and Problem \ref{prob:mkdv.B} respectively, we can construct the solution $m^\RHP_*(z)$ if we find $E_2(z)$. Indeed,  $E_2$ solves the following Riemann-Hilbert problem:
\begin{problem}
\label{prob: E-2}
Find a matrix-valued function $E_2(z)$ on $\bbC \setminus \left( C_A\cup C_B \right)$ with the following properties:
\begin{enumerate}
\item		$E_2(z) \rarr I$ as $ z \rarr \infty$,
\item		$E_2(z) $ is analytic for $z \in  \bbC \setminus \left( C_A\cup C_B \right)$
			with continuous boundary values
			$E_{2\pm}(z)$.
\item On $ C_A\cup C_B $we have the following jump conditions
$$E_{2+}(z)=E_{2-}(z) v^{ (E) }(z)$$
			where 
			\begin{equation}
			\label{v-E}
			 v^{ (E) }(z)=\begin{cases}
			m^{(br)} (z)m^{A'}(z(\zeta) )m^{(br)} (z)^{-1}, \quad z\in C_A\\
			m^{(br)} (z)m^{B'}(z(\zeta) )m^{(br)} (z)^{-1}, \quad z\in C_B
			\end{cases}
			\end{equation}
\end{enumerate}
\end{problem}
Setting 
$$\eta(z)=E_{2-}(z)$$
then by standard theory, we have the following singular integral equation
$$  \eta=I+ C_{v^{(E)}}\eta $$
where the singular integral operator is defined by:
$$  C_{v^{(E)}}\eta =C^-\left( \eta \left( v^{(E)}-I  \right)  \right). $$
We first deduce from \eqref{expansion-A}-\eqref{expansion-B} that 
\begin{equation}
\norm{v^{(E)}-I}{L^\infty} \lesssim t^{-1/2}
\end{equation}
hence the operator norm of $C_{v^{(E)}} $ 
\begin{equation}
\label{norm-vE}
\norm{C_{v^{(E)}}f }{ L^2} \leq \norm{f}{L^2} \norm{v^{(E)}-I}{L^\infty} \lesssim t^{-1/2}.
\end{equation}
Then the resolvent operator $(1-C_{v^{(E)}})^{-1}$ can be obtained through \textit{Neumann} series and we obtain the unique solution to Problem \ref{prob: E-2}:
\begin{equation}
E_2(z)=I+ \dfrac{1}{2\pi i}\int_{C_A\cup C_B} \dfrac{ (1+\eta(s))(v^{(E)}(s)-I )   }{s-z}ds
\end{equation}
which admits the following asymptotic expansion in $z$:
\begin{equation}
\label{E2-expan}
E_2(z)=I+ \dfrac{E_{2,1}}{z} +\mathcal{O}\left( \dfrac{1}{z^2} \right).
\end{equation}
Using the bound on the operator norm \eqref{norm-vE}, we obtain
\begin{align}
\label{bound-E-1}
E_{2,1}(z) &=-\dfrac{1}{2\pi i}\int_{C_A\cup C_B} { (1+\eta(s))(v^{(E)}(s)-I )   }ds  \\
              &=-\dfrac{1}{2\pi i}\int_{C_A\cup C_B} { (v^{(E)}(s)-I )   }ds +\mathcal{O}(t^{-1}).
\end{align}
Given the form of $v^{(E)}$ in \eqref{v-E} and the asymptotic expansions \eqref{expansion-A}-\eqref{expansion-B}, an application of Cauchy's integral formula leads to
\begin{align}
\label{E-1-2 cauchy}
E_{2,1}&=\dfrac{1}{\sqrt{48 z_0 t}} m^{(br)}(z_0) \twomat{0}{-i ( \delta^0_B)^2 {\beta}_{12}}{i ( \delta^0_B)^{-2}{\beta}_{21} }{0}  m^{(br)}(z_0)^{-1}\\
\nonumber
    & \quad + \dfrac{1}{\sqrt{48 z_0 t}} m^{(br)}(-z_0) \twomat{0}{ i ( \delta^0_A)^2 \overline{\beta}_{12} }{-i ( \delta^0_A)^{-2}\overline{\beta}_{21} }{0} m^{(br)}(-z_0)^{-1}\\
    \nonumber
     & \quad + \mathcal{O}(t^{-1}).
\end{align}
We now completed the construction of the matrix-valued function $E_2 (z)$ hence $m^\RHP_*(x,t; z)$. Combining this with Proposition \ref{expo}, we obtain $m^\RHP(z)$ in \eqref{factor-LC}.


\section{The $\dbar$-Problem}
\label{sec:dbar}
 
From \eqref{factor-LC} we have matrix-valued function
\begin{equation}
\label{N3}
m^{(3)}(z;x,t) = m^{(2)}(z;x,t) m^\RHP(z; x,t)^{-1}. 
\end{equation}
The goal of this section is to show that $m^{(3)}$ only results in an error term $E$ with higher order decay rate than the leading order term of the asymptotic formula. The computations and proofs are standard. We follow \cite[ Section 5 ]{ CL19 } with slight modifications. 

Since $m^\RHP(z; x, t)$ is analytic in $\bbC \setminus \left( \Sigma^{(3)} \cup \Gamma  \right)$, we may compute
\begin{align*}
\dbar m^{(3)}(z;x,t) 	&=	\dbar m^{(2)}(z;x,t) m^\RHP(z; x,t)^{-1}\\	
								&=	m^{(2)}(z;x,t) \, \dbar \calR^{(2)}(z) m^\RHP(z; x,t)^{-1}	&\text{(by \eqref{N2.dbar})}\\
								&=	m^{(3)}(z;x,t) m^{\RHP}(z; x, t) \, \dbar \calR^{(2)}(z) m^\RHP(z; x,t)^{-1}	& \text{(by \eqref{N3})}\\
								&=	m^{(3)}(z;x,t)  W(z;x,t)
\end{align*}
where
 \begin{equation}
 \label{W-bound}
 W(z;x,t) = m^{\RHP}(z;x,t ) \, \dbar \calR^{(2)}(z)  m^\RHP(z;x,t )^{-1}. 
 \end{equation}
We thus arrive at the following pure $\dbar$-problem:

\begin{problem}
\label{prob:DNLS.dbar}
Give $r \in H^{1}(\bbR)$, find a continuous matrix-valued function
$m^{(3)}(z;x,t)$ on $\bbC$ with the following properties:
\begin{enumerate}
\item		$m^{(3)}(z;x,t) \rarr I$ as $|z| \rarr \infty$.
\medskip
\item		$\dbar m^{(3)}(z;x,t) = m^{(3)}(z;x,t) W(z;x,t)$.
\end{enumerate}
\end{problem}

It is well understood (see for example \cite[Chapter 7]{AF}) that the solution to this $\dbar$ problem is equivalent to the solution of a Fredholm-type integral equation involving the solid Cauchy transform
$$ (Pf)(z) = \frac{1}{\pi} \int_\bbC \frac{1}{\zeta-z} f(\zeta) \, d\zeta $$
where $d$ denotes Lebesgue measure on $\bbC$.  Also throughout this section, $\zeta$ refers to complex numbers, not to be confused with $\zeta=\sqrt{48z_0 t} (z \pm z_0)$ in the previous section.
\begin{lemma}
A bounded and continuous matrix-valued function $m^{(3)}(z;x,t)$ solves Problem \eqref{prob:DNLS.dbar} if and only if
\begin{equation}
\label{DNLS.dbar.int}
m^{(3)}(z;x,t) =I+ \frac{1}{\pi} \int_\bbC \frac{1}{\zeta-z} m^{(3)}(\zeta;x,t) W(\zeta;x,t) \, d\zeta.
\end{equation}
\end{lemma}

 Using the integral equation formulation \eqref{DNLS.dbar.int}, we will prove:

\begin{proposition}
\label{prop:N3.est}
Suppose that $r \in H^{1}(\bbR)$.
Then, for $t\gg 1$, there exists a unique solution $m^{(3)}(z;x,t)$ for Problem \ref{prob:DNLS.dbar} with the property that 
\begin{equation}
\label{N3.exp}
m^{(3)}(z;x,t) = I + \frac{1}{z}  m^{(3)}_1(x,t) + o\left( \frac{1}{z} \right) 
\end{equation}
for $z=i\sigma$ with $\sigma \rarr +\infty$. Here
\begin{equation}
\label{N31.est}
\left| m^{(3)}_1(x,t) \right| \lesssim (z_0 t)^{-3/4}
\end{equation}
 where the implicit constant in \eqref{N31.est} is uniform for 
$r$ in a bounded subset of $H^{1}(\bbR)$ .
\end{proposition}

\begin{proof} Given Lemmas \ref{lemma:dbar.R.bd}--\ref{lemma:N31.est},
 as in \cite{LPS}, we   first show that, for large $t$, the integral operator $K_W$ 
defined by
\begin{equation*}
\left( K_W f \right)(z) = \frac{1}{\pi} \int_\bbC \frac{1}{\zeta-z} f(\zeta) W(\zeta) \, d \zeta
\end{equation*}
is bounded by
\begin{equation}
\label{dbar.int.est1}
\norm{K_W}{L^\infty \rarr L^\infty} \lesssim (z_0 t)^{-1/4}
\end{equation}
where the implied constants depend only on $\norm{r}{H^{1}}$.  This is the goal of Lemma \ref{lemma:KW}.
It implies that 
\begin{equation}
\label{N3.sol}
m^{(3)} = (I-K_W)^{-1}I
\end{equation}
exists as an $L^\infty$ solution of \eqref{DNLS.dbar.int}.

We  then show in Lemma \ref{lemma:N3.exp} that the solution $m^{(3)}(z;x,t)$ has a large-$z$ asymptotic expansion of the form \eqref{N3.exp}
where $z \rarr \infty$ along the \emph{positive imaginary axis}. Note that, for such $z$, we can bound $|z-\zeta|$ below by a constant times $|z|+|\zeta|$.
Finally,  in Lemma \ref{lemma:N31.est} we  prove  estimate \eqref{N31.est}
where the constants are uniform in $r$ belonging to a bounded subset of $H^{1}(\bbR)$.
Estimates \eqref{N3.exp}, \eqref{N31.est}, and \eqref{dbar.int.est1} result from the  bounds obtained in the next four lemmas.
\end{proof}

\begin{lemma}
\label{lemma:dbar.R.bd}
Set $z=(u \mp \xi)+iv$. We have 
\begin{equation}
\label{dbar.R2.bd}
\left| \dbar \calR^{(2)} e^{\pm 2i\theta}  \right| 	\lesssim		
		\left( |p_i'(\Real (z))| + |z \mp \xi|^{-1/2} + \left\vert \Xi_\calZ(z)  \right\vert \right) e^{-z_0t|u||v|}.		
			\end{equation}
\end{lemma}

\begin{proof}
We only show the inequalities above in $\Omega_1$ and $\Omega_7^+$. Recall that near $z_0$
$$i\theta(z; x, t)=4it \left( (z - z_0)^3 + 3z_0 (z - z_0)^2 - 2z_0^3  \right).$$
In $\Omega_1$, we use the facts that $u\geq 0$, $v\geq 0$ and $|u|\geq |v|$ to deduce
\begin{align*}
\text{Re}(2i\theta) &=8it(3iu^2v-iv^3+6iuvz_0)\\
                              &=8t(-3u^2v+v^3-6uvz_0)\\
                              &\leq 8t(-3u^2v +u^2v-6uvz_0)\\
                              &\leq 8t (-2u^2v-6uvz_0)\\
                              &\leq -8|u| |v| z_0 t.
\end{align*}
Similarly, in $\Omega_7^+$, we have $u\leq 0$, $v\geq 0$ and $|u|\geq |v|$, hence
\begin{align*}
\text{Re}(-2i\theta) &=-8it(3iu^2v-iv^3+6iuvz_0)\\
                              &=8t(3u^2v+6uvz_0)\\
                              &\leq 8t(-3u z_0v +6uvz_0)\\
                            &\leq -8|u| |v| z_0 t.
\end{align*}
Estimate \eqref{dbar.R2.bd} then follows from Lemma \ref{lemma:dbar.Ri}. The quantities $p_i'(\Real z)$ are all bounded uniformly for  $r$ in a bounded subset of $H^{1}(\bbR)$.  

\end{proof}

\begin{lemma}
\label{lemma:RHP.bd}For the localized Riemann-Hilbert problem from Problem \ref{MKDV.RHP.local}, we have
\begin{align}
\label{RHP.bd1}
\norm{m^\RHP(\dotarg; x, t)}{\infty}	&	\lesssim		1,\\[5pt]
\label{RHP.bd2}
\norm{m^\RHP(\dotarg; x,t )^{-1}}{\infty}	&	\lesssim	1.
\end{align}
All implied constants are uniform  for  $r$ in a bounded subset of $H^{1}(\bbR)$.
\end{lemma}

The proof of this lemma is a consequence of the previous section.

\begin{lemma}
\label{lemma:KW}
Suppose that $r\in H^{1}(\bbR)$.
Then, the estimate \eqref{dbar.int.est1}
holds, where the implied constants depend on $\norm{r}{H^{1}}$.
\end{lemma}

\begin{proof}
To prove \eqref{dbar.int.est1}, first note that
\begin{align}
 \norm{K_W f}{\infty} &\leq \norm{f}{\infty} \int_\bbC \frac{1}{|z-\zeta|}|W(\zeta)| \, dm(\zeta) 
                                \end{align}
so that we need only estimate the right-hand integral. We will prove the estimate in the region $ z\in\Omega_1$ since estimates for the remaining regions are identical. From \eqref{W-bound}, it follows
$$ |W(\zeta)| \leq \norm{m^{\RHP}}{\infty} \norm{(m^{\RHP})^{-1}}{\infty} \left| \dbar R_1\right| |e^{2i\theta}|.$$
Setting $z=\alpha+i\beta$ and $\zeta=(u+z_0)+iv$, the region $\Omega_1$ corresponds to $u\geq v \geq 0 $. We then have from \eqref{dbar.R2.bd} \eqref{RHP.bd1}, and \eqref{RHP.bd2} that
$$
 \int_{\Omega_1}  \frac{1}{|z-\zeta|} |W(\zeta)| \, d\zeta  \lesssim  I_1 + I_2 +I_3
$$
where
\begin{align*}
I_1 	&=	\int_0^\infty \int_v^\infty \frac{1}{|z-\zeta|} |p_1'(u)| e^{-tz_0uv} \, du \, dv, \\[5pt]
I_2	&=	\int_0^\infty \int_v^\infty \frac{1}{|z-\zeta|} \left| u+iv \right|^{-1/2} e^{-t z_0 uv} \, du \, dv,\\
I_3	&=	\int_0^\infty \int_v^\infty \frac{1}{|z-\zeta|} \left|  \dbar (\Xi_\calZ(\zeta  ) ) \right| e^{-t z_0 uv} \, du \, dv.
\end{align*}
It now follows from \cite[proof of Proposition D.1]{BJM16} that
$$
|I_1|, \, |I_2|, \, |I_3| \lesssim (z_0t)^{-1/4}.
$$
It then follows that
$$ \int_{\Omega_1} \frac{1}{|z-z_0|} |W(\zeta)| \, d\zeta \lesssim (z_0t)^{-1/4} $$
which, together with similar estimates for the integrations over the remaining $\Omega_i$s,  proves \eqref{dbar.int.est1}.
\end{proof}

\begin{lemma}
\label{lemma:N3.exp}
For $z=i\sigma$ with $\sigma \rarr +\infty$, the expansion \eqref{N3.exp} holds with 
\begin{equation}
\label{N3.1}
m^{(3)}_1(x,t) = \frac{1}{\pi} \int_{\bbC} m^{(3)}(\zeta;x,t) W(\zeta;x,t) \, d\zeta . 
\end{equation}
\end{lemma}

\begin{proof}
We write   \eqref{DNLS.dbar.int} as 
$$
m^{(3)}(z;x,t) = I+ \frac{1}{z} m^{(3)}_1(x,t) + \frac{1}{\pi z} \int_{\bbC} \frac{\zeta}{z-\zeta} m^{(3)}(\zeta;x,t) W(\zeta;x,t) \, dm(\zeta)
$$
where $m^{(3)}_1$is given by \eqref{N3.1}. If $z=i\sigma$, it is easy to see that $|\zeta|/|z-\zeta|$ is bounded above by a fixed constant independent of $z$, while $|m^{(3)}(\zeta;x,t)| \lesssim 1$ by the remarks following \eqref{N3.sol}. If we can show that $\int_\bbC |W(\zeta;x,t)| \, d\zeta$ is finite, it will follow from the Dominated Convergence Theorem that 
$$
\lim_{\sigma \rarr \infty} \int_\bbC \frac{\zeta}{i\sigma-\zeta} m^{(3)}(\zeta;x,t) W(\zeta;x,t) \, d\zeta = 0 
$$ 
which implies the required asymptotic estimate. We will estimate $\dint_{\hspace{-1.25mm} \Omega_1} |W(\zeta)| \, dm(\zeta)$ since the other estimates are identical. One can write
$$\Omega_1= \left\{ (u+z_0,v): v \geq 0, \, v \leq u < \infty\right\}.$$ 
Using \eqref{dbar.R2.bd}, \eqref{RHP.bd1}, and \eqref{RHP.bd2}, we may then estimate
$$
\int_{\Omega_1} |W(\zeta;x,t)| \, d \zeta	\lesssim  I_1+I_2 +I_3
$$
where
\begin{align*}
I_1	&=	\int_0^\infty \, \int_v^\infty \left| p_1'(u+z_0) \right| e^{-tz_0uv} \, du \, dv\\
I_2	&=	\int_0^\infty \int_v^\infty \left| u^2 + v^2 \right|^{-1/2} e^{-tz_0 uv} \, du \, dv \\
I_3	&=	\int_0^\infty \int_v^\infty \left| \dbar \left( \Xi_\calZ(\zeta) \right)\right| e^{-tz_0 uv} \, du \, dv
\end{align*}
It now follows from \cite[Proposition D.2]{BJM16} that
$$I_1, \, I_2, \, I_3\lesssim (z_0t)^{-3/4}.$$
These estimates together show that
\begin{equation}
\label{W.L1.est}
\int_{\Omega_1} |W(\zeta;x,t)| \, d \zeta	\lesssim( z_0t)^{-3/4}
\end{equation}
and that the implied constant depends only on $\norm{r}{H^{1}}$.  In particular, the integral \eqref{W.L1.est} is bounded uniformly as $t \rarr \infty$.
\end{proof}

\begin{lemma}
\label{lemma:N31.est}
The estimate   \eqref{N31.est} 
 holds with constants uniform in $r$ in a bounded subset of $H^{1}(\bbR)$ .
\end{lemma}

\begin{proof}
From the representation formula \eqref{N3.1}, Lemma \ref{lemma:KW}, and the remarks following, we have
$$ \left|m^{(3)}_1(x,t) \right| \lesssim \int_\bbC |W(\zeta;x,t)| \, d\zeta. $$
In the proof of Lemma \ref{lemma:N3.exp}, we bounded this integral by $( z_0t)^{-3/4}$ modulo constants with the required uniformities.
\end{proof}
\section{Long-Time Asymptotics}
\label{sec:large-time}

We now put together our previous results and formulate the long-time asymptotics of $u(x,t)$  in Region I.  Undoing all transformations we carried out previously,  we get back $m$:
\begin{equation}
\label{N3.to.N}
m(z;x,t) = m^{(3)}(z;x,t) m^\RHP(z; z_0) \calR^{(2)}(z)^{-1} \delta(z)^{\sigma_3}.
\end{equation}
By stand inverse scattering theory, the coefficient of $z^{-1}$ in the large-$z$ expansion for $m(z;x,t)$ will be the solution to the mKdV. 
\begin{lemma}
\label{lemma:N.to.NRHP.asy}
For $z=i\sigma$ and $\sigma \rarr +\infty$, the asymptotic relations
\begin{align}
\label{N.asy}
m(z;x,t) 				&=	I + \frac{1}{z} m_1(x,t) + o\left(\frac{1}{z}\right)\\
\label{N.RHP.asy}
m^\RHP(z;x,t)		&=	I + \frac{1}{z} m^\RHP_1(x,t) + o\left(\frac{1}{z}\right)
\end{align}
hold. Moreover,
\begin{equation}
\label{N.to.NRHP.asy} 
\left(m_1(x,t)\right)_{12} = \left(m^\RHP_1(x,t)\right)_{12} + \bigO{(z_0 t)^{-3/4}}.
\end{equation}
\end{lemma}

\begin{proof}
By Lemma \ref{lemma:delta} (iii), 
the expansion
\begin{equation}
\label{delta.sigma.asy} 
\delta(z)^{\sigma_3} = \twomat{1}{0}{0}{1} + \frac{1}{z} \twomat{\delta_1}{0}{0}{\delta_1^{-1}} + \bigO{z^{-2}} 
\end{equation}
holds, with the remainder in \eqref{delta.sigma.asy}
 uniform in $r$ in a bounded subset of $H^{1}$. 
\eqref{N.asy} follows from \eqref{N3.to.N},  \eqref{N.RHP.asy}, the fact that $\calR^{(2)} \equiv I$ in $\Omega_2$,
and \eqref{delta.sigma.asy}.
Notice the fact that the 
diagonal matrix in \eqref{delta.sigma.asy} does not affect the $12$-component of $m$. Hence, for 
$z=i\sigma$,
$$ 
\left(m(z;x,t)\right)_{12} = 
		\frac{1}{z}\left(m^{(3)}_1(x,t)\right)_{12} + 
		\frac{1}{z}\left(m^\RHP_1(x,t)\right)_{12} + o\left(\frac{1}{z}\right)
$$
and result now follows from \eqref{N31.est}. 
\end{proof}
From previous results (see Proposition \ref{Prop:expo}, Problem \ref{prob:mkdv.br} and Problem \ref{prob: E-2} ) we have:
\begin{align}
\label{expan-mLC}
m^\RHP (z) &=E_1(z)m^\RHP_*(z)\\
\nonumber
              &= E_1(z)E_2(z)m^{(br)}(z)\\
 \nonumber
             &= \left(  I+\dfrac{E_{1,1} }{z}+... \right)  \left(   I+\dfrac{E_{2,1} }{z} +...\right)  \left(   I+\dfrac{m^{(br)}_1 }{z} +...\right)                  
\end{align}
as $z\to\infty.$

Together with Lemma \ref{lemma:N.to.NRHP.asy},  we arrive at the asymptotic formula in Region I:

\begin{proposition}
\label{lemma:N.RHP.asy}
The function 
\begin{equation}
\label{q.recon.bis}
u(x,t) = 2 \lim_{z \rarr \infty} z\, m_{12}(z;x,t)
\end{equation}
takes the form 
$$ u(x,t) = u^{(br)} (x,t)+  u_{as}(x,t) +\mathcal{O}\left( (z_0 t)^{-3/4}\right) $$
where  $ u^{(br)} (x,t)$ is given by \eqref{u-breather} and
\begin{align}
\label{u-asym}
u_{as}(x,t) &= \dfrac{1}{\sqrt{48tz_0}}\left( m_{11}^{(br)}( -z_0)^2 \left( i ( \delta^0_A)^2 \overline{\beta}_{12} \right) + m_{12}^{(br)} (-z_0)^2 \left(  i ( \delta^0_A)^{-2}\overline{\beta}_{21} \right)   \right)\\
\nonumber
&\quad + \dfrac{1}{\sqrt{48tz_0}}\left( m_{11}^{(br)} (z_0)^2 \left( -i ( \delta^0_B)^2 {\beta}_{12} \right)    -m_{12}^{(br)}(z_0)^2 \left( i ( \delta^0_B)^{-2}{\beta}_{21} \right) \right).
\end{align}
\end{proposition}
\begin{proposition} If we choose the frame  $x=\mathrm{v} t$ with $\mathrm{v} <0$ and $\mathrm{v} \neq 4\eta^2_j-12\xi^2_j $ for all $1\leq j \leq N_2$, then
$$ u(x,t) =  u_{as}(x,t) +\mathcal{O}\left( (z_0 t)^{-3/4}\right) $$
where
\begin{equation}
\label{as-solitonless}
u_{as}(x,t)=\left( \dfrac{\kappa}{3tz_0}\right)^{1/2}\cos \left(16tz_0^3-\kappa\log(192tz_0^3)+\phi(z_0) \right)
\end{equation}
with
\begin{align*}
\phi(z_0)&=\arg \Gamma(i\kappa)-\dfrac{\pi}{4}-\arg r(z_0)+\dfrac{1}{\pi}\int_{-z_0}^{z_0}\log\left( \dfrac{1+|r(\zeta)|^2}{1+|r(z_0)|^2} \right)\dfrac{d\zeta}{\zeta-z_0}\\
             & \quad -4 \left( \sum_{k=1}^{N_1} \arg(z_0-z_k)  +  \sum_{z_j \in \mathcal{B}_\ell } \arg(z_0- {z_j}) +  \sum_{z_j\in \mathcal{B}_\ell } \arg(z_0 + \overline{z_j})   \right)
\end{align*}
\end{proposition}
\begin{proof}
Indeed, if we choose $v$ such that
$$4\eta_1^2-12\xi_1^2<...<4\eta_{\ell-1}^2-12\xi_{\ell-1}^2<\mathrm{v} <4\eta_{\ell }^2-12\xi_{ \ell }^2<...<4\eta_{\footnotesize{ N_2}}^2-12\xi_{N_2}^2$$
and define the same  $\delta$ function as \eqref{RH.delta.sol}. We follow the same procedure as in Section 3 and arrive at the following set of deformed contours and conclude that on the red portion of the contour all jump matrices decay exponentially as $t\to\infty$. Thus the localized RHP reduces to Problem \ref{prob:mkdv.A} and Problem \ref{prob:mkdv.B}. We then follow \cite{CL19} and  \cite[Section 4 ]{DZ93} to derive the explicit formula of $u_{as}$ in \eqref{as-solitonless}.
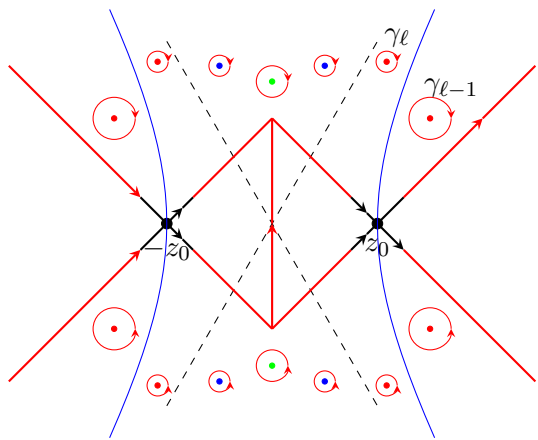
\begin{figure}[h!]
\caption{ $\Sigma^{(3)}\cup\Gamma$}
\vskip 15pt
\begin{tikzpicture}[scale=0.7]
 \draw[dashed] (-2 , -3.464)--(2 , 3.464);
   \draw[dashed] (-2 , 3.464)--(2 , -3.464);
\draw[thick]		(2,0) -- (2.5, 0.5);								
\draw[->,thick,>=stealth] [red]	 (2.5, 0.5)--(4, 2) ;
\draw [thick] [red]	 (4, 2)--(5,3) ;
\draw[thick] 		(-2,0) -- (-2.5, 0.5);					
\draw[->,thick,>=stealth]  [red]	  (-5,3) --  (-2.5, 0.5);	
\draw[thick] 		(-2,0) -- (-2.5, -0.5);					
\draw[->,thick,>=stealth]  [red]	  (-5,-3) --  (-2.5, -0.5);	
\draw[thick,->,>=stealth]		(2,0) -- (2.5,-0.5);								
\draw[thick]	[red]					(2.5,-0.5) -- (5,-3);
\draw [thick]	(-1.7, 0.3) --(-1.5, 0.5); 
\draw[thick,->,>=stealth] 	(-2, 0)--(-1.7, 0.3);						
\draw[thick] [red]	 (0,2) -- (-1.5, 0.5);
\draw[thick] 	(-1.5,- 0.5)--(-1.7, -0.3) ; 
\draw[thick,->,>=stealth] 	(-2, 0)--(-1.7, -0.3);						
\draw[thick] [red]	 (0, -2) -- (-1.5, -0.5);
\draw[thick]	[red]			(0,2) -- (1.5, 0.5);			
\draw	[thick]   (2,0) -- (1.8, 0.2);
\draw [thick,->,>=stealth]	(1.5, 0.5)-- (1.8, 0.2);
\draw[thick][red]		(0,-2) -- (1.5,-0.5);				
\draw[thick,->,>=stealth] 					(1.5, -0.5) -- (1.8, -0.2);
\draw[thick] (1.8, -0.2)--(2, 0);
\draw	[fill]							(-2,0)		circle[radius=0.1];	
\draw	[fill]							(2,0)		circle[radius=0.1];
\draw[->,thick,>=stealth] [red]		(0, -2) -- (0,0);
\draw[thick]	[red]		(0,0) -- (0,2);		
\node[below] at (-2,-0.1)			{$-z_0$};
\node[below] at (2,-0.1)			{$z_0$};
 \pgfmathsetmacro{\c}{2}
    \pgfmathsetmacro{\d}{3.464} 
   \draw [blue] plot[domain=-1:1] ({\c*cosh(\x)},{\d*sinh(\x)});
    \draw  [blue]  plot[domain=-1:1] ({-\c*cosh(\x)},{\d*sinh(\x)});
    \draw [blue, fill=blue] (1,3) circle [radius=0.05];
\draw[->,>=stealth] [red] (1.2, 3) arc(360:0:0.2);
\draw [blue, fill=blue] (-1,3) circle [radius=0.05];
\draw[->,>=stealth] [red] (-0.8, 3) arc(360:0:0.2);
\draw [blue, fill=blue] (1,-3) circle [radius=0.05];
\draw[->,>=stealth] [red] (1.2, -3) arc(0:360:0.2);
\draw [blue, fill=blue] (-1,-3) circle [radius=0.05];
\draw[->,>=stealth] [red] (-0.8, -3) arc(0:360:0.2);
\draw [green, fill=green]  (0, 2.7) circle [radius=0.05];
\draw[->,>=stealth] [red] (0.3, 2.7) arc(360:0:0.3);
\draw [green, fill=green] (0, -2.7) circle [radius=0.05];
\draw[->,>=stealth] [red]  (0.3, -2.7) arc(0:360:0.3);
\draw[->,>=stealth] [red](-2.6,2) arc(360:0:0.4);
\draw[->,>=stealth] [red] (3.4,2) arc(360:0:0.4);
\draw[->,>=stealth] [red](-2.6,-2) arc(0:360:0.4);
\draw[->,>=stealth] [red] (3.4,-2) arc(0:360:0.4);
\draw [red, fill=red] (-3,2) circle [radius=0.05];
\draw [red, fill=red] (3,2) circle [radius=0.05];
\draw [red, fill=red] (-3,-2) circle [radius=0.05];
\draw [red, fill=red] (3,-2) circle [radius=0.05];
\draw [red, fill=red] (2.1749, 3.0764) circle [radius=0.05];
\draw[->,>=stealth] [red]  (2.3749, 3.0764) arc(360:0:0.2);
\draw [red, fill=red] (-2.1749, 3.0764) circle [radius=0.05];
\draw[->,>=stealth] [red] (-1.9749, 3.0764) arc(360:0:0.2);
\draw [red, fill=red] (-2.1749, -3.0764) circle [radius=0.05];
\draw[->,>=stealth] [red] (-1.9749, -3.0764) arc(0:360:0.2);
\draw [red, fill=red] (2.1749, -3.0764) circle [radius=0.05];
\draw[->,>=stealth] [red] (2.3749, -3.0764) arc(0:360:0.2);
\node [above] at (2.3749, 3.1764) {$\gamma_\ell$} ;
\node [above] at (3.4, 2.2) {$\gamma_{\ell-1}$} ;
\end{tikzpicture}
\label{fig:contour-d}
\end{figure}
\end{proof}

\section{Regions II-III}
We now turn to the study of the Regions II-III. We first study Region II. Our starting point is RHP Problem \ref{RHP-1} and the strategy of the proof is as follows:
\begin{itemize}
\item[1.] We conjugate the jump matrices of Problem \ref{RHP-1} of by a scalar function $\psi$.
\item[2.] We scale the conjugated jump matrices by a factor determined by the region.
\item[3.] We use $\dbar$-steepest descent to study the scaled RHP and obtain both leading term and error term.
\item[4.] We multiply by the scaling factor to get the asymptotic formula.
\end{itemize}
We then study Region III. We mention that in both regions the application of $\dbar$ steepest descent method is analogous to that in Section 2-6. For the purpose of brevity we are only going to display the calculations directly related to the leading order term and error terms.
\subsection{Region II}
In this region, $|x/t^{-1/3}|=\mathcal{O}(1)$ as $t\to \infty$.
We first mention that for $x>0$, we have the stationary points 
$$\pm z_0=\pm \sqrt{\dfrac{-x }{ 12 t }}=\pm i  \sqrt{\dfrac{ |x| }{ 12 t }}$$
stay on the imaginary axis. So we are only going to study the case for $x<0$ since for $x>0$ the asymptotic formula will follow from a similar (and simpler) computation. 
For $x<0$, we first notice that
$$z_0=\sqrt{\dfrac{-x}{12t}}= \sqrt{\dfrac{-x}{12t^{1/3}}}t^{-1/3} \to 0 \quad \text{as}\,\, t \to \infty$$
By Remark \ref{painleve-soliton}, for the phase function $e^{i\theta(z; x,t)}$ we have the following signature table:
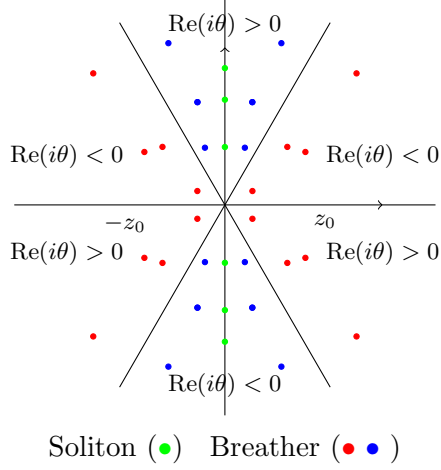
\begin{figure}[H]
\label{sig-table-p}
\caption{signature table-Painleve}
\begin{tikzpicture}[scale=0.7]
\draw [->] (-4,0)--(3,0);
\draw (4,0)--(3,0);
\draw [->] (0,-4)--(0,3);
\draw (0,3)--(0,4);
 \pgfmathsetmacro{\a}{1}
    \pgfmathsetmacro{\b}{1.7320} 
     \draw(-2 , -3.464)--(2 , 3.464);
   \draw (-2 , 3.464)--(2 , -3.464);
    \pgfmathsetmacro{\c}{1.414}
    \pgfmathsetmacro{\d}{2.449} 
    \pgfmathsetmacro{\e}{0.5}
    \pgfmathsetmacro{\f}{0.866} 
 \draw	[fill, green]  (0, 1.1)		circle[radius=0.05];	    
 \draw	[fill, green]  (0, 2)		circle[radius=0.05];	    
 \draw	[fill, green]  (0, 2.6)		circle[radius=0.05];	    
  \draw	[fill, green]  (0, -1.1)		circle[radius=0.05];	    
 \draw	[fill, green]  (0, -2)		circle[radius=0.05];	    
 \draw	[fill, green]  (0, -2.6)		circle[radius=0.05];	    
  \draw	[fill, blue]  (0.5211, 1.953)		circle[radius=0.05];	  
   \draw	[fill, blue]  (-0.5211, 1.953)		circle[radius=0.05];	   
    \draw[fill, blue]  (0.5211, 1.953)		circle[radius=0.05];	  
   \draw	[fill, blue]  (-0.5211, 1.953)		circle[radius=0.05];	   
    \draw[fill, blue]  (0.379, 1.087)		circle[radius=0.05];	  
   \draw	[fill, blue]  (-0.379, 1.087)		circle[radius=0.05];	   
    \draw[fill, blue]  (1.073, 3.074)		circle[radius=0.05];	  
   \draw	[fill, blue]  (-1.073, 3.074)		circle[radius=0.05];	   
   \draw	[fill, blue]  (0.5211, -1.953)		circle[radius=0.05];	  
   \draw	[fill, blue]  (-0.5211, -1.953)		circle[radius=0.05];	   
    \draw[fill, blue]  (0.5211, -1.953)		circle[radius=0.05];	  
   \draw	[fill, blue]  (-0.5211, -1.953)		circle[radius=0.05];	   
    \draw[fill, blue]  (0.379, -1.087)		circle[radius=0.05];	  
   \draw	[fill, blue]  (-0.379, -1.087)		circle[radius=0.05];	   
    \draw[fill, blue]  (1.073, -3.074)		circle[radius=0.05];	  
   \draw	[fill, blue]  (-1.073, -3.074)		circle[radius=0.05];	   
   \draw	[fill, red]  (1.1855, 1.10268)		circle[radius=0.05];	  
   \draw	[fill, red]  (-1.1855, 1.10268)	circle[radius=0.05];	   
    \draw	[fill, red]  (1.529, 1.006)		circle[radius=0.05];	  
   \draw	[fill, red]  (-1.529, 1.006)	circle[radius=0.05];	   
   \draw	[fill, red]  (0.523, 0.2637)		circle[radius=0.05];	  
   \draw	[fill, red]  (-0.523, 0.2637)	circle[radius=0.05];	   
    \draw	[fill, red]  (1.1855, -1.10268)		circle[radius=0.05];	  
   \draw	[fill, red]  (-1.1855, -1.10268)	circle[radius=0.05];	   
    \draw	[fill, red]  (1.529, -1.006)		circle[radius=0.05];	  
   \draw	[fill, red]  (-1.529, -1.006)	circle[radius=0.05];	   
   \draw	[fill, red]  (0.523, -0.2637)		circle[radius=0.05];	  
   \draw	[fill, red]  (-0.523, -0.2637)	circle[radius=0.05];	  
     \draw	[fill, red]  (2.5, 2.5)		circle[radius=0.05];	  
    \draw	[fill, red]  (2.5, -2.5)		circle[radius=0.05];	  
     \draw	[fill, red]  (-2.5, 2.5)		circle[radius=0.05];	  
      \draw	[fill, red]  (-2.5, -2.5)		circle[radius=0.05];	  
   \node [below] at (1.9,0) {\footnotesize $z_0$};
    \node [below] at (-1.9,0) {\footnotesize $-z_0$};
    \node[above]  at (-3, 0.5) {\footnotesize $\text{Re}(i\theta)<0$};
     \node[above]  at (3, 0.5) {\footnotesize $\text{Re}(i\theta)<0$};
      \node[below]  at (-3, -0.5) {\footnotesize $\text{Re}(i\theta)>0$};
     \node[below]  at (3, -0.5) {\footnotesize $\text{Re}(i\theta)>0$};
     \node[above]  at (0, 3) {\footnotesize $\text{Re}(i\theta)>0$};
     \node[below]  at (0, -3) {\footnotesize $\text{Re}(i\theta)<0$};
    \end{tikzpicture}
 \begin{center}
  \begin{tabular}{ccc}
Soliton ({\color{green} $\bullet$})	&	
Breather ({\color{red} $\bullet$} {\color{blue} $\bullet$} ) 
\end{tabular}
 \end{center}
\end{figure}
So we only need the  following upper/lower factorization on $\bbR$:
\begin{equation}
\label{v-ul}
e^{-i\theta\ad\sigma_3}v(z)	=\Twomat{1}{\overline{r(z)}   e^{-2i\theta}}{0}{1} \Twomat{1}{0}{r(z)  e^{2i\theta}}{1}.
						\quad z \in\bbR 
\end{equation}
 Define the following set:
\begin{equation}
\label{B-0-set}
\mathcal{B}_0=\lbrace z_j = \xi_j+i\eta_j: ~  4\eta^2_j-12\xi^2_j  >0 \rbrace 
\end{equation}
and the scalar function:
\begin{equation}
\label{psi-0}
\psi(z) =\left( \prod_{k=1}^{N_1} \dfrac{z- \overline {z_k }}{ z-z_k}\right) \left(\prod_{z_j\in B_0} \dfrac{z-\overline{z_j}}{z-z_j}\right) \left(\prod_{z_j\in B_0} \dfrac{z+z_j}{z+\overline{z_j}}\right) .
\end{equation}
It is straightforward to check that if $m(z;x,t)$ solves Problem \ref{RHP-1}, then the new matrix-valued function $m^{(1)}(z;x,t)=m(z;x,t)\psi(z)^{\sigma_3}$ has the following jump matrices:
\begin{equation}
e^{-i\theta\ad\sigma_3}v^{(1)}(z) 	=\Twomat{1}{\overline{r(z)} \psi^{2} e^{-2i\theta}}{0}{1} \Twomat{1}{0}{r(z)\psi^{-2} e^{2i\theta}}{1},
						\quad z \in\bbR 
\end{equation}
\begin{align}
\label{v-soliton-p}
e^{-i\theta\ad\sigma_3}v^{(1)}(z) = 	\begin{cases}
						\twomat{1}{\dfrac{(1/\psi)'(z_k)^{-2} }{c_k  (z-z_k)}e^{-2i\theta} }{0}{1}	&	z\in \gamma_k, \\
						\\
						\twomat{1}{0}{\dfrac{\psi'( \overline{z_k } )^{-2}}{\overline{c_k}  (z -\overline{z_k })} e^{2i\theta}}{1}
							&	z \in \gamma_k^*,
					\end{cases}
\end{align}
and for $z_j\in \mathcal{B}_0 $
\begin{align}
\label{v-br+}
e^{-i\theta\ad\sigma_3}v^{(1)}(z)= 	\begin{cases}
						\twomat{1}{\dfrac{(1/\psi )'(z_j)^{-2} }{c_j  (z-z_j)}e^{-2i\theta} }{0}{1}	&	z\in \gamma_j, \\
						\\
						\twomat{1}{0}{\dfrac{\psi'( \overline{z_j } )^{-2}}{\overline{c_j}  (z -\overline{z_j })} e^{2i\theta}}{1}
							&	z \in \gamma_j^*,\\
							\\
							\twomat{1}{0}{-\dfrac{\psi'( -{z_j } )^{-2}}{ c_j  (z + z_j ) } e^{2i\theta}}{1}	&	z\in -\gamma_j, \\
						\\
						\twomat{1} {-\dfrac{  (1/\psi )'(  -\overline{z_j } )^{-2} }{ \overline{c_j}  (z +\overline{z_j } ) }  e^{-2i\theta} }{0}{1}
							&	z \in -\gamma_j^*
					\end{cases}
\end{align}						
and for $z_j \in \lbrace z_j \rbrace_{j=1}^{N_2} \setminus \mathcal{B}_0$	
\begin{align}
\label{v-br-}
e^{-i\theta\ad\sigma_3}v^{(1)}(z)= 	\begin{cases}
						\twomat{1}{0}{\dfrac{c_j \psi (z_j)^{-2} }{z-z_j} e^{2i\theta}}{1}	&	z\in \gamma_j, \\
						\\
						\twomat{1}{\dfrac{\overline{c_j} \psi(\overline{z_j})^2 }{z -\overline{z_j}} e^{-2i\theta} }{0}{1}
							&	z \in \gamma_j^*,\\
							\\
							\twomat{1}{\dfrac{-c_j \, \psi(-{z_j})^{2} }{z + z_j} e^{-2i\theta}  }{0}{1}	&	z\in -\gamma_j, \\
						\\
						\twomat{1}{0}{\dfrac{-\overline{c_j} \, \psi(-\overline{z_j})^{-2}e^{2i\theta} }{z +\overline{z_j }}}{1}
							&	z \in -\gamma_j^*.
					\end{cases}
\end{align}		
By the signature table Figure \ref{sig-table-p}, we see that all entries in \eqref{v-br+}-\eqref{v-br-} decay exponentially as $t\to\infty$, so we are allowed to reduce the RHP to a problem on $\bbR$ following the same argument in the proof of Proposition \ref{Prop:expo}.
Now we carry out the following scaling:
\begin{equation}
\label{scale}
z\to \zeta t^{-1/3}
\end{equation}
and \eqref{v-ul} becomes
\begin{equation}
\label{v-scale1}
\Twomat{1}{\overline{r(\zeta t^{-1/3 } ) } \psi^2(\zeta t^{-1/3}) e^{-2i\theta (\zeta t^{-1/3 } ) } }{0}{1} \Twomat{1}{0}{r(\zeta t^{-1/3 } )   \psi^{-2}(\zeta t^{-1/3}) e^{2i\theta (\zeta t^{-1/3 } )  } }{1},
						\quad z \in\bbR 
\end{equation}
where
						$$ \theta (\zeta t^{-1/3 } )=4 \zeta^3+x \zeta t^{-1/3}=4( \zeta^3-3\tau^{2/3} \zeta).$$
Note that the stationary points now become $\pm z_0 t^{1/3}$.

We then study the scaled Riemann-Hilbert problem with jump matrix \eqref{v-scale1}. We will again perform contour deformation and write the solution as a product of solution to a $\dbar$-problem and a "localized" Riemann-Hilbert problem.

\begin{figure}[H]
\caption{$\Sigma^{(1)}-\text{scale}$}
\vskip 15pt
\begin{tikzpicture}[scale=0.7]
\draw[thick]		(5, 3) -- (4,2);						
\draw[->,thick,>=stealth] 		(2,0) -- (4,2);		
\draw[thick] 			(-2,0) -- (-4,2); 				
\draw[->,thick,>=stealth]  	(-5,3) -- (-4,2);	
\draw[->,thick,>=stealth]		(-5,-3) -- (-4,-2);							
\draw[thick]						(-4,-2) -- (-2,0);
\draw[thick,->,>=stealth]		(2,0) -- (4,-2);								
\draw[thick]						(4,-2) -- (5,-3);
\draw[thick]		(0,0)--(2,0);
\draw[thick,->,>=stealth] (-2,0) -- (0, 0);
\draw	[fill]							(-2,0)		circle[radius=0.1];	
\draw	[fill]							(2,0)		circle[radius=0.1];
\draw [dashed] (2,0)--(6,0);
\draw [dashed] (-6,0)--(-2,0);
\node[below] at (-2,-0.35)			{$-z_0 t^{1/3}$};
\node[below] at (2,-0.35)			{$z_0 t^{1/3}$};
\node[right] at (5,3)					{$\Sigma^{(1)}_1$};
\node[left] at (-5,3)					{$\Sigma^{(1)}_2$};
\node[left] at (-5,-3)					{$\Sigma^{(1)}_3$};
\node[right] at (5,-3)				{$\Sigma^{(1)}_4$};
\node[above] at (4.5,0)           {$\Omega_1$};
\node[above] at (-4.5,0)           {$\Omega_2$};
\node[below] at (-4.5,0)           {$\Omega_3$};
\node[below] at (4.5,0)           {$\Omega_4$};
\end{tikzpicture}
\label{fig:contour-scale-1}
\end{figure}
For brevity, we only discuss the $\dbar$-problem in $\Omega_1$. In $\Omega_1$, we write 
$$\zeta= u+z_0 t^{1/3} +iv $$ 
then
\begin{align*}
\text{Re} (2i\theta (\zeta t^{-1/3})) &=8\left( -3(u+z_0 t^{1/3})^2 v +v^3 +3\tau^{2/3}v\right) \\
      & \leq 8 \left( -3u^2v -6 uv z_0 t^{1/3}  +v^3 \right)\\
      &\leq -16u^2 v
\end{align*}
	\begin{align*}
	R_1	&=	\begin{cases}
						\twomat{0}{0}{r(\zeta t^{-1/3} )  \psi^{-2}(\zeta t^{-1/3}) e^{2i\theta (\zeta t^{-1/3 } )  }  }{0}		
								&	\zeta \in (z_0t^{1/3},\infty)\\[10pt]
								\\
						\twomat{0}{0}{r( z_0 )  \psi^{-2}(\zeta t^{-1/3})   (1-\Xi_\calZ)e^{2i\theta (\zeta t^{-1/3 } )  }  }{0}	
								&	\zeta	\in \Sigma_1
					\end{cases}
	\end{align*}
and the interpolation is given by 
$$ \left( r(z_0)+ \left( r\left( \text{Re}\zeta t^{-1/3} \right) -r(z_0)  \right) \cos 2\phi  \right) \psi^{-2}(\zeta t^{-1/3})   (1-\Xi_\calZ)$$
So we arrive at the $\dbar$-derivative in $\Omega_1$ in the $\zeta$ variable:
\begin{align}
\dbar R_1	&= \left( {t^{-1/3}} r'\left( u t^{-1/3} \right) \cos 2\phi- 2\dfrac{ r(ut^{-1/3} )  -r(z_0)  }{  \left\vert \zeta-z_0 t^{1/3} \right\vert   } e^{i\phi} \sin 2\phi  \right)\psi^{-2}(\zeta t^{-1/3})   e^{2i\theta}\\
    &\quad \times  (1-\Xi_\calZ)\\
    &\quad- \left( r(z_0)+ \left( r\left( \text{Re}\zeta t^{-1/3} \right) -r(z_0)  \right) \cos 2\phi  \right) t^{-1/3} \dbar (\Xi_\calZ(\zeta t^{-1/3}) ) \psi^{-2}(\zeta t^{-1/3})  e^{2i\theta}
\end{align}
\begin{equation}
\label{R1.bd1}
\left| \dbar R_1 e^{ 2i\theta}  \right| 	\lesssim\left( |  t^{-1/3} r'\left( u t^{-1/3} \right) | +\dfrac{\norm{r'}{L^2} }{ t^{1/3} |\zeta t^{-1/3}-z_0  |^{1/2} } + t^{-1/3} \dbar (\Xi_\calZ(\zeta t^{-1/3}) )  \right) e^{-16u^2 v}.
\end{equation}
We proceed as in the previous section and study the integral equation related to the $\dbar$ problem. Setting $z=\alpha+i\beta$ and $\zeta=(u+z_0t^{1/3})+iv$, the region $\Omega_1$ corresponds to $u\geq v \geq 0 $. We decompose the integral operator into three parts:
$$
 \int_{\Omega_1}  \dfrac{1}{|z-\zeta|} |W(\zeta)| \, d\zeta  \lesssim  I_1 + I_2 +I_3
$$
where
\begin{align*}
I_1 	&=	\int_0^\infty \int_v^\infty \dfrac{1}{|z-\zeta|} \left\vert t^{-1/3} r'\left( u t^{-1/3} \right)   \right\vert e^{-16u^2v} \, du \, dv, \\[5pt]
I_2	&=	\int_0^\infty \int_v^\infty \frac{1}{|z-\zeta|}  \dfrac {1} { t^{1/3} \left| u t^{-1/3} +ivt^{-1/3}  \right|^{1/2} } e^{-16u^2v} \, du \, dv,\\[5pt]
I_3 	&=	\int_0^\infty \int_v^\infty \dfrac{1}{|z-\zeta|} \left\vert t^{-1/3} \dbar (\Xi_\calZ(\zeta t^{-1/3}) )   \right\vert e^{-16u^2v} \, du \, dv. 
\end{align*}
We first note that 
$$ \left( \int_\bbR \left\vert t^{-1/3} r'\left( u t^{-1/3} \right)   \right\vert^2 du \right)^{1/2} = t^{-1/6}\norm{r'}{L^2}$$
Using this and the following estimate from \cite[proof of Proposition D.1]{BJM16} 
\begin{equation}
\label{BJM16.bd1}
\norm{\frac{1}{|z-\zeta|}}{L^2(v,\infty)} \leq \frac{\pi^{1/2}}{|v-\beta|^{1/2}}.
\end{equation}
 and Schwarz's inequality on the $u$-integration we may bound $I_1$  by constants times
$$
 t^{-1/6} \norm{r'}{2} \int_0^\infty \frac{1}{|v-\beta|^{1/2}} e^{-v^3} \, dv \lesssim t^{-1/6}.
$$
For $I_2$, taking $p>4$ and $q$ with $1/p+1/q=1$, we estimate
\begin{align*}
\norm{ \dfrac {1} { t^{1/3} \left| u t^{-1/3} +ivt^{-1/3}  \right|^{1/2} } }{ L^p (v, \infty) } &\leq  \left( \int_v ^\infty t^{-p/3} \left( \dfrac{  1  }{  (u t^{-1/3} )^2 + (v t^{-1/3})^2  } \right)^{p/4} du \right)^{1/p}\\ 
   &= t^{(3-p)/(3p)} \left(  \int_v ^\infty \left( \dfrac{  1  }{  (u t^{-1/3} )^2 + (v t^{-1/3})^2  } \right)^{p/4} d (ut^{-1/3}) \right)^{1/p}\\
   &=t^{(3-p)/(3p)} \left(  \int_{v' }^\infty \left( \dfrac{  1  }{  (u')^2 + (v')^2  } \right)^{p/4} du' \right)^{1/p}\\
   &\leq c t^{(3-p)/(3p)} v'^{(1/p-1/2)}\\
   &= c t ^{(2/(3p)-1/6 )} v^{1/p-1/2}. 
\end{align*}
Now by \eqref{BJM16.bd1} and an application of the H\"older inequality we get
\begin{align*}
|I_2|  &\leq  \int_0^\infty \norm{ \dfrac {1} { t^{1/3} \left| u t^{-1/3} +ivt^{-1/3}  \right|^{1/2} } }{ L^p (v, \infty) }  \norm{ \dfrac{1}{|z-\zeta|}}{L^q (v, \infty)} e^{-16v^3} dv\\
   &\leq c \int_0^\infty t ^{(2/(3p)-1/6 )} v^{1/p-1/2} \left\vert v-\beta \right\vert^{1/q-1} e^{-16v^3} dv\\
   &\leq c t ^{(2/(3p)-1/6 )}.
\end{align*}
The estimate on $I_3$ is similar to that of $I_1$ and
$$|I_3|\leq c t^{-1/6} .$$
This proves that 
$$
 \int_{\Omega_1}  \dfrac{1}{|z-\zeta|} |W(\zeta)| \, d\zeta  \lesssim  t ^{(2/(3p)-1/6 )} 
$$
for all $4 <p<\infty$.  We now show that 
$$
 \int_{\Omega_1} |W(\zeta)| \, d\zeta  \lesssim  t ^{(2/(3p)-1/6 )} 
.$$
Again we decompose the integral above into three parts
\begin{align*}
I_1 	&=	\int_0^\infty \int_v^\infty  \left\vert t^{-1/3} r'\left( u t^{-1/3} \right)   \right\vert e^{-16u^2v} \, du \, dv \\[5pt]
I_2	&=	\int_0^\infty \int_v^\infty \dfrac {1} { t^{1/3} \left| u t^{-1/3} +ivt^{-1/3}  \right|^{1/2} } e^{-16u^2v} \, du \, dv.\\[5pt]
I_3 	&=	\int_0^\infty \int_v^\infty  \left\vert t^{-1/3} \dbar \Xi_{\calZ} \left( \zeta t^{-1/3} \right)   \right\vert e^{-16u^2v} \, du \, dv 
\end{align*}
By Cauchy-Schwarz inequality:
\begin{align*}
I_1 &\leq \int_0^\infty t^{-1/6} \norm{r'}{2} \left(  \int_v^\infty e^{-16u^2 v} du\right)^{1/2} dv\\
 &\leq c t^{-1/6} \int_0^\infty \dfrac{e^{-16 v^3} }{\sqrt[4]{v}}dv\\
 &\leq c t^{-1/6}.
\end{align*}
By H\"older's  inequality:
\begin{align*}
I_2 &\leq c t^{(2/(3p)-1/6 )} \int_0^\infty  v^{1/p-1/2} \left( \int_v^\infty e^{-16qu^2 v} du \right)^{1/q} dv\\
     &\leq c t^{(2/(3p)-1/6 )} \int_0^\infty v^{3/(2p)-1} e^{-16v^3}dv\\
     &\leq c t^{(2/(3p)-1/6 )}. 
\end{align*}
Again the estimate on $I_3$ is similar to that of $I_1$ and
$$I_3\leq c t^{-1/6} .$$
We can apply the fundamental theorem of calculus to get
$$r(\zeta t^{-1/3}) \psi(\zeta t^{-1/3})^{-2} e^{2i\theta}-r(0)e^{2i\theta} \leq \left\vert  \dfrac{\zeta}{t^{1/6}} e^{8i (\zeta^3-3 \tau^{2/3}\zeta)}\right\vert.$$
Given the fact that $z_0 t^{1/3}=\tau^{1/3}\leq (M')^{1/3}$, we have that
$$\norm{  \dfrac{\zeta}{t^{1/6}} e^{8i (\zeta^3-3 \tau^{2/3}\zeta)}}{L^1\cap L^2\cap L^\infty} \lesssim t^{-1/6}.$$
Also notice that $\psi(0)^{\pm}=1$ so we have reduce the problem to a problem on the following contour
\begin{figure}[H]
\caption{$\Sigma^{(2)}$-Scale }
\vskip 15pt
\begin{tikzpicture}[scale=0.7]
\draw[thick]		(3, 3) -- (2,2);						
\draw[->,thick,>=stealth] 		(0,0) -- (2,2);		
\draw[thick] 			(0,0) -- (-2,2); 				
\draw[->,thick,>=stealth]  	(-3,3) -- (-2,2);	
\draw[->,thick,>=stealth]		(-3,-3) -- (-2,-2);							
\draw[thick]						(-2,-2) -- (0,0);
\draw[thick,->,>=stealth]		(0,0) -- (2,-2);								
\draw[thick]						(2,-2) -- (3,-3);
\draw	[fill]							(0,0)		circle[radius=0.1];	
\node [below] at (0,0) {0};
\draw [dashed] (0,0)--(4,0);
\draw [dashed] (-4,0)--(0,0);
\node[right] at (3,3)					{$\Sigma^{(2)}_1$};
\node[left] at (-3,3)					{$\Sigma^{(2)}_2$};
\node[left] at (-3,-3)					{$\Sigma^{(2)}_3$};
\node[right] at (3,-3)				{$\Sigma^{(2)}_4$};
\end{tikzpicture}
\label{fig:Painleve}
\end{figure}
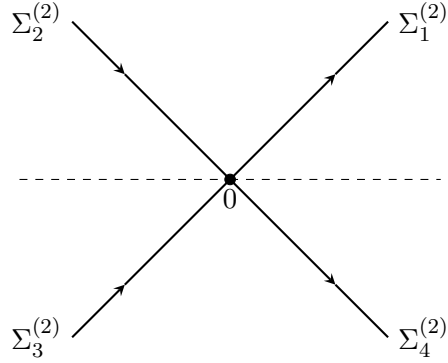
with jump matrices:
\begin{align}
\label{matrices-Painleve}
e^{-i\theta\ad\sigma_3 } v^{(2)}(\zeta) &=e^{-4i\left(\zeta^3+\left( x/(4 t^{1/3}) \right)  \zeta \right) \ad\sigma_3 }  \twomat{1} {0} { {r(0)} }{1} , \quad \zeta\in \Sigma^{(2)}_1\cup \Sigma^{(2)}_2 \\
\nonumber
                                                                                                                   &=e^{-4i\left(\zeta^3+\left( x/(4 t^{1/3}) \right)  \zeta \right) \ad\sigma_3 } \twomat{1}{\overline{r(0)}}{0}{1}, \quad \zeta\in \Sigma^{(2)}_3 \cup \Sigma^{(2)}_4
\end{align}
In \cite{GPR} and \cite{HN1} the leading order term of the focusing mKdV is RHP is given by the solution to the Painlev\'e II equation:
\begin{equation}
\label{Painleve-1}
P''(s)-sP(s)+2P^3(s)=0.
\end{equation}
In fact, by changing $P(s)\mapsto  -iP(s)$, we have the solution to the following Painlev\'e II equation:
\begin{equation}
\label{Painleve-2}
P''(s)-sP(s)-2P^3(s)=0.
\end{equation}
The RHP on $\Sigma^{(2)}$ is related to Equation \eqref{Painleve-2}. We now follows the argument of \cite[Section 5]{DZ93} and \cite{DZ95 } to obtain the long-time asymptotic formula in Region II ($x<0$).
\begin{figure}[H]
\caption{Painlev\'e six ray}
\label{fig:six-ray}
\begin{tikzpicture}[scale=0.7]
\draw [->] (0,0)--(3,0);
\draw [->] (0,0)--(-3,0);
\draw (-3,0)--(-4,0);
\draw (4,0)--(3,0);
\draw [->](0,0)--(-1 , -1.732);
     \draw (-1 , -1.732)--(-2, -3.464);
    \draw [->](0,0)--(1 , 1.732);
     \draw (1 , 1.732)--(2, 3.464);
       \draw [->](0,0)--(-1 , 1.732);
     \draw (-1 , 1.732)--(-2, 3.464);
       \draw [->](0,0)--(1 , -1.732);
     \draw (1 , -1.732)--(2, -3.464);
\node[above]  at (-1, 0.2) {\footnotesize $\Omega^P_3$};
     \node[above]  at (1, 0.2) {\footnotesize $\Omega^P_1$};
      \node[below]  at (-1, -0.1) {\footnotesize $ \Omega^P_4 $};
     \node[below]  at (1, -0.1) {\footnotesize $\Omega^P_6$};
     \node[above]  at (0,0.5) {\footnotesize $\Omega^P_2$};
     \node[below]  at (0, -0.5) {\footnotesize $\Omega^P_5$};
     \node[right] at (2, 3.464) {$\Sigma^P_1$};
      \node[left] at (-2, 3.464) {$\Sigma^P_2$};
       \node[left] at (-4, 0) {$\Sigma^P_3$};
       \node[left] at (-2, -3.464) {$\Sigma^P_4$};
        \node[right] at (2, -3.464) {$\Sigma^P_5$};
         \node[right] at (4, 0) {$\Sigma^P_6$};
    \end{tikzpicture}
\end{figure}
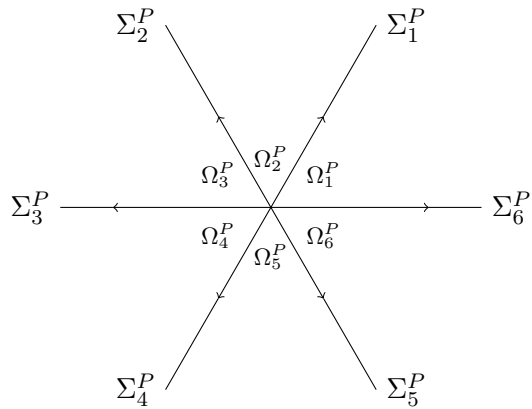
Associate to each ray $\Sigma_i^P$, $i=1,2,...,6$, a jump matrix independent of $\zeta$
\begin{align}
\label{jump-six}
\begin{cases}
& S_1 =\twomat{1}{0}{\tq}{1}, \quad S_2=\twomat{1}{\tr}{0}{1}, \quad  S_3=\twomat{1}{0}{\tp}{1}\\
\\
& S_4=\twomat{1}{\tq}{0}{1}, \quad S_5=\twomat{1}{0}{\tr}{1}, \quad S_6=\twomat{1}{\tp}{0}{1}
\end{cases}
\end{align}
where $\tp,\tq,\tr$ satisfies the constraint
\begin{equation}
\label{constraint}
\tp+\tq+\tr+\tp\tq\tr=0.
\end{equation}
To construct \eqref{jump-six}, we let $m^{(2)}(\zeta)$ denote the solution to the RHP on Figure \ref{fig:Painleve} and set
\begin{align}
m^{(3)}(\zeta) &=m^{(2)}(\zeta),\quad & \zeta \in  \Omega^P_1\cup \Omega^P_2\cup\Omega^P_4\cup \Omega^P_4 \\
                      &=m^{(2)}(\zeta)e^{-i\theta\ad\sigma_3 } v^{(2)}(\zeta)  ,\quad  &\zeta \in \Omega_3^P\\
                      &=m^{(2)}(\zeta) \left( e^{-i\theta\ad\sigma_3 } v^{(2)}(\zeta)\right)^{-1}  ,\quad  &\zeta \in \Omega_6^P.
\end{align}
From \eqref{minus} we deduce that $r(0)=- \overline{r(0)}$, so $r(0)$ is purely imaginary.
Setting
$$\tp=r(0),\quad \tq=\overline{r(0)}$$
from \eqref{constraint} we deduce
$$\tr=-(\tp+\tq)/(1+\tp\tq)=0.$$
We then observe that $\Psi =m^{(3)}(\zeta/ 3^{1/3}) e^{ -\left( (4i/3) \zeta^3   +is\zeta   \right)   \sigma}$ satisfies the jump \eqref{jump-six} on $\Sigma^P$ given by Figure \ref{fig:six-ray} with
$$\Psi_{i+1}(s, \zeta)= \Psi_{i}(s, \zeta)S_i,  \quad 1\leq i\leq 6$$
with $s=x/t^{1/3}$. 
From \cite{DZ95} we know that $\Psi$ can be uniquely obtained and 
$\Psi$ solves the following linear problem
$$\dfrac{d\Psi}{d\zeta}=\twomat{-4i\zeta^2-is-2iP^2}{4iP\zeta-2P'}{-4iP\zeta-2P'}{4i\zeta^2+is+2iP^2} \Psi$$
where $P(s)$ is a \emph{ purely imaginary} solution to \eqref{Painleve-2}. Indeed, we have
$$P=P(x/t^{1/3}, r(0))=\lim_{\zeta\to \infty} 2i \zeta \left( \Psi e^{ \left( (4i/3) \zeta^3   +is\zeta   \right)   \sigma} -I\right)_{12}.$$
Finally, by sending $P\mapsto -iP$ and recalling the scaling \eqref{scale} and  combining the error term resulting from the $\dbar$-extension, we arrive at the long time asymptotics in Region II:
\begin{equation}
\label{asym-II}
u(x,t)=\dfrac{1}{(3t)^{1/3}}P\left( \dfrac{x}{ (3t)^{1/3} } \right)+\mathcal{O} \left(  t^{ 2/(3p)-1/2 } \right)
\end{equation}
where $4<p<\infty$ and $P$ is a \emph{real} solution of the Painlev\'e II equation 
$$P''(s)-sP(s)+2P^3(s)=0.$$

\subsection{Region III} 
In this region, $|x/t|=\mathcal{O}(1)$ as $t\to \infty$ and $x>0$. We have the stationary points 
$$\pm z_0=\pm \sqrt{\dfrac{-x }{ 12 t }}=\pm i  \sqrt{\dfrac{ |x| }{ 12 t }}$$
which is purely imaginary and have a fixed distance from the real axis.  The signature table of the phase function $i\theta$ is as follows:
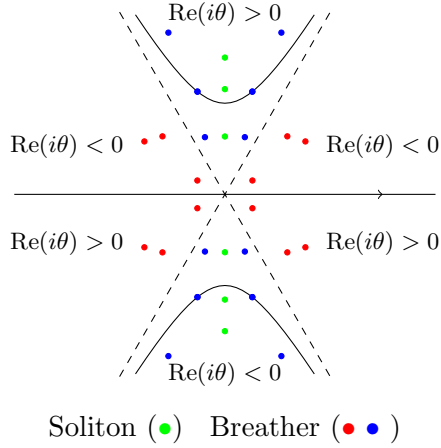
\begin{figure}[H]
\caption{Signature-solitons}
\begin{tikzpicture}[scale=0.7]
\draw [->] (-4,0)--(3,0);
\draw (4,0)--(3,0);
 \pgfmathsetmacro{\a}{1}
    \pgfmathsetmacro{\b}{1.7320} 
     \draw[dashed] (-2 , -3.464)--(2 , 3.464);
   \draw[dashed] (-2 , 3.464)--(2 , -3.464);
    \pgfmathsetmacro{\c}{1.414}
    \pgfmathsetmacro{\d}{2.449} 
    \pgfmathsetmacro{\e}{0.5}
    \pgfmathsetmacro{\f}{0.866} 
    \draw plot[domain=-1.3:1.3] ( {\a*sinh(\x)}, {\b*cosh(\x)} );
   \draw plot[domain=-1.3:1.3] ( {-\a*sinh(\x)}, {-\b*cosh(\x)} );
 \draw	[fill, green]  (0, 1.1)		circle[radius=0.05];	    
 \draw	[fill, green]  (0, 2)		circle[radius=0.05];	    
 \draw	[fill, green]  (0, 2.6)		circle[radius=0.05];	    
  \draw	[fill, green]  (0, -1.1)		circle[radius=0.05];	    
 \draw	[fill, green]  (0, -2)		circle[radius=0.05];	    
 \draw	[fill, green]  (0, -2.6)		circle[radius=0.05];	    
  \draw	[fill, blue]  (0.5211, 1.953)		circle[radius=0.05];	  
   \draw	[fill, blue]  (-0.5211, 1.953)		circle[radius=0.05];	   
    \draw[fill, blue]  (0.5211, 1.953)		circle[radius=0.05];	  
   \draw	[fill, blue]  (-0.5211, 1.953)		circle[radius=0.05];	   
    \draw[fill, blue]  (0.379, 1.087)		circle[radius=0.05];	  
   \draw	[fill, blue]  (-0.379, 1.087)		circle[radius=0.05];	   
    \draw[fill, blue]  (1.073, 3.074)		circle[radius=0.05];	  
   \draw	[fill, blue]  (-1.073, 3.074)		circle[radius=0.05];	   
   \draw	[fill, blue]  (0.5211, -1.953)		circle[radius=0.05];	  
   \draw	[fill, blue]  (-0.5211, -1.953)		circle[radius=0.05];	   
    \draw[fill, blue]  (0.5211, -1.953)		circle[radius=0.05];	  
   \draw	[fill, blue]  (-0.5211, -1.953)		circle[radius=0.05];	   
    \draw[fill, blue]  (0.379, -1.087)		circle[radius=0.05];	  
   \draw	[fill, blue]  (-0.379, -1.087)		circle[radius=0.05];	   
    \draw[fill, blue]  (1.073, -3.074)		circle[radius=0.05];	  
   \draw	[fill, blue]  (-1.073, -3.074)		circle[radius=0.05];	   
   \draw	[fill, red]  (1.1855, 1.10268)		circle[radius=0.05];	  
   \draw	[fill, red]  (-1.1855, 1.10268)	circle[radius=0.05];	   
    \draw	[fill, red]  (1.529, 1.006)		circle[radius=0.05];	  
   \draw	[fill, red]  (-1.529, 1.006)	circle[radius=0.05];	   
   \draw	[fill, red]  (0.523, 0.2637)		circle[radius=0.05];	  
   \draw	[fill, red]  (-0.523, 0.2637)	circle[radius=0.05];	   
    \draw	[fill, red]  (1.1855, -1.10268)		circle[radius=0.05];	  
   \draw	[fill, red]  (-1.1855, -1.10268)	circle[radius=0.05];	   
    \draw	[fill, red]  (1.529, -1.006)		circle[radius=0.05];	  
   \draw	[fill, red]  (-1.529, -1.006)	circle[radius=0.05];	   
   \draw	[fill, red]  (0.523, -0.2637)		circle[radius=0.05];	  
   \draw	[fill, red]  (-0.523, -0.2637)	circle[radius=0.05];	   
   \node[above]  at (-3, 0.5) {\footnotesize $\text{Re}(i\theta)<0$};
     \node[above]  at (3, 0.5) {\footnotesize $\text{Re}(i\theta)<0$};
      \node[below]  at (-3, -0.5) {\footnotesize $\text{Re}(i\theta)>0$};
     \node[below]  at (3, -0.5) {\footnotesize $\text{Re}(i\theta)>0$};
      \node[above]  at (0, 3) {\footnotesize $\text{Re}(i\theta)>0$};
       \node[below]  at (0, -3) {\footnotesize $\text{Re}(i\theta)<0$};
    \end{tikzpicture}
 \begin{center}
  \begin{tabular}{ccc}
Soliton ({\color{green} $\bullet$})	&	
Breather ({\color{red} $\bullet$} {\color{blue} $\bullet$} ) 
\end{tabular}
 \end{center}
\end{figure}
We write 
$$ \text{Re} i\theta(x,t; z)=t\left( 4(-3u^2v+v^3)-\dfrac{x}{t} v  \right) $$ 
then it is  clear that if we set $x/t=v_{b_j}=4\eta_j^2-12\xi_j^2$, then $ \text{Re} i\theta(x,t; z_j)=0$. We again choose the frame of a single breather/soliton:
$$x/t=4\eta_\ell^2-12\xi_\ell^2.$$
 Define the following set
\begin{equation}
\label{B-set+}
\mathcal{B}_\ell=\lbrace z_j = \xi_j+i\eta_j: ~  4\eta^2_j-12\xi^2_j  >4\eta^2_\ell-12\xi^2_\ell\rbrace \cup \lbrace z_k = i\zeta_k: ~  4\zeta_k^2 >4\eta^2_\ell-12\xi^2_\ell\rbrace .
\end{equation}
and scalar function
\begin{equation}
\label{psi-l}
\psi(z) =\left( \prod_{z_k\in \mathcal{B}_\ell } \dfrac{z- \overline {z_k }}{ z-z_k}\right) \left(\prod_{z_j\in \mathcal{B}_\ell } \dfrac{z-\overline{z_j}}{z-z_j}\right) \left(\prod_{z_j\in \mathcal{B}_\ell} \dfrac{z+z_j}{z+\overline{z_j}}\right) .
\end{equation}

We follow the strategy of the previous subsection. If $m(z;x,t)$ solves Problem \ref{RHP-1}, then the new matrix-valued function $m^{(1)}(z;x,t)=m(z;x,t)\psi(z)^{\sigma_3}$ has exponentially decaying jumps across all small circles except  $\pm \gamma_\ell\cup \pm \gamma_\ell^*$. Also we can deform $\bbR$ as follows:
\begin{figure}[H]
\caption{$\Sigma^{(4)}-solitons$}
\vskip 15pt
\begin{tikzpicture}[scale=0.7]
\draw[thick]		(2, 3) -- (1,2);						
\draw[->,thick,>=stealth] 		(0,1) -- (1,2);		
\draw[thick] 			(-1,2) -- (0, 1); 				
\draw[->,thick,>=stealth]  	(-2,3) -- (-1,2);	
\draw[->,thick,>=stealth]		(-2,-3) -- (-1,-2);							
\draw[thick]						(-1,-2) -- (0, -1);
\draw[thick,->,>=stealth]		(2, -3) -- (1,-2);								
\draw[thick]						(1,-2) -- (0,-1);
\draw[thick]		(0,0)--(4,0);
\draw[thick,->,>=stealth] (-4,0) -- (0, 0);
\draw [dashed] (0,1)--(2,1);
\draw [dashed] (0,1)--(-2,1);
\node[right] at (2,3)					{$\Sigma^{(4)}_1$};
\node[right] at (4,0)					{$\Sigma^{(4)}_2$};
\node[right] at (2,-3)				{$\Sigma^{(4)}_3$};
\node[below] at (0,3)           {$\Omega_1$};
\node[above] at (2.5,0)           {$\Omega_2$};
\node[below] at (2.5,0)           {$\Omega_3$};
\node[above] at (0,-3)           {$\Omega_4$};
\node[below] at (0,1.1)      {$ih $};
\node[above] at (0,-1)      {$-ih$};
\node[above] at (-2, 0)  {(1)};
\node[above] at (1, 1)  {(2)};
\node[above] at (-1, 1)  {(3)};
 \pgfmathsetmacro{\a}{1}
    \pgfmathsetmacro{\b}{1.7320} 
 \draw plot[domain=-1.3:1.3] ( {\a*sinh(\x)}, {\b*cosh(\x)} );
  \draw plot[domain=-1.3:1.3] ( {-\a*sinh(\x)}, {-\b*cosh(\x)} );
\end{tikzpicture}
\label{fig:contour-scale-IV}
\end{figure}
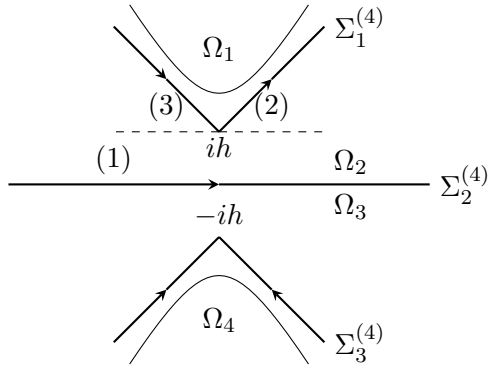	
where $ h$ is chosen such that  $h>0$ and 
\begin{equation}
\label{h-bound}
4\eta_\ell^2-12\xi_\ell^2 -12h^2=c>0.
\end{equation}
On $\Sigma^{(4)}_1$ we set
\begin{align*}
	R_1	&=	\begin{cases}
						\twomat{0}{0}{r( z )  \psi^{-2}( z ) e^{2i\theta  }  }{0}		
								&	z \in \bbR \\[10pt]
								\\
						\twomat{0}{0}{r( 0 )  \psi^{-2}(z)   (1-\Xi_\calZ)e^{2i\theta }  }{0}	
								&	z	\in \Sigma_1^{(4)}
					.\end{cases}
	\end{align*}
We only study $\Omega_2$. In Part $(1)$ of $\Omega_2$, we extend $r( z)=r(\text{Re}z)$. Also in this region, by \eqref{h-bound}
\begin{align*}
\text{Re} (2i\theta (z ) ) &=2t\left( 4(-3u^2v+v^3)-\dfrac{x}{t} v  \right) \\
& \leq  -24u^2v t +  2\left(4 h^2 -\dfrac{x}{t}\right)vt\\
& \leq  -24u^2 v t - 2cvt.
\end{align*}
We now integrate and find that
\begin{align}
\label{error-(1)}
 \int_{(1)} \left\vert  r'(u ) e^{2i\theta(z)} \right\vert dz
 & =\int_0^\eta \int_{-\infty}^{\infty} \left\vert  r'(u) e^{-(24u^2 v  + 2cv)t} \right\vert dudv\\
 \nonumber
 & \lesssim \int_0^\infty \dfrac{e^{- 2cvt}}{\sqrt{vt}} dv\\
 & \lesssim t^{-1}.
\end{align}
In Part $(2)$ of $\Omega_2$, we interpolate
$$ \left( r(0)+ \left( r\left( \text{Re}z  \right) -r(0)  \right) \cos 2\phi  \right) \psi^{-2}(z)   (1-\Xi_\calZ)$$
and calculate 
\begin{align*}
\dbar R_1	&= \left(  r'\left( u  \right) \cos 2\phi- 2\dfrac{ r(u)  -r(0)  }{  \left\vert z\right\vert   } e^{i\phi} \sin 2\phi  \right)\psi^{-2}( z )   e^{2i\theta}\\
    &\quad \times  (1-\Xi_\calZ)\\
    &\quad- \left( r( 0)+ \left( r\left( u \right) -r(0)  \right) \cos 2\phi  \right)  \dbar (\Xi_\calZ(z ) ) \psi^{-2}(z )  e^{2i\theta}
\end{align*}
Also in this region, changing variable $v\mapsto v+h$,  from \eqref{h-bound} and the fact that $u\geq v\geq 0$,
\begin{align*}
\text{Re} (2i\theta (z ) ) &=2t\left( 4(-3u^2(v+h)+(v+h)^3)-\dfrac{x}{t} (v +h) \right) \\
& \leq  2t\left( -8u^2v+ \left(12h^2-\dfrac{x}{t}  \right) v  + \left(4h^2-\dfrac{x}{t} \right ) h \right)\\
& \leq  -16u^2 v t - 2cvt.
\end{align*}
Since
\begin{equation}
\label{R1-bd-soliton}
\left\vert W(z) \right\vert=\left| \dbar R_1 e^{2i\theta}  \right| 	\lesssim\left( |  r'\left( u  \right) | +\dfrac{\norm{r'}{L^2} }{  |z  |^{1/2} } +\dbar (\Xi_\calZ( z ) )  \right) e^{-16u^2 vt-2cvt}
\end{equation}
we still have 
$$\int_{(2)} |W(z)|dz=I_1+I_2+I_3$$
with
\begin{align*}
I_1 	&=	\int_0^\infty \int_{v}^\infty  \left\vert  r'\left( u \right)   \right\vert e^{-16u^2vt-2cvt} \, du \, dv ,\\[5pt]
I_2	&=	\int_0^\infty \int_{v}^\infty \dfrac {1} {  \left| u +i(v+h) \right|^{1/2} } e^{-16u^2vt-2cvt} \, du \, dv ,\\[5pt]
I_3 	&=	\int_0^\infty \int_{v}^\infty  \left\vert \dbar \Xi_{\calZ} \left( u+i(v+h) \right)   \right\vert e^{-16u^2vt-2cvt} \, du \, dv .
\end{align*}
Direct calculation gives
\begin{equation}
\label{error-soliton}
\int_{(2)} |W(z)|dz \lesssim t^{-1}.
\end{equation}
Notice that in Figure \ref{fig:contour-scale-IV}, we have deformed $\bbR$ into $\bbC^\pm$ So the jump matrices across $\Sigma_1^{(4)}\cup \Sigma_3^{(4)}$ enjoy the property of exponential decay as $t\to +\infty$. Thus the reflection coefficient $r$ makes no contribution to the leading order term in Region III. If we choose the frame $x/t=4\eta^2_\ell- 12\xi^2_\ell >0$, then by  solving Problem \ref{prob:mkdv.br} (with $\delta$ replaced by $\psi$ ) we obtain the following asymptotic formula in Region III: 
\begin{equation}
\label{asym-III-1}
u(x, t)=-4\dfrac{\eta_\ell}{\xi_\ell }\dfrac{ \xi_\ell \cosh(\nu_2+\tomega_2 ) \sin(\nu_1+\tomega_1 )+ \eta_\ell \sinh(\nu_2+\tomega_2  ) \cos(\nu_1+\omega_1 )}{\cosh^2(\nu_2+\tomega_2  ) + (\eta_\ell/\xi_\ell )^2 \cos^2(\nu_1+\tomega_1 )} + \mathcal{O }\left( t^{-1}  \right)
\end{equation}
with
\begin{align*}
\nu_1 &=2\xi_\ell (x+4(\xi^2_\ell -3\eta^2_\ell )t )\\
\nu_2 &=2\eta_\ell (x -4(\eta^2_\ell -3\xi^2_\ell )t ).
\end{align*}
And
\begin{align}
\tan \tomega_1&= \dfrac{\tB\xi_\ell-\tA\eta_\ell}{\tA\xi_\ell+\tB\eta_\ell}\\
e^{-\tomega_2} &= \left\vert \dfrac{\xi_\ell }{ 2\eta_\ell } \right\vert \sqrt{\dfrac{\tA^2+\tB^2}{\xi^2_\ell+\eta^2_\ell }  }
\end{align}
where we set 
$$\tilde{c}_\ell=c_\ell \psi(z_\ell )^{-2}=\tA+i\tB .$$
Similarly, if we choose the frame $x/t=4\zeta^2_\ell$, the velocity of the $l-$th soliton, then
\begin{equation}
\label{asym-III-2}
u(x,t)=2\zeta_\ell  \eps_{\pm, \ell}\sech(-2\zeta_\ell(x-4\zeta^2_\ell t)+\omega )+\mathcal{O}(t^{-1})
\end{equation}
where 
$$\omega=\log \left(  \dfrac{ \left\vert c_\ell \right\vert  }{2\zeta_\ell} \right)+2\sum_{z_k\in S_\ell }\log \left\vert \dfrac  { z_k-z_\ell} {z_\ell- \overline {z_k } }\right\vert +2\sum_{z_j\in S_\ell }\log \left\vert \dfrac { z_\ell-z_j} {z_\ell- \overline {z_j }} \right\vert + 2\sum_{z_j \in S_\ell }\log \left\vert \dfrac { z_\ell+ \overline{z_j }}  {z_\ell+ {z_j }}\right\vert . $$ 
Here $S_\ell$ is defined by
$$S_\ell=\lbrace z_j = \xi_j+i\eta_j: ~  4\eta^2_j-12\xi^2_j  >4\zeta^2_\ell \rbrace \cup \lbrace z_k = i\zeta_k: ~  4\zeta_k^2 >4\zeta^2_\ell\rbrace .$$

\section{Global approximations of solutions}

In this section, as in our earlier work, \cite{CL19}, we apply a global
approximation arguments to extend our long-time asymptotics to the
focusing mKdV with rougher data. Again two important spaces are $H^{1}\left(\mathbb{R}\right)$
and $H^{\frac{1}{4}}\left(\mathbb{R}\right)$. In $H^{1}\left(\mathbb{R}\right)$,
the mKdV enjoys the natural conservation. For $H^{\frac{1}{4}}\left(\mathbb{R}\right)$,
this space is the lowest regularity that the solution can be constructed
by iterations, see Theorem $\ref{thm:KVPlocal}$. We will show that
the long-time asymptotics and soliton resolution remain valid in spaces
$H^{s}\left(\mathbb{R}\right)$ for $s\geq\frac{1}{4}$ after we pay
the price of weights. Unlike in \cite{CL19}, here we deal with $H^{s}\left(\mathbb{R}\right)$
with $s\geq\frac{1}{4}$ in a uniform manner using the local well-posedness
due to Kenig-Ponce-Vega \cite{KPV}, see Theorem $\ref{thm:KVPlocal}$,
and the new construction of conservation laws in $H^{s}$ with $s>-\frac{1}{2}$
due to the recent  work of Killip-Visan-Zhang \cite{KiViZh}
and Koch-Tataru \cite{KoTa}.

In contrast to the defocusing problem in our earlier paper, \cite{CL19},
in the focusing problem, the Beals-Coifman solutions are more complicated
and need to take care of solitons and breathers. To implement our
approximation arguments, not only we need to ensure the radiation
terms converge well but also have to make sure the discrete scattering
data including the number and locations of singularities have meaningful
limits. More refined analysis of Jost functions is necessary. See
Appendix \ref{discrete}.

\subsection{Strong solutions }

For the sake of completeness, as in \cite{CL19}, we first sketch the
uniqueness of the strong solution and the local existence of the strong
solution for the focusing mKdV in $H^{s}$ for $s\ge\frac{1}{4}$.
The following discussion will be similar to our earlier work \cite{CL19}
up to the change of the sign of the nonlinearity. We mainly follow
Kenig-Ponce-Vega \cite{KPV} and Linares-Ponce \cite{LP}. 

First of all, we define the solution operator to the linear Airy function
as
\[
W\left(t\right)u_{0}=e^{-t\partial_{xxx}}u_{0}.
\]
In other words, using the Fourier transform, one has
\[
\mathcal{F}_{x}\left[W\left(t\right)u_{0}\right]\left(\xi\right)=e^{it\xi^{3}}\hat{u}_{0}\left(\xi\right).
\]
The\emph{ strong solution} is defined as the following integral sense:
\begin{definition}
	We say the the function $u\left(t,x\right)$ is a \emph{strong solution}
	in $H^{s}\left(\mathbb{R}\right)$ to the focusing mKdV
	\begin{equation}
	\partial_{t}u+\partial_{xxx}u+6u^{2}\partial_{x}u=0,\,u\left(0\right)=u_{0}\in H^{s}\left(\mathbb{R}\right)\label{eq:IVP}
	\end{equation}
	if and only if $u\in C\left(I,H^{s}\left(\mathbb{R}\right)\right)$
	satisfies
	\begin{equation}
	u=W\left(t\right)u_{0}-\int_{0}^{t}W\left(t-s\right)\left(6u^{2}\partial_{x}u\left(s\right)\right)\,ds.\label{eq:mild}
	\end{equation}
	We also define
	\[
	\mathcal{D}_{x}^{s}h\left(x\right)=\mathcal{F}^{-1}\left[\left|\xi\right|^{s}\hat{h}\left(\xi\right)\right]\left(x\right).
	\]
\end{definition}

\begin{theorem}[Kenig-Ponce-Vega]
	\label{thm:KVPlocal}Let $s\ge\frac{1}{4}$. Then
	for any $u_{0}\in H^{s}\left(\mathbb{R}\right)$ there is $T=T\left(\left\Vert \mathcal{D}_{x}^{\frac{1}{4}}u_{0}\right\Vert _{L^{2}}\right)\sim\left\Vert \mathcal{D}_{x}^{\frac{1}{4}}u_{0}\right\Vert _{L^{2}}^{-4}$
	such that there exists a unique strong solution $u\left(t\right)$
	to the initial-value problem
	\[
	\partial_{t}u+\partial_{xxx}u+6u^{2}\partial_{x}u=0,\,u\left(0\right)=u_{0}
	\]
	satisfying
	\begin{equation}
	u\in C\left(\left[-T,T\right]:H^{s}\left(\mathbb{R}\right)\right)\label{eq:E1}
	\end{equation}
	\begin{equation}
	\left\Vert \mathcal{D}_{x}^{s}\partial_{x}u\right\Vert _{L_{x}^{\infty}\left(\mathbb{R}:L_{t}^{2}\left[-T,T\right]\right)}<\infty,\label{eq:E2}
	\end{equation}
	\begin{equation}
	\left\Vert \mathcal{D}_{x}^{s-\frac{1}{4}}\partial_{x}u\right\Vert _{L_{x}^{20}\left(\mathbb{R}:L_{t}^{\frac{5}{2}}\left[-T,T\right]\right)}<\infty,\label{eq:E3}
	\end{equation}
	\begin{equation}
	\left\Vert \mathcal{D}_{x}^{s}u\right\Vert _{L_{x}^{5}\left(\mathbb{R}:L_{t}^{10}\left[-T,T\right]\right)}<\infty,\label{eq:E4}
	\end{equation}
	and
	\begin{equation}
	\left\Vert u\right\Vert _{L_{x}^{4}\left(\mathbb{R}:L_{t}^{\infty}\left[-T,T\right]\right)}<\infty.\label{eq:E5}
	\end{equation}
	Moreover, there exists a neighborhood $\mathcal{N}$ of $u_{0}$ in
	$H^{s}\left(\mathbb{R}\right)$ such that the solution map: $\tilde{u}_{0}\in\mathcal{N}\longmapsto\tilde{u}$
	is smooth with respect to the norms given by \eqref{eq:E1}, \eqref{eq:E2},
	\eqref{eq:E3}, \eqref{eq:E4} and \eqref{eq:E5}.
\end{theorem}

\begin{proof}
	Given $T$ and $\mathcal{C}$, define the space
	\begin{equation}
	\mathcal{X}_{T}^{s}=\left\{ v\in C\left(\left[-T,T\right]:H^{s}\left(\mathbb{R}\right)\right):\left\Vert \left|v\right|\right\Vert _{\mathcal{X}_{T}^{s}}<\infty\right\} \label{eq:XT}
	\end{equation}
	and
	\begin{equation}
	\mathcal{X}_{T,\mathcal{C}}^{s}=\left\{ v\in C\left(\left[-T,T\right]:H^{s}\left(\mathbb{R}\right)\right):\left\Vert \left|v\right|\right\Vert _{\mathcal{X}_{T}^{s}}\leq\mathcal{C}\right\} \label{eq:XTC}
	\end{equation}
	where
	\begin{align*}
	\left\Vert \left|v\right|\right\Vert _{\mathcal{X}_{T}^{s}} & =\left\Vert \mathcal{D}_{x}^{s}v\right\Vert _{L_{t}^{\infty}\left(\left[-T,T\right]:H^{s}\left(\mathbb{R}\right)\right)}+\left\Vert v\right\Vert _{L_{x}^{4}\left(\mathbb{R}:L_{t}^{\infty}\left[-T,T\right]\right)}\\
	& +\left\Vert \mathcal{D}_{x}^{s}v\right\Vert _{L_{x}^{5}\left(\mathbb{R}:L_{t}^{10}\left[-T,T\right]\right)}+\left\Vert \mathcal{D}_{x}^{s-\frac{1}{4}}\partial_{x}v\right\Vert _{L_{x}^{20}\left(\mathbb{R}:L_{t}^{\frac{5}{2}}\left[-T,T\right]\right)}\\
	& +\left\Vert \mathcal{D}_{x}^{s}\partial_{x}v\right\Vert _{L_{x}^{\infty}\left(\mathbb{R}:L_{t}^{2}\left[-T,T\right]\right)}.
	\end{align*}
	To obtain a strong solution to the initial-value problem we need find
	appropriate $T$ and $\mathcal{C}$ such that the operator
	\[
	\mathcal{S}\left(v,u_{0}\right)=\mathcal{S}\left(v\right)=W\left(t\right)u_{0}-\int_{0}^{t}W\left(t-s\right)\left(6v^{2}\partial_{x}v\left(s\right)\right)\,ds
	\]
	is a contraction map on $\mathcal{X}_{T,\mathcal{C}}^{s}$.
	
	Using linear estimates for $W\left(t\right)$ and the Leibniz rule
	for fractional derivatives one can show that
	\[
	\left\Vert \left|\mathcal{S}\left(v\right)\right|\right\Vert _{\mathcal{X}_{T}^{s}}\leq c\left\Vert u_{0}\right\Vert _{H^{s}}+cT^{\frac{1}{2}}\left\Vert \left|v\right|\right\Vert _{\mathcal{X}_{T}^{s}}^{3}
	\]
	where $c$ is from linear estimates etc independent of the initial
	data. See Kenig-Ponce-Vega \cite{KPV} and Linares-Ponce \cite{LP}
	for details. Then choose $\mathcal{C}=2c\left\Vert u_{0}\right\Vert _{H^{s}}$
	and $T$ such that $c\mathcal{C}^{2}T^{\frac{1}{2}}<\frac{1}{4}$,
	we obtain that
	\[
	\mathcal{S}\left(\cdot,u_{0}\right):\,\mathcal{X}_{T,\mathcal{C}}^{s}\rightarrow\mathcal{X}_{T,\mathcal{C}}^{s}.
	\]
	Similarly, one can also show
	\begin{align*}
	\left\Vert \left|\mathcal{S}\left(v_{1}\right)-\mathcal{S}\left(v_{2}\right)\right|\right\Vert _{\mathcal{X}_{T}^{s}} & \leq cT^{\frac{1}{2}}\left(\left\Vert \left|v_{1}\right|\right\Vert _{\mathcal{X}_{T}}^{2}+\left\Vert \left|v_{2}\right|\right\Vert _{\mathcal{X}_{T}}^{2}\right)\left\Vert \left|v_{1}-v_{2}\right|\right\Vert _{\mathcal{X}_{T}^{s}}\\
	& \leq2cT^{\frac{1}{2}}\mathcal{C}^{2}\left\Vert \left|v_{1}-v_{2}\right|\right\Vert _{\mathcal{X}_{T}^{s}}.
	\end{align*}
	Therefore, with our choice of $T$ and $\mathcal{C}$, $\mathcal{S}\left(\cdot,u_{0}\right)$
	is a contraction on $\mathcal{X}_{T,\mathcal{C}}^{s}$. So there is
	a unique fixed point of this $\mathcal{S}\left(\cdot,u_{0}\right)$
	in $\mathcal{X}_{T,\mathcal{C}}^{s}$ so we obtain the unique strong
	solution:
	\[
	u=\mathcal{S}\left(u\right)=W\left(t\right)u_{0}-\int_{0}^{t}W\left(t-s\right)\left(6u^{2}\partial_{x}u\left(s\right)\right)\,ds.
	\]
	To check the dependence on the initial data, using similar arguments
	as above, one can show that
	\begin{align*}
	\left\Vert \left|\mathcal{S}\left(u_{1},u_{1}\left(0\right)\right)-\mathcal{S}\left(u_{2},u_{2}\left(0\right)\right)\right|\right\Vert _{\mathcal{X}_{T_{1}}^{s}} & \leq c\left\Vert u_{1}\left(0\right)-u_{2}\left(0\right)\right\Vert _{H^{s}}\\
	& +cT_{1}^{\frac{1}{2}}\left(\left\Vert \left|u_{1}\right|\right\Vert _{\mathcal{X}_{T_{1}}^{s}}^{2}+\left\Vert \left|u_{2}\right|\right\Vert _{\mathcal{X}_{T_{1}}^{s}}^{2}\right)\left\Vert \left|u_{1}-u_{2}\right|\right\Vert _{\mathcal{X}_{T_{1}}^{s}}.
	\end{align*}
	This can be used to show that for $T_{1}\in\left(0,T\right)$, the
	solution map from a neighborhood $\mathcal{N}$ of $u_{0}$ depending
	on $T_{1}$ to $\mathcal{X}_{T_{1},\mathcal{C}}^{s}$ is Lipschitz.
	Further work can be used to show actually the solution map is smooth.
	
	Again for more details, see Kenig-Ponce-Vega \cite{KPV} and Linares-Ponce
	\cite{LP}.
\end{proof}
Finally, we notice that if $u_{0}$ is smooth, say, Schwartz, then
the solution $u$ to the initial-value problem is also smooth and
hence is a classical solution. The uniqueness of the classical solution
is well-known, see for example Bona-Smith \cite{BS} and Saut \cite{S79}.

Next, we recall the consequences of low regularity conservation laws
from Killip-Visan-Zhang \cite{KiViZh} and Koch-Tataru \cite{KoTa}.
Here we formulate the Corollary 1.2 from Koch-Tataru \cite{KoTa}.
\begin{theorem}[Koch-Tataru]
	\label{thm:KocTar}Let $s>-\frac{1}{2}$, let $R>0$,
	and let $u_{0}$ be an initial datum for the focusing mKdV
	\begin{equation}
	\partial_{t}u+\partial_{xxx}u+6u^{2}\partial_{x}u=0,\label{eq:IVP-1}
	\end{equation}
	so that
	\[
	\left\Vert u_{0}\right\Vert _{H^{s}}\leq R.
	\]
	Then the corresponding solution $u$ satisfies that global bound
	\[
	\left\Vert u\left(t\right)\right\Vert _{H^{s}}\lesssim F\left(R,s\right)=\begin{cases}
	R+R^{1+2s} & s\geq0\\
	R+R^{\frac{1+4s}{1+2s}} & s<0
	\end{cases}.
	\]
\end{theorem}

We will use the above results for $s\geq\frac{1}{4}$.

\subsection{Approximation}

As before given $u_{0}\in H^{2,1}\left(\mathbb{R}\right)$, one can
solve the focusing mKdV using the inverse scattering transform. By
the Beals-Coifman representation, one can write the solution as
\begin{align*}
u(x,t) & =\lim_{z\rightarrow\infty}2 zm_{12}(x,t,z)\\
& =\left(\int_{\Sigma}\mu(w_{\theta}^{-}+w_{\theta}^{+})\right)_{12}\\
& =\dfrac{1}{\pi}\int_{\mathbb{R}}\mu_{11}(x,t;z)\overline{r}(z)e^{-2i\theta}dz+\sum_{k=1}^{N_{1}}\mu_{11}(\overline{z_{k}})\overline{c_{k}}e^{-2i\theta(\overline{z_{k}})}\\
& -\sum_{j=1}^{N_{2}}\mu_{11}(\overline{z_{j}})\overline{c_{j}}e^{-2i\theta(\overline{z_{j}})}+\sum_{j=1}^{N_{2}}\mu_{11}(-z_{j})c_{j}e^{-2i\theta(-z_{j})}.
\end{align*}
But as we discussed above using PDE techniques, one can construct
solutions with rougher data at least locally. Motivated by Deift-Zhou
\cite{DZ03}, as in our earlier paper \cite{CL19}, we try to understand
the relations between Beals-Coifman solutions and strong solutions.
Again first of all, if $u_{0}$ is Schwartz, one can also show $u$
is Schwartz, see for example Deift-Zhou \cite{DZ93}. So in this case,
the strong solution is surely the same as the Beals-Coifman solution. 

The leading logic of our approximation argument is that Beals-Coifman
solutions gives asymptotic formulae and the strong solutions can be
used to pass to the limit. Using the bijectivity between the initial
data and scattering data due to Zhou \cite{Zhou98}, one can show in low regularity
spaces with weights, one can always find a limit of a sequence of
smooth sequence of Beals-Coifman solutions which has asymptotic formula.
But due to low regularity, one can not make sense of this limiting
function using the inverse scattering mechanism. On the other hand,
due to the uniqueness of smooth solutions, these smooth Beals-Coifman
solutions are also strong solutions. Then one can use strong solutions
to pass to the limit which again is a strong solution to the mKdV.
Combining these together, we conclude that the asymptotic formula
remains valid for low regularity solutions. In the reaming part of
our paper, we make the philosophy rigorous. Compared with our earlier
work \cite{CL19} for the defocusing problem, we have discrete scattering
data associated to solitons and breathers in our current setting.
\begin{theorem}
	\label{thm:Hs}For $u_{0}\in H^{s,1}\left(\mathbb{R}\right)$ with
	$s\geq\frac{1}{4}$, the strong solution given by the Duhamel formulation
	\eqref{eq:mild} with initial data $u_{0}$ has the same asymptotics
	as in our main Theorem  \ref{main1} (up to null sets).
\end{theorem}

\begin{proof}
	Suppose $u_{0}\in H^{s,1}\left(\mathbb{R}\right)$ with $s\geq\frac{1}{4}$,
	we can find a sequence $\left\{ u_{0,q}\right\} $ of Schwartz functions
	such that it is a Cauchy sequence in $H^{s,1}\left(\mathbb{R}\right)$
	and $u_{0,q}\rightarrow u_{0}$ in $H^{s,1}\left(\mathbb{R}\right)$
	and
	\begin{equation}
	\sup_{q}\left\Vert u_{0,q}-u_{0}\right\Vert _{H^{s,1}}\leq\epsilon\ll1.\label{eq:smalcon1}
	\end{equation}
	Moreover, we assume that for all $q$, there is a uniform bound
	\[
	\left\Vert u_{0,q}\right\Vert _{\dot{H}^{s}\left(\mathbb{R}\right)}\lesssim\left\Vert u_{0,q}\right\Vert _{H^{s}\left(\mathbb{R}\right)}\lesssim\left\Vert u_{0,q}\right\Vert _{H^{s,1}\left(\mathbb{R}\right)}\leq C.
	\]
	Then applying Theorem \ref{thm:KVPlocal} the we can find a strong
	solution $u_{q}$ with initial data $u_{0,q}$ in $\mathcal{X}_{T,\mathcal{C}}^{s}$
	where $T$ and $\mathcal{C}$ are chose are in Theorem \ref{thm:KVPlocal}. 
	
	By Theorem \ref{thm:KVPlocal}, we also have
	\[
	\left\Vert \left|u_{q}-u_{\ell}\right|\right\Vert _{\mathcal{X}_{T,\mathcal{C}}^{s}}\lesssim\left\Vert u_{0,q}-u_{0,\ell}\right\Vert _{H^{s}\left(\mathbb{R}\right)}.
	\]
	So in $\mathcal{X}_{T,C}^{s}$, $u_{q}$ converges to a limit $u_{\infty}$
	which is a strong solution. Using the notation above, we have
	\[
	u_{\infty}=\mathcal{S}\left(u_{\infty},u_{0}\right)\in\mathcal{X}_{T,C}^{s}.
	\]
	From the inverse scattering transform, we can also have the Beals-Coifman
	solutions
	\begin{align*}
	\tilde{u}_{q} & =\dfrac{1}{\pi}\int_{\mathbb{R}}\mu_{q,11}(x,t;z)\overline{r}_{q}(z)e^{-2i\theta}dz+\sum_{k=1}^{N_{1}}\mu_{q,11}(\overline{z_{q,k}})\overline{c_{q,k}}e^{-2i\theta(\overline{z_{q,k}})}\\
	& -\sum_{j=1}^{N_{2}}\mu_{q,11}(\overline{z_{q,j}})\overline{c_{q,j}}e^{-2i\theta(\overline{z_{q,j}})}+\sum_{j=1}^{N_{2}}\mu_{q,11}(-z_{q,j})c_{q,j}e^{-2i\theta(-z_{q,j})}
	\end{align*}
	with initial data $u_{0,q}$. Note that due to Proposition \ref{prop2} and
	the smallness condition \eqref{eq:smalcon1}, the numbers of zeros of
	$\breve{a}_{q}\left(z\right)$ and $\breve{a}\left(z\right)$ are
	the same, in particular, the numbers of the imaginary zeros are $N_{1}$
	and the numbers of zeros of the imaginary axis are $N_{2}$ (up to
	symmetry reduction).
	
	Since $u_{0,q}$ is Schwartz, so $u_{q}$ and $\tilde{u}_{q}$ are
	also Schwartz. Therefore we have $u_{q}=\tilde{u}_{q}.$ 
	
	Using the bijectivity of the direct transformation, see Zhou \cite{Zhou98}, in terms
	of reflection coefficients,
	\[
	r_{q}=\mathcal{R}\left(u_{0,q}\right)\in H^{1,s},
	\]
	we have
	\[
	\left\Vert r_{q}-r_{\ell}\right\Vert _{H^{1,s}\left(\mathbb{R}\right)}\lesssim\left\Vert u_{0,q}-u_{0,\ell}\right\Vert _{H^{s,1}\left(\mathbb{R}\right)}.
	\]
	By the Lipshicitiz continuity of discrete scattering data, it follows
	that
	\[
	\left|c_{q,k}-c_{\ell,k}\right|+\left|z_{q,k}-z_{\ell,k}\right|\lesssim\left\Vert u_{0,q}-u_{0,\ell}\right\Vert _{H^{s,1}\left(\mathbb{R}\right)},\ 1\leq k\leq N_{1},
	\]
	and
	\[
	\left|c_{q,j}-c_{\ell,j}\right|+\left|z_{q,j}-z_{\ell,j}\right|\lesssim\left\Vert u_{0,q}-u_{0,\ell}\right\Vert _{H^{s,1}\left(\mathbb{R}\right)},\ 1\leq j\leq N_{2}.
	\]
	Combining with the resolvent estimates, one also has
	\[
	\left\Vert \tilde{u}_{\ell}-\tilde{u}_{q}\right\Vert _{L^{\infty}\left(\mathbb{R}\right)}\lesssim\left\Vert u_{0,q}-u_{0,\ell}\right\Vert _{H^{s,1}\left(\mathbb{R}\right)}.
	\]
	Hence as $r_{q}$ converges to a function $r_{\infty}$ in $H^{1}\left(\mathbb{R}\right)$
	and with the convergence of discrete scattering data $\left\{ c_{q,k},z_{q,k},c_{q,j},z_{q,j}\right\} $
	to $\left\{ c_{\infty,k},z_{\infty,k},c_{\infty,j},z_{\infty,j}\right\} $,
	the corresponding Beals-Coifman solution converges to a limit
	\[
	\tilde{u}_{\infty}=\lim_{q\rightarrow\infty}\tilde{u}_{q}
	\]
	in the sense of the $L^{\infty}$ norm. Indeed, we can write
	\begin{align*}
	\tilde{u}_{q} & =\dfrac{1}{\pi}\int_{\mathbb{R}}\left(\mu_{q,11}(x,t;z)-I\right)\overline{r}_{q}(z)e^{-2i\theta}dz\\
	& +\dfrac{1}{\pi}\int_{\mathbb{R}}\overline{r}_{q}(z)e^{-2i\theta}dz\\
	& +\sum_{k=1}^{N_{1}}\mu_{q,11}(\overline{z_{q,k}})\overline{c_{q,k}}e^{-2i\theta(\overline{z_{q,k}})}\\
	& -\sum_{j=1}^{N_{2}}\mu_{q,11}(\overline{z_{q,j}})\overline{c_{q,j}}e^{-2i\theta(\overline{z_{q,j}})}+\sum_{j=1}^{N_{2}}\mu_{q,11}(-z_{q,j})c_{q,j}e^{-2i\theta(-z_{q,j})}\\
	& =\text{I}_{q}+\text{II}_{q}+\text{III}_{q}.
	\end{align*}
	where
	\[
	\text{I}_{q}=\dfrac{1}{\pi}\int_{\mathbb{R}}\left(\mu_{q,11}(x,t;z)-I\right)\overline{r}_{q}(z)e^{-2i\theta}dz
	\]
	\[
	\text{II}_{q}=\dfrac{1}{\pi}\int_{\mathbb{R}}\overline{r}_{q}(z)e^{-2i\theta}dz
	\]
	and
	\begin{align*}
	\text{III}_{q} & =\sum_{k=1}^{N_{1}}\mu_{q,11}(\overline{z_{q,k}})\overline{c_{q,k}}e^{-2i\theta(\overline{z_{q,k}})}\\
	& -\sum_{j=1}^{N_{2}}\mu_{q,11}(\overline{z_{q,j}})\overline{c_{q,j}}e^{-2i\theta(\overline{z_{q,j}})}+\sum_{j=1}^{N_{2}}\mu_{q,11}(-z_{q,j})c_{q,j}e^{-2i\theta(-z_{q,j})}.
	\end{align*}
	Then due to the resolvent estimate, $\left(\mu_{q}-I\right)$ has
	the $L^{2}$ estimate and the $L^{2}$ estimate for $\overline{r}_{q}(z)e^{-2i\theta}$
	is straightforward, so $\text{I}_{q}$ can be made sense pointwise.
	For $\text{II}_{q}$, one simply notices that $\int\dfrac{1}{\pi}\int_{\mathbb{R}}\overline{r}_{q}(z)e^{-2i\theta}dz$
	is proportional to $W\left(t\right)\check{r}_{q}=e^{-t\partial_{xxx}}\check{r}_{q}$,
	by the standard stationary phase analysis, for $r_{q}\in H^{1}$,
	$\text{II}_{q}$ is a function in $L^{\infty}\left(\mathbb{R}\right)$
	with the standard pointwise decay estimates for the Airy equation.
	Finally, since $\mu_{q}\left(x,z\right)\in H^{1}$, by Sobolev's embedding,
	we can evaluate $\text{III}_{q}$ pointwise.
	
	Hence as
	\[
	\left|c_{q,k}-c_{\infty,k}\right|+\left|z_{q,k}-z_{\infty,k}\right|\rightarrow0,\ 1\leq k\leq N_{1},
	\]
	\[
	\left|c_{q,j}-c_{\infty,j}\right|+\left|z_{q,j}-z_{\infty,j}\right|\rightarrow0,\ 1\leq j\leq N_{2}
	\]
	and
	\[
	\left\Vert r_{q}-r_{\infty}\right\Vert _{H^{1}}\rightarrow0
	\]
	one has
	\[
	\left\Vert \tilde{u}_{q}-\tilde{u}_{\infty}\right\Vert _{L^{\infty}}\rightarrow0\,\ \text{as}\,\ k\rightarrow\infty.
	\]
	It follows that as $q\rightarrow\infty$, we have
	\[
	\left\Vert \left|u_{q}-u_{\infty}\right|\right\Vert _{\mathcal{X}_{T,C}^{s}}=\left\Vert \left|\tilde{u}_{q}-u_{\infty}\right|\right\Vert _{\mathcal{X}_{T,C}^{s}}\rightarrow0.
	\]
	In particular, as $q\rightarrow\infty$, one has
	\[
	\sup_{t\in\left[-T,T\right]}\left\Vert u_{q}-u_{\infty}\right\Vert _{H^{s}\left(\mathbb{R}\right)}=\sup_{t\in\left[-T,T\right]}\left\Vert \tilde{u}_{q}-u_{\infty}\right\Vert _{H^{s}\left(\mathbb{R}\right)}\rightarrow0.
	\]
	By construction, as $q\rightarrow\infty,$
	\[
	\sup_{t\in\left[-T,T\right]}\left\Vert \tilde{u}_{q}-\tilde{u}_{\infty}\right\Vert _{L^{\infty}\left(\mathbb{R}\right)}\rightarrow0.
	\]
	Hence for given $t\in\left[-T,T\right]$, one has
	\[
	u_{\infty}\left(t\right)=\tilde{u}_{\infty}\left(t\right)
	\]
	up to a measure zero set. By construction, we moreover have
	\[
	u\left(t\right)=u_{\infty}\left(t\right)=\tilde{u}_{\infty}\left(t\right)
	\]
	up to null sets.
	
	Next, applying the consequence of the conservation law for $H^{s}$,
	Theorem \ref{thm:KocTar}, we can repeat the above construct infinity
	many times to extend the interval $\left[-T,T\right]$ to $\mathbb{R}$
	and conclude that for $t\in\mathbb{R}_{+},$
	\[
	u\left(t\right)=u_{\infty}\left(t\right)=\tilde{u}_{\infty}\left(t\right).
	\]
	For fixed $t$, by our resolution formula, one can write the solution
	as a superposition of breathers, solitons, radiation
	\[
	\tilde{u}_{q}\left(t\right)=\sum_{j=1}^{N_{2}}u_{q,j}^{\left(br\right)}\left(x,t\right)+\sum_{k=1}^{N_{1}}u_{q,k}^{\left(so\right)}\left(x,t\right)+R_{q}\left(x,t\right).
	\]
	Moreover, for the radiation, we can write
	\[
	R_{q}\left(x,t\right)=L_{q}\left(x,t\right)+E_{q}\left(x,t\right)
	\]
	where $L_{q}\left(x,t\right)$ gives the leading order behavior and
	$E_{q}\left(x,t\right)$ collects the error term, see Theorem \ref{main1}. By
	the convergence of scattering data, we know
	\[
	\sum_{j=1}^{N_{2}}u_{q,j}^{\left(br\right)}\left(x,t\right)+\sum_{k=1}^{N_{1}}u_{q,k}^{\left(so\right)}\left(x,t\right)+L_{q}\left(x,t\right)\rightarrow\sum_{j=1}^{N_{2}}u_{\infty,j}^{\left(br\right)}\left(x,t\right)+\sum_{k=1}^{N_{1}}u_{\infty,k}^{\left(so\right)}\left(x,t\right)+L_{\infty}\left(x,t\right)
	\]
	pointwise. Also note that by our computations, the error terms estimates
	only depend on the $H^{1}$ norm of $r_{q}$ which by bijectivity
	only depend on the $\left\Vert u_{0,q}\right\Vert _{H^{0,1}\left(\mathbb{R}\right)}\leq C$
	uniformly.
	
	Therefore for an arbitrary fixed $t$, as the pointwise limit of $\tilde{u}_{q}\left(t\right)$,
	one can write
	\[
	\tilde{u}_{\infty}\left(t\right)=\sum_{\ell=1}^{N_{2}}u_{\infty,\ell}^{\left(br\right)}\left(x,t\right)+\sum_{\ell=1}^{N_{1}}u_{\infty,\ell}^{\left(so\right)}\left(x,t\right)+L_{\infty}\left(x,t\right)+E_{\infty}\left(x,t\right)
	\]
	where the decay estimates for $E_{\infty}\left(x,t\right)$ is the
	same as $E_{q}\left(x,t\right)$ due to the uniform error estimates.
	
	Hence up the null sets, one can write
	\[
	u\left(t\right)=\sum_{\ell=1}^{N_{2}}u_{\infty,\ell}^{\left(br\right)}\left(x,t\right)+\sum_{\ell=1}^{N_{1}}u_{\infty,\ell}^{\left(so\right)}\left(x,t\right)+L_{\infty}\left(x,t\right)+E_{\infty}\left(x,t\right)
	\]
	which has the same form as Theorem \ref{main1}.
\end{proof}
\begin{remark}
	Finally, we should point out that one essential step in our approximation
	argument is that the convergence of initial data gives the convergence
	of scattering data due to Zhou's bijectivity results and the leading
	order terms of solutions can be computed precisely using these scattering
	data. All of these computations are independent of $t$. Finally,
	since the error terms have estimates uniformly depending on the weighted
	norms of initial data, one can conclude the asymptotics of the limit.
	All of these use the machinery of the inverse scattering. From the
	view of PDEs, one can also compute the leading order terms using some
	ODE arguments see Hayashi-Naumkin \cite{HN1,HN2} and Germain-Pusateri-Rousset \cite{GPR}. But
	to find these leading order terms, it introduces extra error terms.
	If one try to use approximation argument in this setting, it is not
	clear these extra error terms give decay fast enough.
\end{remark}

\appendix
\section{Continuity of the discrete scattering data}
\label{discrete}

In this Appendix, we recall some results concerning the continuity of Jost functions
with respect to the potentials. The generic condition for initial data will also be briefly discussed.

Recall that we have
\[
a\left(z\right)=1-\int_{\mathbb{R}}iu\left(y\right)m_{21}^{+}\left(y,z\right)\,dy,
\]
\[
\check{a}\left(z\right)=1-\int_{\mathbb{R}}iu\left(y\right)m_{12}^{+}\left(y,z\right)\,dy,
\]
\[
\check{b}\left(z\right)=1-\int_{\mathbb{R}}iu\left(y\right)m_{11}^{+}\left(y,z\right)\,dy,
\]
\[
b\left(z\right)=1-\int_{\mathbb{R}}iu\left(y\right)m_{21}^{+}\left(y,z\right)\,dy.
\]
Note that $a\left(z\right)$ and $\check{a}\left(z\right)$ are independent
of $t$. So actually, we can replace the $u$ in the above formulae
by $u_{0}$ and $m_{21}^{+},\,m_{12}^{+}$ by the corresponding solutions
constructed with respect to $u_{0}$. 

Denote $a\left(z,u_{1}\right)$ and $a\left(z,u_{2}\right)$ as the
$a$ component of the scattering matrix constructed using $u_{1}$
and $u_{2}$ respectively. Using the continuous dependence on the
potential, we obtain that 
\[
\left|a\left(z,u_{1}\right)-a\left(z,u_{2}\right)\right|\lesssim\left\Vert u_{1}-u_{2}\right\Vert _{L^{2,1}}.
\]
Therefore, we notice that if we have a sequence $u_{n}\in L^{2,1}$
converges to a function $u_{0}\in L^{2,1}$, then the zeros of $a\left(z,u_{j}\right)$
will converge pointwise to $z_{\infty,\ell}$. (here we have two options,
one is since $u_{n}$ has a uniform upper bound, then $z_{j,\ell}$
is bounded sequence in the complex plane. Then there is a subsequence
converges to a limit. Secondly, we can use the continuity to get the
convergence and moreover, the limit is unique). 

The same arguments apply to $\breve{a}\left(z\right)$, $b\left(z\right)$
and $\breve{b}\left(z\right)$ associated with the initial data. By
similar arguments, one can also get the initial norming constant for
each singularity also enjoys the similar properties. 

We record the following proposition to show there is a dense subset
of $u\in L^{2,1}\left(\mathbb{R}\right)$ such that $a\left(z;u\right)$
has at most finitely many simple zeros in $\mathbb{C}^{-}$ and no
zeros on $\mathbb{R}$. 
\begin{proposition}\label{prop1}
	Suppose $R>0$ and $u\in C_{0}^{\infty}\left(\left[-R,R\right]\right)$.
	Let $a\left(z;u\right)$ be the $\left(1,1\right)$ entry of the scattering
	matrix for $u$. For $\varphi\in C_{0}^{\infty}\left(\mathbb{R}\right)$,
	denote $a\left(z,\mu\right)$ as the $\left(1,1\right)$ entry of
	the scattering matrix for $u+\mu\varphi$. By construction, $a\left(z,0\right)=a\left(z;u\right)$.
	
	(1) Suppose $S=\left\{ z_{i}\right\} _{i=1}^{N}$ are the isolated
	zeros of $a\left(z;u\right)$ in $\mathbb{C}^{-}\bigcup\mathbb{R}$
	and for some $i$ such that $z_{i}\neq0$ is one of the zeros of $a\left(z;u\right)$
	of multiplicity $\lambda\geq2$, in other words, $a\left(z;u\right)=\left(z-z_{i}\right)^{\lambda}g\left(z\right)$
	for some analytic function $g\left(z\right)$ with $g\left(z_{i}\right)\neq0$.
	Then for some $\varphi\in C_{0}^{\infty}\left(\mathbb{R}\right)$
	and all sufficiently small $\mu\neq0$, $a\left(z,\mu\right)$ has
	$\lambda$ simple zeros in the disc $D_{r_{i}}\left(z_{i}\right)$
	with $r_{i}$ small enough.
	
	(2) Suppose that after the perturbation in (1), for small $j$, $\mathcal{Z}_{j}$
	is a simple zero of $a\left(z,\mu\right)$ on the real axis such that
	$\mathcal{Z}_{j}\neq0$. Then for some $\varphi\in C_{0}^{\infty}\left(\mathbb{R}\right)$
	and all sufficiently small $\mu'\neq0$, $a\left(z,\mu'\right)$ has
	no zeros on the real axis near $\mathcal{Z}_{j}$. 
	
	In each case, one can choose $\varphi$ to have support in $\left(-2R,-R\right)\bigcup\left(R,2R\right)$.
\end{proposition}

Since as we discussed above, $a\left(z;u\right)$ is continuous with
respect to $u$ and $a\left(z;u\right)$ is analytic in $\mathbb{C}^{-}$,
this set is also open.

Note that due to symmetry, a breather corresponds two simple zeros
off the imaginary axis. A simple zero on the imaginary axis results
a soliton. These structure is fairly stable since the zero of the
soliton is simple. It will not bifurcate into two simple zeros under
small perturbations.

More precisely, we focus on $\breve {a}\left(z;u_{0}\right)$. This
is analytic in the $\mathbb{C}^{+}$. By our assumption, $\breve{a}\left(z;u_{0}\right)$
has exactly $N_{1}$ simple zeros on the imaginary axis , $2N_{2}$
simple zeros off the imaginary axis and no zeros on $\mathbb{R}$.
Focusing on one zero on the imaginary axis, say, $z_{i}\left(u_{0}\right)$,
we consider the integral of $\frac{\breve{a}'\left(z\right)}{\breve{a}\left(z\right)}$
over a small circle $\mathcal{C}_{i}$ with radius small enough centered
at $z_{i}\left(u_{0}\right)$, then by Cauchy's argument principle,
\[
\frac{1}{2\pi}\int_{\mathcal{C}_{i}}\frac{\breve{a}'\left(z;u_{0}\right)}{\breve{a}\left(z;u_{0}\right)}\,dz=1
\]
due to the analyticity of $\breve{a}\left(z;u_{0}\right)$ and the
simplicity of $z_{i}\left(u_{0}\right)$. Under a sufficiently small
perturbation, the zeros of $\breve{a}\left(z;u\right)$ coressponding
to $z_{i}\left(u_{0}\right)$ where $\left\Vert u-u_{0}\right\Vert _{L^{2,1}}$
is sufficiently small will be located in the disc $\mathcal{D}_{i}$
surrounded by $\mathcal{C}_{i}$. By the continuity of $\breve{a}\left(z;u_{0}\right)$,
again we have
\[
\frac{1}{2\pi}\int_{\mathcal{C}_{i}}\frac{\breve{a}'\left(z;u\right)}{\breve{a}\left(z;u\right)}\,dz=1.
\]
Therefore the number of zeros of $\breve{a}\left(z;u\right)$ in $\mathcal{D}_{i}$
should be one and located on the imaginary axis. Otherwise, if the
zero is off the imaginary axis, it should come with a pair due to
the symmetry of the mKdV and will give
\[
\frac{1}{2\pi}\int_{\mathcal{C}_{i}}\frac{\breve{a}'\left(z;u\right)}{\breve{a}\left(z;u\right)}\,dz=2
\]
which is a contradiction. 

Similar analysis can be applied to those simple zeros located near
the imaginary axis. They will not degenerate to zeros on the imaginary
axis.

Therefore, under the the simplicity assumption, any sufficiently small
perturbation will not destroy the structures of breathers and solitons.
\begin{proposition}\label{prop2}
	Suppose that $u_{0}\in L^{2,1}\left(\mathbb{R}\right)$ and $a\left(z;u_{0}\right)$
	has exactly $N_{1}$ simple zeros on the imaginary axis, $N_{2}$
	simple zeros off the imaginary axis and no zeros on $\mathbb{R}$.
	There is a neighborhood $\mathcal{N}$ of $u_{0}$ in $L^{2,1}\left(\mathbb{R}\right)$
	so that all $u\in\mathcal{N}$ have these same properties.
	
	Suppose that $u_{0}$ is a generic potential with $n$ simple zeros
	of $\breve{a}\left(z,u_{0}\right)$ in $\mathbb{C}^{+}$. (Here we
	do not distinguish breathers and solitons). Let $S_{1}=\left\{ z_{i}\right\} _{i=1}^{N_{1}}$
	and $S_{2}=\left\{ z_{j}\right\} _{j=1}^{N_{2}}$ be a list of the
	zeros of $\breve{a}\left(z,u_{0}\right)$ in $\mathbb{C}^{+}$. Set
	\[
	d_{1}\left(u_{0}\right)=\min\left(\min_{1\leq j\neq k\leq N_{1}}\left|z_{j}\left(u_{0}\right)-z_{k}\left(u_{0}\right)\right|,\min_{1\leq j\neq k\leq N_{2}}\left|z_{j}\left(u_{0}\right)-z_{k}\left(u_{0}\right)\right|\right)
	\]
	\[
	d_{2}\left(u_{0}\right)=\min\left(d\left(S_{1}\left(u_{0}\right),S_{2}\left(u_{0}\right)\right),\min_{1\leq j\leq N_{1}}\left(\Im z_{j}\right),\min_{1\leq j\leq N_{1}}\left(\Im z_{j}\right)\right)
	\]
	and
	\[
	d_{S}\left(u_{0}\right)=\min\left(d_{1}\left(u_{0}\right),d_{2}\left(u_{0}\right)\right).
	\]
	There is a neighborhood $\mathcal{N}$ of $u_{0}$ in $L^{2,1}$ so
	that
	
	(1) For any $u\in\mathcal{N}$, $\breve{a}\left(z,u\right)$ has exactly
	$N_{1}+N_{2}$ zeros in $\mathbb{C}^{+}$, no zeros in $\mathbb{R}$,
	and
	\[
	\left|z_{i}\left(u\right)-z_{i}\left(u_{0}\right)\right|\leq\frac{1}{2}d_{S}\left(u_{0}\right)
	\]
	
	\[
	\left|z_{j}\left(u\right)-z_{j}\left(u_{0}\right)\right|\leq\frac{1}{2}d_{S}\left(u_{0}\right).
	\]
	
	(2) We also have
	\[
	\left|z_{i}\left(u\right)-z_{i}\left(u_{0}\right)\right|,\ z_{j}\left(u\right)-z_{j}\left(u_{0}\right)\leq C\left\Vert u-u_{0}\right\Vert _{L^{2,1}}
	\]
	hold for $C$ uniform in $u\in\mathcal{N}$.
	
	(3)
	\[
	\left|b_{i}\left(u\right)-b_{i}\left(u_{0}\right)\right|,\:\left|b_{j}\left(u\right)-b_{j}\left(u_{0}\right)\right|\leq\leq C\left\Vert u-u_{0}\right\Vert _{L^{2,1}}
	\]
	hold for $C$ uniform in $u\in\mathcal{N}$.
	
	(4)
	\[
	\left|C_{i}\left(u\right)-C_{i}\left(u_{0}\right)\right|,\:\left|C_{j}\left(u\right)-C_{j}\left(u_{0}\right)\right|\leq\leq C\left\Vert u-u_{0}\right\Vert _{L^{2,1}}
	\]
	hold for $C$ uniform in $u\in\mathcal{N}$.
\end{proposition}

\begin{proof}
	We can find the neighborhood $\mathcal{N}$ of $u_{0}$ by our general
	argument. By construction, for any $u\in\mathcal{N}$, $\breve{a}\left(z,u\right)$
	has exactly $N_{1}+N_{2}$ zeros in $\mathbb{C}^{+}$, no zeros in
	$\mathbb{R}$, and
	\[
	\left|z_{i}\left(u\right)-z_{i}\left(u_{0}\right)\right|\leq\frac{1}{2}d_{S}\left(u_{0}\right),
	\]	
	\[
	\left|z_{j}\left(u\right)-z_{j}\left(u_{0}\right)\right|\leq\frac{1}{2}d_{S}\left(u_{0}\right).
	\]
	To establish other claims, we need to use the simplicity of zeros
	and the implicit function theorem. We will study the equation $a\left(z_{j}\left(u\right);u\right)=0$
	and regard it as a function on $\mathbb{C}^{-}\times L^{2,1}$. First
	of all, we know this function is analytic in $z\in\mathbb{C}^{-}$
	and is differnetiable in $u$ from
	\[
	a\left(z;u\right)=1-\int iu\left(y\right)m_{21}^{+}\left(y,z\right)\,dy
	\]
	where $m_{21}^{+}\left(z,y\right)$ is analytic in $\mathbb{C}^{-}$
	and depends on $u$ smoothly. Since $z_{j}\left(u_{0}\right)$'s and
	$z_{i}\left(u_{0}\right)$'s are simply zeros, we also have
	\[
	a'\left(z_{j}\left(u_{0}\right);u_{0}\right)\neq0
	\]
	which is the condition for the implicit function theorem to be applied.
	
	(2) The implicit function theorem also guarantees that the function
	$z_{i}\left(q\right)$ and $z_{j}\left(q\right)$ will be $C^{1}$
	as a function of $u$, and hence surely Lipschitz continuous. See
	P\"oschel-Trubowitz \cite{PT}.
	
	(3) The estimates for $b\left(z\right)$ can be obtained similarly
	as $a\left(z\right)$.
	
	(4) Finally, for the norming constants, $a'\left(z_{i}\right)$ and
	$a'\left(z_{j}\right)$ can be expressed in terms of $a$ via the
	Cauchy integral over a small circle around $z_{i}$ and $z_{j}$ respectively
	since $a$ is analytic in $\mathbb{C}^{-}$. Since $a\left(z_{i}\left(u\right);u\right)$,
	$a\left(z_{j}\left(u\right);u\right)$, $b\left(z_{i}\left(u\right);u\right)$
	and $b\left(z_{j}\left(u\right);u\right)$ are Lipschtiz continuous
	with respect to $u$. Then we can extend the Lipschitz continuity
	to the norming constants.
\end{proof}

\end{document}